\documentclass{article}
\usepackage[utf8]{inputenc}
\usepackage[a4paper, total={15.92cm, 24.37cm}, top=1in, left=1in]{geometry}
\usepackage[T1]{fontenc}
\usepackage{lmodern}
\usepackage{dsfont} 
\usepackage{amsfonts}
\usepackage{graphicx}
\usepackage{amsmath,mathrsfs} 
\usepackage{xcolor}
\usepackage{verbatim}
\usepackage[titletoc,title]{appendix}
\usepackage{tikz}
\usetikzlibrary{patterns}
\usetikzlibrary{arrows}
\usepackage[title]{appendix}
\usepackage{amsthm,amssymb}
\usepackage{amsbsy}
\usepackage{bbold}

\makeatletter
\@addtoreset{equation}{section}
\makeatother

\makeatletter
\@addtoreset{enunciato}{section}
\makeatother

\newcounter{enunciato}[section]

\newtheorem{ittheorem}{Theorem}
 \newtheorem{itlemma}{Lemma}
 \newtheorem{itproposition}{Proposition}
 \newtheorem{itdefinition}{Definition}
 \newtheorem{itremark}{Remark}
 \newtheorem{itclaim}{Claim}
 \newtheorem{itfact}{Fact}
 \newtheorem{itexample}{Example}
 \newtheorem{itconjecture}{Conjecture}
 \newtheorem{itobservation}{Observation} 
 \newtheorem{itcorollary}{Corollary}

 \newenvironment{theorem}{\addtocounter{enunciato}{1}
 \begin{ittheorem}}{\end{ittheorem}}

 \newenvironment{lemma}{\addtocounter{enunciato}{1}
 \begin{itlemma}}{\end{itlemma}}

 \newenvironment{proposition}{\addtocounter{enunciato}{1}
 \begin{itproposition}}{\end{itproposition}}

 \newenvironment{definition}{\addtocounter{enunciato}{1}
 \begin{itdefinition}}{\end{itdefinition}}

 \newenvironment{remark}{\addtocounter{enunciato}{1}
 \begin{itremark}}{\end{itremark}}

\newtheorem*{ClmPk}{Claim $\boldsymbol{{\cal P}(k)}$}
\newtheorem*{ClmPkt}{Claim $\boldsymbol{\tilde{\cal P}(k)}$}
\newtheorem*{ClmPkt'}{Claim $\boldsymbol{\tilde{\cal P}'(k)}$}
\newtheorem*{ClmPktbar}{Claim $\boldsymbol{\bar{\cal P}(k)}$}

\def \Z {\mathbb Z}
\def \R {\mathbb R}

\def \N {\mathbb N}

\def \ee {\mathrm e}
\def \dd {\mathrm d}

\def \ba {\begin{array}}
\def \ea {\end{array}}

\def \Z {{\mathbb Z}}
\def \R {{\mathbb R}}

\def \N {{\mathbb N}}

\def \ra {\rightarrow}

\def \cX {{\cal X}}

\def \cL {{\cal L}}

\def \L {{\Lambda}}

\def \b {{\beta}}

\def \h {{\eta}}

\def \n {{\nu}}
\def \D {{\Delta}}

\def \T {{\Theta}}

\def \d {{\delta}}

\def \n {{\nu}}

\def \SES {{\hbox{\footnotesize\rm SES}}}

\newcommand{\dist}{\mathrm{dist}}

\begin{document}

	
\author{
\renewcommand{\thefootnote}{\arabic{footnote}}
S.\ Baldassarri \footnotemark[1] \textsuperscript{,}\footnotemark[2]
\\
\renewcommand{\thefootnote}{\arabic{footnote}}
A.\ Gaudilli\`ere \footnotemark[2]
\\
\renewcommand{\thefootnote}{\arabic{footnote}}
F.\ den Hollander \footnotemark[3]
\\
\renewcommand{\thefootnote}{\arabic{footnote}}
F.R.\ Nardi \footnotemark[1]
\\
\renewcommand{\thefootnote}{\arabic{footnote}}
E.\ Olivieri \footnotemark[4]
\\
\renewcommand{\thefootnote}{\arabic{footnote}}
E.\ Scoppola \footnotemark[5]
}
	
\title{Droplet dynamics in a two-dimensional rarefied gas\\ 
under Kawasaki dynamics}
	
\footnotetext[1]{
Universit\`a degli Studi di Firenze, Dipartimento di Matematica e Informatica ``Ulisse Dini'', Viale Morgagni 67/a 50134, Firenze, Italy}
	
\footnotetext[2]{
Aix-Marseille Universit\'e, CNRS, Centrale Marseille, I2M UMR CNRS 7373, 39, rue Joliot Curie, 13453 Marseille Cedex 13, France}
	
\footnotetext[3]{
Mathematisch Instituut, Universiteit Leiden, Niels Bohrweg 1, 2333 CA Leiden, The Netherlands}
	
\footnotetext[4]{
Dipartimento di Matematica, Universit\`a di Roma Tor Vergata, Via della Ricerca Scientifica, 00133 Roma, Italy}
	
\footnotetext[5]{
Dipartimento di Matematica e Fisica, Universit\`a di Roma Tre, Largo S.\ Leonardo Murialdo 1, 00146 Roma, Italy}
	
\maketitle
	
	
\begin{abstract}
This is the second in a series of three papers in which we study a lattice gas subject to Kawasaki conservative dynamics at inverse temperature~$\beta>0$ in a large finite box $\Lambda_\beta \subset \Z^2$ whose volume depends on $\beta$. Each pair of neighbouring particles has a negative \emph{binding energy} $-U<0$, while each particle has a positive \emph{activation energy} $\Delta>0$. The initial configuration is drawn from the grand-canonical ensemble restricted to the set of configurations where all the droplets are subcritical. Our goal is to describe, in the metastable regime $\Delta \in (U,2U)$ and in the limit as $\beta\to\infty$, how and when the system nucleates, i.e., grows a supercritical droplet somewhere in $\Lambda_\beta$. 
		
In the first paper we showed that subcritical droplets behave as quasi-random walks.
In the present paper we use the results in the first paper to analyse
how subcritical droplets form and dissolve on multiple space-time scales when the volume is \emph{moderately large},
namely, $|\Lambda_\beta| = \ee^{\Theta\beta}$ with $\Delta < \Theta < 2\Delta-U$.
In the third paper we consider the setting where the volume is \emph{very large}, namely,  $|\Lambda_\beta| = \ee^{\Theta\beta}$ with $\Delta < \Theta < \Gamma-(2\Delta-U)$, where $\Gamma$ is the energy of the critical droplet in the local model with fixed volume, and use the results in the first two papers to identify the nucleation time. We will see that in a very large volume critical droplets appear more or less independently in boxes of moderate volume, a phenomenon referred to as \emph{homogeneous nucleation}.  

Since Kawasaki dynamics is \emph{conservative}, i.e., particles move around and interact but are preserved, we need to control non-local effects in the way droplets are formed and dissolved. This is done via a \emph{deductive approach}: the tube of typical trajectories    
leading to nucleation is described via a series of events, whose complements have negligible probability, on which the evolution of the gas consists of \emph{droplets wandering around on multiple space-time scales} in a way that can be captured by a coarse-grained Markov chain on a space of droplets.
      		
{\it Acknowledgement.}
FdH and FRN were supported through NWO Gravitation Grant NETWORKS 024.002.003. SB was supported through LYSM and ``Gruppo Nazionale per l'Analisi Matematica, la Probabilità e le loro Applicazioni'' (GNAMPA-INdAM). 

\end{abstract}
	

\newpage	
\tableofcontents
\newpage


\section{Model and results}
\label{sec:modres}

In Section~\ref{sec:back} we provide background on metastability. In Section~\ref{sec:kaw} we introduce the Kawasaki dynamics that is the subject of the present paper. In Section~\ref{sec:mainres} we state our main theorems. In Section~\ref{sec:discout} we discuss the theorems and provide an outline of the remainder of the paper.   
	

\subsection{Background}
\label{sec:back}

Metastability for interacting particle systems is a thriving area in mathematical physics that is full of challenges. The goal is to describe the crossover from a \emph{metastable state} (in which the system starts from a quasi-equilibrium) to a \emph{stable state} (in which the system reaches equilibrium) under the influence of a \emph{stochastic dynamics}. Examples are the magnetisation of Ising spins subject to Glauber dynamics and the condensation of a lattice gas subject to Kawasaki dynamics. The former is an example of a \emph{non-conservative} dynamics (the number of up-spins is not preserved), while the latter is an example of a \emph{conservative dynamics} (the number of particles is preserved). Conservative systems are harder to deal with than non-conservative systems because the dynamics is non-local. The monographs \cite{OV04} and \cite{BdH15} contain plenty of examples of metastable systems, and include extensive references to the literature. The focus is on the \emph{average crossover time} from the metastable state to the stable state in parameter regimes that characterise metastability, on the set of configurations that form the \emph{saddle points} for the crossover -- referred to as the `critical droplet' -- and on the sequence of configurations the system sees \emph{prior to and after} the crossover -- referred to as the `tube of typical trajectories'. 

In the present paper we adopt the point of view that the identification of `tube of typical trajectories' is the key towards getting full control on the metastable crossover. Already in the early mathematical papers on metastability \cite{CGOV,OS1995, OS1996,S1993}, and later in papers on Kawasaki dynamics in finite volume \cite{GOS,dHOS00a}, the main strategy was to identify sets of configurations of \emph{increasing regularity} that are \emph{resistant} to the dynamics on corresponding increasing time scales. These sets of configurations form the backbone in the construction of the `tube of typical trajectories'. In particular, the idea was to define temporal configurational environments within which the trajectories of the process remain confined with high probability on appropriate time scales. This approach involves an analysis of all the possible evolutions of the process, and requires the exclusion of rare events via large deviation a priori estimates.

The present paper is the second in a series of three papers dealing with \emph{nucleation} in a \emph{supersaturated lattice gas} in a large volume. In particular, we consider a two-dimensional lattice gas at low density and low temperature that evolves under Kawasaki dynamics, i.e., particles hop around randomly subject to hard-core repulsion and nearest-neighbour attraction. We are interested in how the gas \emph{nucleates} in large volumes, i.e., how the particles form and dissolve subcritical droplets until they manage to build a critical droplet that is large enough to trigger the nucleation. 

In the first paper we showed that subcritical droplets behave as quasi-random walks. In the present paper we use the results in the first paper to analyse how subcritical droplets form and dissolve on multiple space-time scales when the volume is \emph{moderately large}. In large volumes the possible evolutions of the Kawasaki lattice gas are much more involved than in small volumes, and multiple events must be considered and controlled compared to the case of finite volume treated earlier. In particular, it is important to control the \emph{history} of the particles. For this reason we introduce several new tools, such as assigning colours to the particles that summarise information about how they interacted with the surrounding gas in the past. The focus remains on the `tube of typical trajectories', even though the control of all the possible evolutions of the Kawasaki lattice gas requires the use of multiple graphs describing multiple temporal configurational environments. These graphs will be identified in Section \ref{sub:lmmgen}, which is the core of the present paper and contains the proofs of all the principal lemmas. In the third paper we consider the setting where the volume is \emph{very large} and use the results in the first two papers to identify the \emph{nucleation time}. The outcome of the three papers together shows the following:
\begin{itemize}
\item[(1)] 
Most of the time the configuration consists of \emph{quasi-squares} and \emph{free particles}. That is why we use the terminology \emph{droplet dynamics}. The crossover time between configurations of this type is identified on a time scale that is exponential in $\beta$ (see Theorem \ref{thm:qsrec}).  
\item[(2)] 
Starting from configurations consisting of quasi-squares and free particles, the dynamics typically \emph{resist}, i.e., the dimensions of the quasi-squares do not change, for an exponential time scale in $\beta$ depending only on the dimensions of the smallest quasi-square (see Theorem \ref{thm:res}).  
\item[(3)] 
Starting from configurations consisting of quasi-squares and free particles, the dynamics typically either creates a larger quasi-square or a smaller quasi-square, depending on the dimensions of the starting quasi-square (see Theorem \ref{thm:cambiotipico}). There is a non-negligible probability that a \emph{subcritical} quasi-square follows an \emph{atypical} transition, in that it grows a larger quasi-square, and this lets the dynamics \emph{escape from metastability} (see Theorem \ref{thm:cambioatipico}).
\item[(4)] 
The crossover from the gas to the liquid (= nucleation) occurs because a \emph{supercritical quasi-square} is created somewhere in a moderately large box and \emph{subsequently grows into a large droplet}. This issue will be addressed in \cite{BGdHNOS}.
\item[(5)] 
The configurations in moderately large boxes behave as if they are \emph{essentially independent} and as if the surrounding gas is \emph{ideal}. No information travels between these boxes on the relevant time scale that grows exponentially fast with $\beta$. The supercritical quasi-square appears more or less independently in different boxes, a phenomenon referred to as \emph{homogeneous nucleation}. This issue will be addressed in \cite{BGdHNOS}.
\item[(6)]
The \emph{tube of typical trajectories leading to nucleation} is described via a series of events on which the evolution of the gas consists of \emph{droplets wandering around on multiple space-time scales}. This control is achieved via what we call the \emph{deductive approach} in Section \ref{sub:lmmgen}.
\item[(7)] 
The asymptotics of the nucleation time is identified on a time scale that is exponential in $\beta$ and depends on the \emph{entropic factor} related to the size of the box. This issue will be addressed in \cite{BGdHNOS}.
\end{itemize}

\begin{remark}
	{\rm Kawasaki dynamics in large volumes at low temperatures was studied earlier in \cite{BdHS10}. There, the average nucleation time was computed for a specific starting distribution called the \emph{last-exit-biased distribution} for the transition from subcritical to supercritical. The techniques employed in that paper rely on potential theory, which is tailored to deal with hitting probabilities and hitting times. It does \emph{not} provide information on \emph{how} the nucleation takes place. Since the last-exit-biased distribution is not a good description of the metastable equilibrium, the resulting average nucleation time is not necessarily physically realistic. However, by controlling the droplet dynamics with the tools of the present paper, we can show that the last-exit-biased distribution falls into the basin of attraction of the metastable equilibrium, and that therefore the average nucleation time computed in \cite{BdHS10} provides an accurate description, including prefactors.}
	\hfill$\spadesuit$
\end{remark}

\begin{remark}
	{\rm Kawasaki dynamics in large volumes at low temperature was also studied in \cite{GL15} (with the help of techniques developed in \cite{BL15}). There, the transitions between the different ground states are analysed in a regime where there is no pure-gas metastable state\footnote{The condition $n^4 L^2 e^{-\beta} \ll 1$ in \cite{GL15}, in the notation introduced in Section~\ref{sec:kaw}, reads $2 (\Theta - \Delta) + \Theta - U < 0$, which, together with $\Theta > \Delta$, implies that $\Delta < U$: particles immediately aggregate up to gas depletion.} and the process is started from a large square droplet with no surrounding gas. In that setting the interaction between the gas and the droplet, which is at the core of the present work, is largely avoided. Both \cite{BL15} and \cite{GL15} are closely related to the aforementioned wandering droplet issue, about which we will say more later on.}
	\hfill$\spadesuit$
\end{remark}

\begin{remark}
	{\rm It remains a challenge to describe what happens \emph{after} the exit from metastability, i.e., when the system has grown a large supercritical droplet that subsequently grows, moves around, absorbs smaller droplets, thereby depleting the surrounding gas, etc. The fact that Kawasaki lattice-gas dynamics is \emph{conservative} represents a major hurdle. For Glauber spin-flip dynamics, which is \emph{non-conservative}, this phase of the dynamics, which is \emph{beyond metastability}, has been \emph{completely} elucidated at low temperatures in \cite{DS97} and \emph{partially} elucidated at all subcritical temperatures in \cite{SS98}. While at low temperatures the escape from metastability and the successive growth of supercritical droplets occurs along increasingly larger Wulff shapes (up to fluctuations), these are used in \cite{SS98} only as a mathematical tool to control the average transition time via monotonicity, i.e., attractiveness. The description of the typical transition paths for the non-conservative Glauber spin-flip dynamics at all subcritical temperatures is still an open problem: only the Wulff shape of the critical configurations is known, and simulations suggest that subcritical configurations are ``rounder'' and supercritical configurations are ``straighter''. Since the techniques, used in \cite{GMV20} to control local relaxation times and show the absence of memory of the transition time for Glauber dynamics at all subcritical temperatures, do not rely on monotonicity, we might hope to be able to extend our understanding of the gas-droplet interaction and thereby extend our control of the transition time for non-monotone Kawasaki dynamics at higher temperatures. However, this will not be sufficient to describe the shape of evaporating subcritical clusters in a depleted gas where critical clusters can still grow. It is also beyond the scope of the present work, in which we analyse the gas-droplet interaction in the much simpler context of low-temperature dynamics and are able to fully characterise the typical escape paths from metastability, while the metastable state of Kawasaki dynamics in large volume does not look like a ground state of a restricted dynamics.}
	\hfill$\spadesuit$
\end{remark} 

For more background on metastability, we refer the reader to the monographs \cite{OV04} and \cite{BdH15}. Our reference list is restricted to those papers that are directly relevant to the work in the present paper.    

	
\subsection{Kawasaki dynamics}
\label{sec:kaw}


\paragraph{$\bullet$ Hamiltonian, generator and equilibrium.}	

Let $\b>0$ denote the inverse temperature. Let $\L_\b \subset \mathbb Z^2$ be the square box with volume
\begin{equation}
\label{Lvol}
|\L_\b| = \ee^{\T\b}, \qquad \T>0,
\end{equation}
centered at the origin with periodic boundary conditions. With each $x\in \L_\b$ associate an occupation variable $\h(x)$, assuming the values $0$ or $1$. A lattice gas configuration is denoted by  $\h \in \cX_\beta= \{0,1\}^{\L_\b}$. With each configuration $\eta$ associate an energy given by the Hamiltonian
\begin{equation}
\label{Ham*}
H(\h) = -U \sum_{\{x,y\}\in \L_\b^*}\h(x)\h(y),
\end{equation}
where $\L_\b^*$ denotes the set of bonds between nearest-neighbour sites in $\L_\b$, i.e., there is a \emph{binding energy} $-U<0$ between neighbouring particles. Let 
\begin{equation}
\label{numpart*}
|\eta|=\sum_{x\in\Lambda_\beta}\h(x)
\end{equation}
be the number of particles in $\Lambda_\beta$ in the configuration $\h$,  and let
\begin{equation}
\label{Npart}
\cX_N = \{\h\in\cX_\beta\colon\,|\eta|=N\}
\end{equation}
be the set of configurations with $N$ particles. We define Kawasaki dynamics as the continuous-time Markov chain $X=(X(t))_{t\geq0}$
with state space $\cX_N$  given by the generator
\begin{equation}
\label{gendef}
(\cL f)(\h)=\sum_{\{x,y\}\in\L_\b^*}c(x,y,\h)[f(\h^{x,y})-f(\h)],
\qquad \eta\in{\cal X}_{\beta},
\end{equation}
where
\begin{equation}
\label{confexch}
\h^{x,y}(z)= \left\lbrace\begin{array}{lcl}
\h(z) &{\rm if} & z \neq x,y, \\
\h(x) &{\rm if} & z=y, \\
\h(y) &{\rm if} & z=x,  \end{array}\right.
\end{equation}
and
\begin{equation}
\label{rate}
c(x,y,\h)= \ee^{-\b [H(\h^{x,y})-H(\h)]_+}.
\end{equation}
Equations~\eqref{gendef}--\eqref{rate} represent the standard \emph{Metropolis dynamics} associated with $H$, and is \emph{conservative} because it preserves the number of particles, i.e., $|X(t)|=|X(0)|$ for all $t>0$. The \emph{canonical Gibbs measure} $\n_N$ defined as
\begin{equation}
\label{nuN}
\n_N(\h) = {\ee^{-\b H(\h)}\mathbb{1}_{\cX_N}(\h) \over Z_N},
\qquad Z_N = \sum_{\h\in \cX_N} \ee^{-\b H(\h)},
\qquad \h \in \cX_\beta,
\end{equation}
is the reversible equilibrium of this stochastic dynamics for any $N$:
\begin{equation}
\label{rev}
\n_N(\h) c(x,y,\h) = \n_N(\h^{x,y}) c(x,y,\h^{x,y}).
\end{equation}
We augment the energy $H(\eta)$ of configuration $\eta$ by adding a term $\Delta|\eta|$, with $\Delta>0$ an \emph{activation energy} per particle. This models the presence of an external reservoir that keeps the density of particles in $\Lambda_\beta$ fixed at $\ee^{-\beta\Delta}$.


\paragraph{$\bullet$ Subcritical, critical and supercritical droplets.}

The initial configuration is chosen according to the \emph{grand-canonical Gibbs measure} restricted to the set of subcritical droplets. More precisely, denote by 
\begin{equation}
\label{def:lc}
\ell_ c = \Big\lceil \frac{U}{2U-\Delta}\Big\rceil
\end{equation}
the critical length introduced in \cite{dHOS00a} for the \emph{local model} where $\Lambda_\beta=\Lambda$ does not depend on $\beta$ (see Fig.~\ref{fig-cancrit}). 

\begin{figure}
\setlength{\unitlength}{0.24cm}
\begin{picture}(15,15)(-3,-1)
\qbezier[40](10,-4)(20,-4)(30,-4)
\qbezier[40](10,-4)(10,3)(10,15)
\qbezier[40](10,15)(20,15)(30,15)
\qbezier[40](30,-4)(30,10)(30,15)
\put(15,0){\line(1,0){10}}
\put(15,11){\line(1,0){10}}
\put(15,0){\line(0,1){11}} 
\put(26,12){\line(1,0){1}}
\put(27,12){\line(0,1){1}}
\put(27,13){\line(-1,0){1}}
\put(26,13){\line(0,-1){1}}
 \put(25,0){\line(0,1){5}}
\put(25,6){\line(0,1){5}}
\put(26,5){\line(0,1){1}}
\put(25,5){\line(1,0){1}}
\put(25,6){\line(1,0){1}}
\qbezier[5](25,5)(25,5.5)(25,6)
\put(13,5.5){$\ell_c$}
\put(18.5,-2){$\ell_c-1$}
\put(28.5,-3){$\Lambda$}
\end{picture}
\vspace{1cm}
\caption{\small A critical droplet in a finite volume $\Lambda$: a protocritical droplet, consisting of an $(\ell_c-1) \times \ell_c$ quasi-square with a single protuberance attached to one of the longest sides, and a free particle nearby. When the free particle attaches itself to the protuberance, the droplet becomes supecritical. \normalsize}
\label{fig-cancrit}
\end{figure}
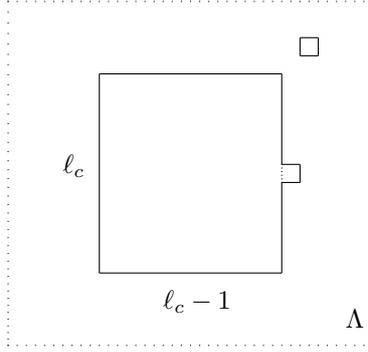 

Define
\begin{equation}
\label{def:R}
{\cal R} = \big\{\eta\in{\cal X}_\beta\colon\,\hbox{all clusters of } \eta \hbox{ have volume at most } \ell_c(\ell_c-1)+2\big\}
\end{equation}
and put
\begin{equation}
\label{muR}
\mu_{\cal R}(\h) = \frac{\ee^{-\b[H(\h)+\Delta |\eta|]}}{Z_{\cal R}}
\mathbb{1}_{\cal R}(\eta), \qquad \h\in\cX_\beta,
\end{equation}
where
\begin{equation}
\label{ZR}
Z_{\cal R} = \sum_{\h\in {\cal R}} \ee^{-\b [H(\h)+\Delta |\eta|]}
\end{equation}
is the normalising partition sum. The initial configuration $X(0)$ is drawn from $\mu_{\cal R}$. 

We will be interested in the regime
\begin{equation}
\D \in (U,2U), \quad \b \to\infty.
\end{equation}
which corresponds to \emph{metastable behaviour}.\footnote{In order to avoid trivialities, we assume that $\ell_c>2$, i.e., $\Delta>\tfrac32 U$.} We will see that in this regime droplets with side length smaller than $\ell_c$ have a tendency to shrink, while droplets with a side length larger that $\ell_c$ have a tendency to grow. We will refer to the former as \emph{subcritical droplets} and to the latter as \emph{supercritical droplets}. The main difficulty in analysing the metastable behaviour is a proper description of the interaction between the droplets and the surrounding gas. As part of the nucleation process, droplets grow and shrink by exchanging particles with the gas around them, as is typical for conservative dynamics. To describe the evolution of our system in terms of a droplet dynamics, we will show that on an appropriate time scale the dynamics typically returns to the set of configurations consisting of quasi-square droplets, provided the volume is not too large.
The main results of the present paper provide a description of the dynamics in terms of \emph{growing and shrinking} wandering droplets. In particular, Theorems \ref{thm:qsrec}--\ref{thm:res} and \ref{thm:cambiotipico}--\ref{thm:cambioatipico} below identify the dominant rates of growing and shrinking of droplets up to a time horizon that goes well beyond the exit time of ${\cal R}$, namely, up to the time of formation of a droplet with volume of order $\lambda(\beta)$, an unbounded but slowly increasing function of $\beta$. In the follow-up paper \cite{BGdHNOS}, these theorems will be used to identify the nucleation time, i.e., the time of exit of $\cal R$.

\begin{figure}
\centering
\begin{tikzpicture}[scale=0.25,transform shape]		
		\draw [gray, fill=gray, semithick, even odd rule] 
		(20,10) rectangle (37,20) (23,12) rectangle (34,18);
		\draw [gray, fill=gray] (29,14) rectangle (30,15);
		\node at (30.7,15.2){\huge{10}};
		\draw [gray, fill=gray] (19,9) rectangle (20,10);
		\node at (20.5,8.8){\huge{15}};
		\draw [gray, fill=gray] (18,8) rectangle (19,9);
		\node at (19.5,7.8){\huge{14}};
		\draw [gray, fill=gray] (17,7) rectangle (18,8);
		\node at (16.5,8.2){\huge{13}};
		\draw [gray, fill=gray] (16,6) rectangle (17,7);
		\node at (15.5,7.2){\huge{12}};
		\draw [gray, fill=gray] (15,5) rectangle (16,6);
		\node at (14.5,6.2){\huge{11}};
		\draw [gray, fill=gray, semithick] 
		(15,5) rectangle (11,-3);
		\draw [gray, fill=gray, semithick] 
		(11,5) rectangle (10,6);
		\node at (9.7,4.7){\huge{4}};
		\draw [gray, fill=gray, semithick] 
		(10,6) rectangle (9,7);
		\node at (10.3,6.4){\huge{5}};
		\draw [gray, fill=gray, semithick] 
		(11,7) rectangle (10,8);
		\node at (11.3,8.4){\huge{2}};
		\draw [gray, fill=gray, semithick] 
		(9,7) rectangle (8,8);
		\node at (7.7,8.4){\huge{1}};
		\draw [gray, fill=gray, semithick] 
		(9,5) rectangle (8,6);
		\node at (7.7,4.8){\huge{3}};
		\draw [gray, fill=gray, semithick] 
		(15,-3) rectangle (16,-4);
		\node at (15.8,-2.5){\huge{9}};
		\draw [gray, fill=gray, semithick] 
		(15,-5) rectangle (16,-6);
		\node at (14.5,-5){\huge{8}};
		\draw [gray, fill=gray, semithick] 
		(15,-6) rectangle (10,-9);
		\draw [gray, fill=gray, semithick] 
		(18,-6) rectangle (21,-9);
		\draw [gray, fill=gray, semithick] 
		(18,-6) rectangle (17,-5);
		\node at (16.8,-6.5){\huge{7}};
		\draw [gray, fill=gray, semithick] 
		(18,-4) rectangle (17,-3);
		\node at (18.5,-3.6){\huge{6}};
		\draw [gray, fill=gray, semithick] 
		(18,-3) rectangle (23,0);
		\draw [gray, fill=gray, semithick] 
		(32,-1) rectangle (33,0);
		\node at (33.7,0.5){\huge{16}};
\end{tikzpicture}
	
\vskip 0 cm
\caption{\small Each particle is represented by a unit square. A particle is free when it is not touching any other particles and can be moved to infinity without doing so. A particle is clusterised when it is part of a cluster. Particles 1--5 and 16 are free, particles 6--9, 10, 11--15 are not free. All other particles are clusterised. \normalsize}
\label{fig:freeparticle}
\end{figure}
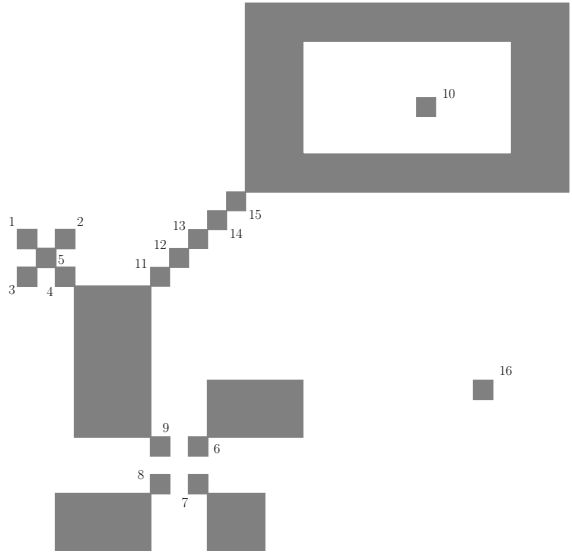
    
\subsection{Main results}
\label{sec:mainres}

	
\subsubsection{Definitions and notation}


\paragraph{$\bullet$ Time horizons.}

In order to state our main results, we first need to clarify the time horizons we are interested in. To this end, we define the set
\begin{equation}
\label{def:R'}
{\cal R}':=\left\{\eta\in{\cal X}_\beta\colon\,
\begin{array}{ll}
\hbox{all clusters of } \eta \hbox{ have volume at most } \ell_c(\ell_c-1)+2 \\
\hbox{except for at most one cluster with volume less than } \tfrac18\lambda(\beta)
\end{array}
\right \},
\end{equation}
where $\lambda(\b)$ is an unbounded but slowly increasing function of $\beta$ satisfying
\begin{equation}
\label{def:lambda}
\lambda(\b)\log\lambda(\b)=o(\log\b), \qquad \beta \to \infty,
\end{equation}
e.g.\ $\lambda(\b)=\sqrt{\log\b}$. For $C^{\star}>0$ large enough, our theorems will hold up to time $T^{\star}$ defined as
\begin{equation}
\label{def:Tstella}
T^{\star} = \ee^{C^{\star}\beta}\wedge\min\{t \geq 0\colon\, X(t)\notin{\cal R'}\}.
\end{equation}
We will see in \cite{BGdHNOS} that our dynamics starting from $\mu_{\cal R}$ typically exits ${\cal R'}$ within a time that is exponentially large in $\beta$, and with a probability tending to $1$ does so through the formation of a single  large cluster $\cal C$ of volume $\tfrac18\lambda(\beta)$, rather than through two supercritical droplets. Hence, $T^{\star}$ indeed coincides with the appearance time of $\cal C$, provided $C^{\star}$ is large enough.

	
\paragraph{$\bullet$ Active and sleeping particles.}
	
As in~\cite{GdHNOS09}, the notion of active and sleeping particle will be crucial throughout the paper. Since the precise definition requires additional notations, we give an intuitive description only. For precise definitions we refer to Section \ref{subsec:color}.
	
The division of particles into active and sleeping is related to the notion of free particles. Intuitively, a particle is \emph{free} if it does not belong to a cluster (= a connected component of nearest-neighbour particles) and can be moved to infinity without clusterisation, i.e., by moving non-clusterised particles only (see Fig.~\ref{fig:freeparticle}). Let 
$$
D=U+d,
$$ 
with $d>0$ sufficiently small. For $t > \ee^{D\beta}$, a particle is said to be \emph{sleeping} at time $t$ if it was not free during the time interval $[t - \ee^{D \beta}, t]$. Non-sleeping particles are called \emph{active}. (Note that being active or sleeping depends on the history of the particle.) By convention, we say that prior to time $\ee^{D\beta}$ sleeping particles are those that belong to a large enough \emph{quasi-square}, where quasi-squares are clusters with sizes $\ell_1 \times \ell_2$ in the set 
\begin{equation}
\label{QS}
\hbox{QS}=\{(\ell_1,\ell_2)\in\mathbb{N}^2\colon\,\ell_1\leq \ell_2\leq \ell_1+1\}.
\end{equation}
In order to declare all the particles in the quasi-square as sleeping before time $\ee^{D\beta}$ we require that $\ell_1\geq 2$.
	

\paragraph{$\bullet$ Local boxes.}	
	
To define a finite box $\Lambda$ as the union of a finite number $k$ of disjoint local boxes $\bar{\Lambda}_i$, $1 \leq i \leq k$, in analogy with the local model introduced in~\cite{dHOS00a}, we associate with each configuration a local configuration
$$
\bar\eta \in \{0,1\}^{\bar\Lambda}=\prod_{1 \leq i \leq k}\{0,1\}^{\bar\Lambda_i},
$$
which we identify with $\{0,1\}^{\Lambda}$. These local boxes allow us to control the global properties of the gas in terms of its local properties, namely, via the duality between gas and droplets, which is represented by the duality between active and sleeping particles, respectively. First, the local boxes have to contain all the sleeping particles. Second, the local boxes are dynamic, namely, $\bar\Lambda_i=\bar\Lambda_i(t)$. Indeed, droplets can move and we want to avoid seeing sleeping particle outside of the local boxes. In particular, the boxes follow the droplets, i.e., must be redefined only when a particular event occurs,  e.g.\ two droplets are too close to each other, or a cluster is too close to the boundary of a box, or a particle outside the boxes falls asleep, or particles in a box all turn active. We denote by $\dist(\cdot,\cdot)$ the distance associated with the $\ell_\infty$-norm on $\R^2$:
\begin{equation}
\label{def:norm}
\|\cdot\|_\infty\colon\, (x,y)\in\R^2 \mapsto |x|\vee|y|.
\end{equation}
Following \cite{G09}, we introduce a map $g_5$ as an iterative map that merges into single rectangles those rectangles that have distance $<5$ between them, while we leave the other rectangles unchanged. (We refer to \eqref{def:mapg} for the precise definition.) At any time $t\geq0$, we require that the collection of the $k(t)$ local boxes $\bar\Lambda(t)=(\bar\Lambda_i(t))_{1 \leq i \leq k(t)}$ satisfy the following conditions associated with $\eta_t=X(t)$:
\begin{description}
\item[B1.] 
$\Lambda(t) = \cup_{1 \leq i \leq k(t)} \bar\Lambda_i(t)$ contains all the sleeping particles.
\item[B2.] 
For all $1 \leq i \leq k(t)$, $\bar\Lambda_i(t)$ contains at least one sleeping particle.
\item[B3.] 
For all $1 \leq i \leq k(t)$, all particles in the restriction $\bar\eta_i(t)$ of $\eta_t$ to $\bar\Lambda_i(t)$ are either free or at distance $>1$ from the internal border of $\bar\Lambda_i(t)$.
\item[B4.] 
For all $1 \leq i,j \leq k(t)$ with $i\neq j$, $\dist(\bar\Lambda_i(t),\bar\Lambda_j(t)) \geq 5$.
\end{description}

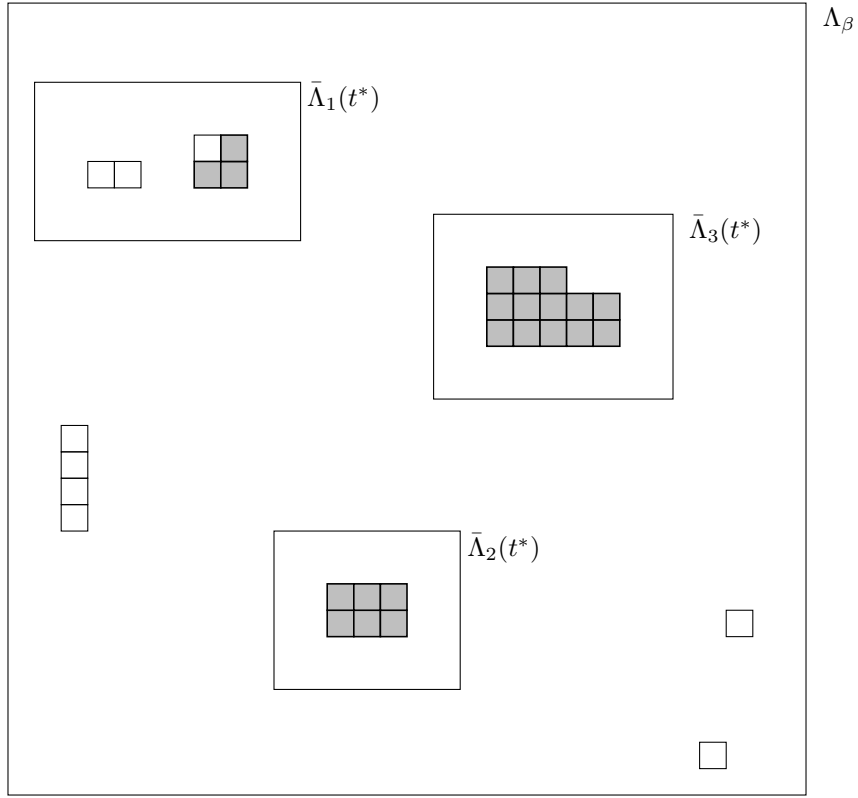
\begin{figure}
\centering
\begin{tikzpicture}[scale=0.35,transform shape]
			
                \draw (0,0) rectangle (30,30);
                \put (305,290){$\Lambda_\beta$}
			
	        \draw (26,1) rectangle (27,2);
	        \draw (27,6) rectangle (28,7);
	        
	        \draw [black, fill=lightgray, semithick] (12,6) rectangle (13,7);
	        \draw [black, fill=lightgray, semithick] (13,6) rectangle (14,7);
	        \draw [black, fill=lightgray, semithick] (14,6) rectangle (15,7);
	        \draw [black, fill=lightgray, semithick] (12,7) rectangle (13,8);
	        \draw [black, fill=lightgray, semithick] (13,7) rectangle (14,8);
	        \draw [black, fill=lightgray, semithick] (14,7) rectangle (15,8);
	        \draw (10,4) rectangle (17,10);
	        \put (172,90){$\bar\Lambda_2(t^*)$}
	        
	        \draw (3,23) rectangle (4,24);
	        \draw (4,23) rectangle (5,24);
	        \draw [black, fill=lightgray, semithick] (7,23) rectangle (8,24);
	        \draw [black, fill=lightgray, semithick] (8,23) rectangle (9,24);
	        \draw [black, fill=lightgray, semithick] (8,24) rectangle (9,25);
	        \draw (7,24) rectangle (8,25);
	        \draw (1,21) rectangle (11,27);
	        \put (112,260){$\bar\Lambda_1(t^*)$}
	        
	        \draw [black, fill=lightgray, semithick] (18,17) rectangle (19,18);
	        \draw [black, fill=lightgray, semithick] (19,17) rectangle (20,18);
	        \draw [black, fill=lightgray, semithick] (20,17) rectangle (21,18);
	        \draw [black, fill=lightgray, semithick] (21,17) rectangle (22,18);
	        \draw [black, fill=lightgray, semithick] (22,17) rectangle (23,18);
	        \draw [black, fill=lightgray, semithick] (18,18) rectangle (19,19);
	        \draw [black, fill=lightgray, semithick] (19,18) rectangle (20,19);
	        \draw [black, fill=lightgray, semithick] (20,18) rectangle (21,19);
	        \draw [black, fill=lightgray, semithick] (21,18) rectangle (22,19);
	        \draw [black, fill=lightgray, semithick] (22,18) rectangle (23,19);
	        \draw [black, fill=lightgray, semithick] (18,19) rectangle (19,20);
	        \draw [black, fill=lightgray, semithick] (19,19) rectangle (20,20);
	        \draw [black, fill=lightgray, semithick] (20,19) rectangle (21,20);
	        \draw (16,15) rectangle (25,22);
	        \put (255,210){$\bar\Lambda_3(t^*)$}
	        
	        \draw (2,13) rectangle (3,14);
	        \draw (2,12) rectangle (3,13);
	        \draw (2,11) rectangle (3,12);
	        \draw (2,10) rectangle (3,11);
\end{tikzpicture}
\caption{\small An example of local boxes $\bar\Lambda(t^*)=(\bar\Lambda_i(t^*))_{1 \leq i \leq 3}$ for $t^*>0$, where the gray and the white particles are sleeping, respectively, active. \normalsize}
\label{fig:scatole}
\end{figure}
						
\begin{definition}
\label{def:boxes}
{\rm The collection of boxes $\bar\Lambda(t)=(\bar\Lambda_i(t))_{1 \leq i \leq k(t)}$ is constructed as follows. At time $t=0$, consider the collection $\bar S(0)$ of $5\times5$ boxes centered at the clusterised particles, and define $\bar\Lambda(0)=g_5(\bar{S}(0))\setminus\bar\Lambda^*(0)$, where $\bar\Lambda^*(0)$ denotes the collection of boxes belonging to $g_5(\bar{S}(0))$ that contain active particles only. Let ${\cal B}$ be the set of special times associated to boxes, refer to as {\it boxes special times}, defined by
\begin{equation}
\label{def:B}
{\cal B}=\left\{t\geq0\colon\,
\hbox{at time } t \hbox{ at least one of the conditions {\bf B1--B4} above is violated by } \bar\Lambda(t^-)\right\}.
\end{equation}
For $t>0$, define $\bar\Lambda(t)$ as follows:
\begin{itemize}
\item 
If $t \in {\cal B}$, then define the collection $\bar S(t)$ of $5\times5$ boxes centered at the clusterised particles, and define $\bar\Lambda(t)=g_5(\bar{S}(t)) \setminus \bar\Lambda^*(t)$, where $\bar\Lambda^*(t)$ denotes the collection of boxes belonging to $g_5(\bar{S}(t))$ that contain active particles only.
\item 
If $t\notin{\cal B}$, then define $\bar\Lambda(t)=\bar\Lambda(t^-)$.
\end{itemize}
We will suppress the dependence on $t$ from the notation whenever it is not relevant. See Fig.~\ref{fig:scatole} for an example of local boxes.} \hfill$\spadesuit$
\end{definition}
	
Since at each time $t$ all the sleeping particles belong to $\bar\Lambda(t)$, the boxes induce a partition of the sleeping particles. We say that a {\it coalescence occurs at time $t$} if there exist two sleeping particles that are in different local boxes at time $t^-$, but are in the same local box at time $t$, i.e., if there exist $1 \leq i_1,i_2 \leq k(t^-)$, $i_1\neq i_2$, $1 \leq i^* \leq k(t)$ and two sleeping particles $s_1,s_2$ such that $s_j\in\bar\Lambda_{i_j}(t^-)$ and $s_j\in\bar\Lambda_{i^*}(t)$, $j=1,2$. This phenomenon is related to the possibility that two droplets join to form a single larger droplet. Coalescence is difficult to control quantitatively, which is why in the present paper we limit ourselves to what happens \emph{in the absence of coalescence}. In the follow-up paper \cite{BGdHNOS} we show that metastable nucleation is unlikely to occur via coalescence.

    
\subsubsection{Key theorems: Theorems \ref{thm:qsrec}--\ref{thm:res} and \ref{thm:cambiotipico}--\ref{thm:cambioatipico}}

    
\paragraph{$\bullet$ Sets and hitting times.}
    
Let ${\cal X}_{\Delta^+}$ be the set of configurations without droplets or with droplets that are quasi-squares with $\ell_1\geq 2$ (and with additional regularity conditions on the gas surrounding droplets to be specified in Definition \ref{def:recurrence}). Let ${\cal X}_E$ be the set of configurations in ${\cal X}_{\Delta^+}$ without droplets (see \eqref{def:vuoto} and Definition \ref{def:recurrence}). Define $(\bar\tau_k)_{k\in\N_0}$ as the sequence of return times in ${\cal X}_{\Delta^+}$ after an active particle is seen in $\Lambda$. Define the hitting time of the set $A\subset{\cal X}_\beta$ for the process $X$ as
\begin{equation}
\tau_A(X) = \inf\{t \geq 0\colon\, X(t)\in A\}.
\end{equation}
Put $\bar\tau_0 = \tau_{{\cal X}_{\Delta^+}}$ and, for $i \in\N_0$, define
\begin{equation}
\label{activereturn1}
\bar\sigma_{i+1} = 
\begin{cases}
\inf\big\{t > \bar\tau_i\colon\,\hbox{there is an active particle in $\Lambda(t)$ at time $t$}\big\}, 
&\hbox{ if } X(\bar\tau_i)\in{\cal X}_{\Delta^+}\setminus{\cal X}_E, \\
\ee^{\Delta\beta}, 
&\hbox{ if } X(\bar\tau_i)\in{\cal X}_E,
\end{cases}
\end{equation}
and
\begin{equation}
\label{deltareturn1}
\bar\tau_{i + 1} = \inf\left\{t > \bar\sigma_{i + 1}\colon\,X(t) \in {\cal X}_{\Delta^+}\right\}.
\end{equation}
Recall that $|\Lambda_\beta|=\ee^{\Theta\b}$. We assume that $\Delta<\Theta\leq\theta$, with $\theta$ defined as follows. Let $\epsilon = 2U - \Delta$, and let $r(\ell_1,\ell_2)$ be the \emph{resistance} of the $\ell_1\times \ell_2$ quasi-square with $1\leq \ell_1\leq \ell_2$ given by
\begin{align}
\label{def:res}
r(\ell_1,\ell_2)
&= \min\{(\ell_1 - 2) \epsilon + 2U, 2\Delta - U\} \notag \\ 
&= \min\{(2U - \Delta) \ell_1 - U + 2\Delta - U, 2\Delta - U\}.
\end{align}

\begin{figure}
	\centering
	\includegraphics[width=0.95\textwidth]{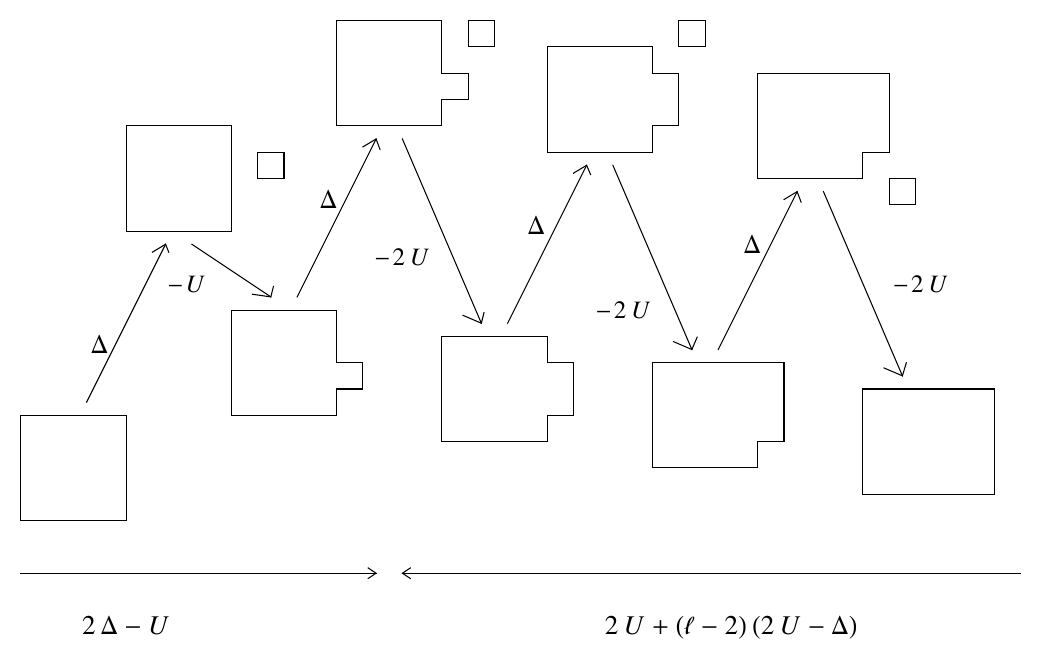}
	\caption{\small Cost of adding or removing a row of length $\ell$ in a finite volume. \normalsize}
	\label{fig:resistenza}
\end{figure}
\noindent
See Fig.~\ref{fig:resistenza}. Let $\theta = 2 \Delta - U - \gamma$ be the resistance of the largest subcritical quasi-square. Since this quasi-square has sizes $(\ell_c - 1) \times \ell_c$, we have $2\Delta - U - \gamma = 2U + ((\ell_c - 1) - 2) \epsilon$, so that
\begin{equation}
\label{def:gamma}
\gamma = \Delta - U - (\ell_c - 2) \epsilon.
\end{equation}
We will see that $\gamma>0$ is an important parameter. The previously mentioned regularity conditions on the gas uses an extra parameter $\alpha>0$ (see below Definition \ref{def:recurrence}), which can be chosen as small as desired. Since we defined $D=U+d$, $\Delta^+$ is defined by $\Delta^+=\Delta+\alpha$. Call a function $f(\b)$ superexponentially small, written $\SES(\beta)$, if
\begin{equation}
\lim_{\b\ra\infty}{1 \over \b}\log f(\b)=-\infty.
\end{equation}


\paragraph{$\bullet$ Key theorems.}

Theorems \ref{thm:qsrec}--\ref{thm:res} and \ref{thm:cambiotipico}--\ref{thm:cambioatipico} below control the transitions between configurations consisting of quasi-squares and free particles, the times scales on which these transitions occur, and the most likely trajectories they follow.  

\medskip\noindent
{\bf (I)} Our first theorem describes the typical return times to the set ${\cal X}_{\Delta^+}$.  
   
\begin{theorem}{\bf [Typical return times]}
\label{thm:qsrec}
If $\Delta<\Theta\leq\theta$, then for any $\d>0$, and any $d$ and $\alpha$ small enough, 
\begin{equation}
P_{\mu_{\cal R}}\Big(\bar\tau_0 \geq \ee^{(\Delta+\alpha+\delta)\beta},\ \bar\tau_0\leq T^{\star}\Big)=\SES(\beta)
\end{equation}
and
\begin{equation}
P_{\mu_{\cal R}}\Big(\ee^{(\Delta-\alpha-\delta)\beta}\leq\bar\tau_{i+1}-\bar\tau_i\leq \ee^{(\Delta+\alpha+\delta)\beta}
\,\,\forall\, i\in\N_0\colon\, \bar\tau_{i+1}\leq T^{\star}\Big)=1-\SES(\beta).
\end{equation}
\end{theorem}

\begin{figure}
\centering
\begin{tikzpicture}[scale=0.35,transform shape]
		
		\draw (0,0) rectangle (30,30);
		\put (305,290){\LARGE{$\Lambda_\beta$}}

		\draw [black, fill=lightgray, semithick] (12,6) rectangle (13,7);
		\draw [black, fill=lightgray, semithick] (13,6) rectangle (14,7);
		\draw [black, fill=lightgray, semithick] (14,6) rectangle (15,7);
		\draw [black, fill=lightgray, semithick] (12,7) rectangle (13,8);
		\draw [black, fill=lightgray, semithick] (13,7) rectangle (14,8);
		\draw [black, fill=lightgray, semithick] (14,7) rectangle (15,8);
		\draw [black, fill=lightgray, semithick] (12,8) rectangle (13,9);
		\draw [black, fill=lightgray, semithick] (13,8) rectangle (14,9);
		\draw [black, fill=lightgray, semithick] (14,8) rectangle (15,9);

		\draw [black, fill=lightgray, semithick] (7,23) rectangle (8,24);
		\draw [black, fill=lightgray, semithick] (8,23) rectangle (9,24);
		\draw [black, fill=lightgray, semithick] (8,24) rectangle (9,25);
		\draw [black, fill=lightgray, semithick] (7,24) rectangle (8,25);
		\draw [black, fill=lightgray, semithick] (9,23) rectangle (10,24);
		\draw [black, fill=lightgray, semithick] (9,24) rectangle (10,25);

		\draw [black, fill=lightgray, semithick] (18,14) rectangle (19,15);
		\draw [black, fill=lightgray, semithick] (19,14) rectangle (20,15);
		\draw [black, fill=lightgray, semithick] (20,14) rectangle (21,15);
		\draw [black, fill=lightgray, semithick] (21,14) rectangle (22,15);
		\draw [black, fill=lightgray, semithick] (22,14) rectangle (23,15);
		\draw [black, fill=lightgray, semithick] (18,15) rectangle (19,16);
		\draw [black, fill=lightgray, semithick] (19,15) rectangle (20,16);
		\draw [black, fill=lightgray, semithick] (20,15) rectangle (21,16);
		\draw [black, fill=lightgray, semithick] (21,15) rectangle (22,16);
		\draw [black, fill=lightgray, semithick] (22,15) rectangle (23,16);
		\draw [black, fill=lightgray, semithick] (18,16) rectangle (19,17);
		\draw [black, fill=lightgray, semithick] (19,16) rectangle (20,17);
		\draw [black, fill=lightgray, semithick] (20,16) rectangle (21,17);
		\draw [black, fill=lightgray, semithick] (21,16) rectangle (22,17);
		\draw [black, fill=lightgray, semithick] (22,16) rectangle (23,17);
		\draw [black, fill=lightgray, semithick] (18,17) rectangle (19,18);
		\draw [black, fill=lightgray, semithick] (19,17) rectangle (20,18);
		\draw [black, fill=lightgray, semithick] (20,17) rectangle (21,18);
		\draw [black, fill=lightgray, semithick] (21,17) rectangle (22,18);
		\draw [black, fill=lightgray, semithick] (22,17) rectangle (23,18);
		\draw [black, fill=lightgray, semithick] (18,18) rectangle (19,19);
		\draw [black, fill=lightgray, semithick] (19,18) rectangle (20,19);
		\draw [black, fill=lightgray, semithick] (20,18) rectangle (21,19);
		\draw [black, fill=lightgray, semithick] (21,18) rectangle (22,19);
		\draw [black, fill=lightgray, semithick] (22,18) rectangle (23,19);
		\draw [black, fill=lightgray, semithick] (18,19) rectangle (19,20);
		\draw [black, fill=lightgray, semithick] (19,19) rectangle (20,20);
		\draw [black, fill=lightgray, semithick] (20,19) rectangle (21,20);
		\draw [black, fill=lightgray, semithick] (21,19) rectangle (22,20);
		\draw [black, fill=lightgray, semithick] (22,19) rectangle (23,20);
		
		\draw (26,1) rectangle (27,2);
		\draw (27,6) rectangle (28,7);
		\draw (13,11) rectangle (14,12);
		\draw (7,9) rectangle (8,10);
		\draw (1,2) rectangle (2,3);
		\draw (27,17) rectangle (28,18);
		\draw (23,27) rectangle (24,28);
\end{tikzpicture}
\caption{\small An example of a configuration $\eta\in{\cal X}_{\Delta^+}$, where the gray and the white particles are sleeping, respectively, are active, such that $\pi(\eta)=\{(2,3),(3,3),(5,6)\}$.\normalsize}
\label{fig:proiezione}
\end{figure}
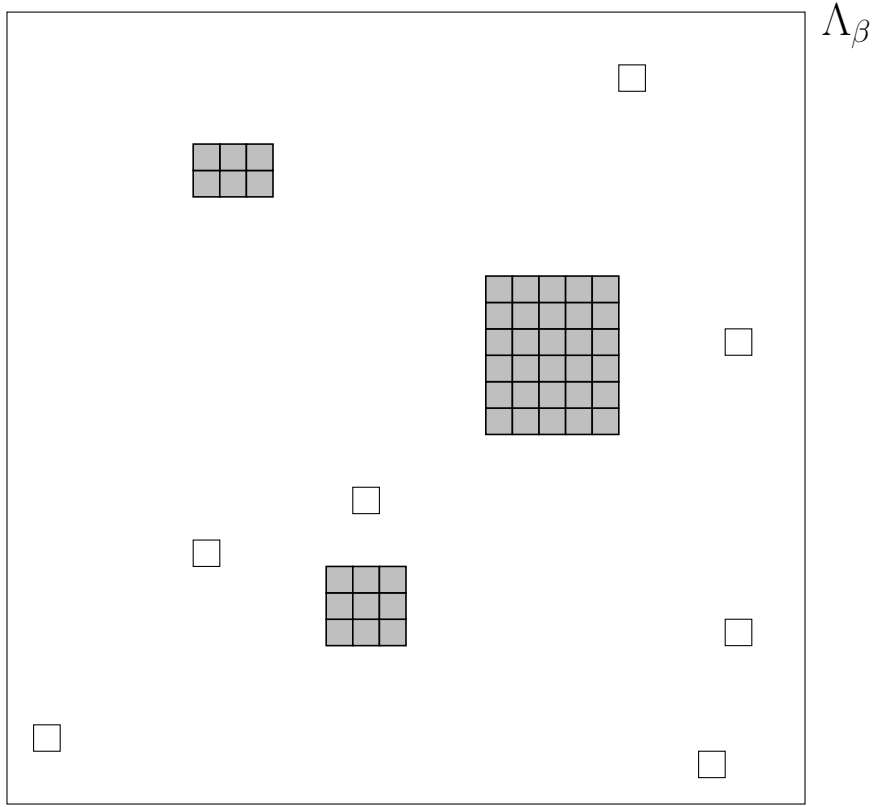

\medskip\noindent
{\bf (II)} 
Our second theorem describes the typical update times for a configuration in $\mathcal{X}_{\Delta^+}$. Denoting by $\pi$ a projection from ${\cal X}_{\Delta^+}$ to a finite space
\begin{equation}
\bar{\cal X}_{\Delta} = \bigcup_{k\geq0}\hbox{QS}^1\times\cdots\times\hbox{QS}^{k},
\end{equation}
where $\hbox{QS}^i$ are the sizes of the quasi-square clusters contained in the local boxes $\bar\Lambda_i$ and are defined in \eqref{QS}. See Fig.~\ref{fig:proiezione}. We can define a dynamics on the space $\bar{\cal X}_{\Delta}$ of sizes of quasi-squares, arranged for example in increasing lexicographic order. For $i\in\N_0$, we denote by $(l_{1,i},l_{2,i})$ in $QS$, with $l_{1,i} \geq 2$, the sizes of the smallest quasi-square at time $\bar\tau_i$, if any, and otherwise we set $l_{1,i}=l_{2,i}=0$. Define 
\begin{equation}
\label{def:cambio}
\bar\tau_{c,i} = \min\{\bar\tau_k\geq\bar\tau_i\colon\,\pi(X(\bar\tau_k))\neq\pi(X(\bar\tau_i))\},
\end{equation}
recall \eqref{def:res}, and define the resistance of a configuration in ${\cal X}_E$ by
\begin{equation}
\label{def:resvuoto}
r(0,0)=4\Delta-2U-\theta.
\end{equation}

\begin{theorem}{\bf [Typical update times]}
\label{thm:res}
If $\Delta<\Theta\leq\theta$, then for any $\delta>0$, any $d$ and $\alpha$ small enough, and any $i\in\N_0$,
\begin{equation}
\label{eq1}
P_{\mu_{\cal R}}\left(
\begin{array}{ll}
\hbox{if } \bar\tau_{c,i}\leq T^{\star}, \hbox{ then }
\bar\tau_{c,i}-\bar\tau_i\leq \ee^{(r(l_{1,i},l_{2,i})+\delta)\beta} \\
\hbox{or a coalescence occurs between } \bar\tau_i \hbox{ and } \bar\tau_{c,i}
\end{array}
\right)
=1-\SES(\beta)
\end{equation}
and
\begin{equation}
\label{eq2}
\lim_{\beta\to\infty}
P_{\mu_{\cal R}}\left(
\begin{array}{ll}
\hbox{if } \bar\tau_{c,i}\leq T^{\star}, \hbox{ then }
\bar\tau_{c,i}-\bar\tau_i\geq \ee^{(r(l_{1,i},l_{2,i})-\delta)\beta} \\
\hbox{or a coalescence occurs between } \bar\tau_i \hbox{ and } \bar\tau_{c,i}
\end{array}
\right)
=1.
\end{equation}
\end{theorem}
	  
\begin{remark}
{\rm Theorem \ref{thm:res} states that, starting from $\mu_{\cal R}$ and unless a coalescence occurs, for any $i\in\N_0$ the projected dynamics typically remains in $\pi(X(\bar\tau_i))$ through successive visits in ${\cal X}_{\Delta^+}$ for a time of order $\ee^{r(l_{1,i},l_{2,i})\beta}$. The $\SES$ error in \eqref{eq1} is related to an anomalously large realisation of a geometric random variable, while an anomalously small realisation leads to an error that is only exponentially small in \eqref{eq2}. Note that for $l_{1,i} \geq \ell_c$ all the quasi-squares have the same resistance $2\Delta-U$. For the case in which $X(\bar\tau_i)$ has no quasi-square, its resistance $r(0,0)$ involves the resistance of the empty configuration in the local model and a spatial entropy that comes from the position in $\Lambda_\beta$ where the new droplet can appear.}\hfill$\spadesuit$
\end{remark}

\medskip\noindent
{\bf (III)} 
Our third theorem describes the typical transition of the system between two successive visits to $\mathcal{X}_{\Delta^+}$ conditional on the dynamics not returning to the same configuration at time $\bar\tau_{i+1}$. Given a configuration $X(\bar\tau_i)\in{\cal X}_{\Delta^+}$, define the typical transition $\pi'_i$ as follows. For $l_{1,i}\geq \ell_c$, set 
$$
\pi'_i=\big\{\pi(\eta')\colon\, \eta' \hbox{ is a configuration obtained from }
X(\bar\tau_i) \hbox{ by adding a row to an arbitrary quasi-square}\big\}. 
$$
For $l_{1,i}<\ell_c$, we need to distinguish between the cases $l_{2,i}\geq3$, $l_{2,i}=2$ and $l_{2,i}=0$. If $l_{1,i}<\ell_c$ and $l_{2,i}\geq3$ (respectively, $l_{2,i}=2$), then we define $\pi'_i$ as the singleton made up of the collection of sizes of quasi-squares obtained from $\pi(X(\bar\tau_i))$ by modifying one of the smallest quasi-squares, which becomes $(l_{2,i}-1)\times l_{1,i}$ (respectively, $0\times0$). If $l_{1,i}=l_{2,i}=0$, then we define $\pi'_i=\{\pi(\eta')\}$, where $\eta'$ is the configuration obtained from $X(\bar\tau_i)$ by creating a $2\times2$ square droplet, namely, $\pi'_i=\{(2,2)\}$.
	  	  
\begin{theorem}{\bf [Typical transitions]}
\label{thm:cambiotipico}
If $\Delta<\Theta\leq\theta$, then for any $d$ and $\alpha$ small enough, and any $i\in\N_0$,
\begin{equation}
\lim_{\beta\to\infty}
P_{\mu_{\cal R}} \left( \left.
\begin{array}{ll}
\hbox{ if } \bar\tau_{i+1}\leq T^{\star}, \hbox{ then }
\pi(X(\bar\tau_{i+1}))\in\pi'_i\\
\hbox{or a coalescence occurs between } \bar\tau_i \hbox{ and } \bar\tau_{i+1}
\end{array}
\right|
\pi(X(\bar\tau_{i+1}))\neq\pi(X(\bar\tau_i))
\right) =1.
\end{equation}
\end{theorem}
	
\medskip\noindent
{\bf (IV)}  
Our fourth and last theorem characterises the atypical transitions of the system, starting from a subcritical configuration consisting of a single quasi-square, between two successive visits to $\mathcal{X}_{\Delta^+}$, with no creation of new boxes and conditional on the dynamics not returning to the same configuration at time $\bar\tau_i$. To this end, given $X(\bar\tau_i)\in{\cal X}_{\Delta^+}$ with $2\leq l_{1,i}<\ell_c$, we define $\pi''_i=(l_{2,i},l_{1,i}+1)$. Moreover, we say that a {\it box creation occurs at time $t$} if there exists an active particle at time $t^-$ that does not belong to $\Lambda(t^-)$ and falls asleep at time $t$. 
              
\begin{theorem}{\bf [Atypical transitions]}
\label{thm:cambioatipico}
If $\Delta<\Theta\leq\theta$, then for any $d$ and $\alpha$ small enough, and any $i\in\N_0$ such that $X(\bar\tau_i)\in{\cal X}_{\Delta^+}$ consists of a single quasi-square with $2\leq l_{1,i}<\ell_c$,
\begin{equation}
P_{\mu_{\cal R}}\left( \left.
\begin{array}{ll}
\hbox{ if } \bar\tau_{i+1}\leq T^{\star}, \hbox{ then } \pi(X(\bar\tau_{i+1}))=\pi''_i \hbox{ and} \\
\hbox{no box creation occurs between } \bar\tau_i \hbox{ and } \bar\tau_{i+1}
\end{array}
\right |
\pi(X(\bar\tau_{i+1}))\neq\pi(X(\bar\tau_i))
\right )
\geq \ee^{-[(2\Delta-U)-r(\ell_1,\ell_2)+\delta]\beta}.
\end{equation}
\end{theorem}
       
\begin{remark}
{\rm Theorem \ref{thm:cambioatipico} provides a lower bound for the atypical transition of `going against the drift' in the case of a subcritical quasi-square. As we will show in the follow-up paper \cite{BGdHNOS}, the {\it escape from metastability} occurs via nucleation of a supercritical droplet somewhere in the box $\Lambda_\beta$. Indeed, we will characterise the time the dynamics needs to exit ${\cal R}$, as well as the typical paths of configurations visited by the wandering cluster until the formation of a large droplet. The results of the present paper, which are limited to the case $\Theta<2\Delta-U-\gamma$, will allow us to accomplish this task for larger values of $\Theta$, namely, $\Theta<\Gamma-(2\Delta-U)$, where $\Gamma$ is the energy of the critical droplet in the local model.} \hfill$\spadesuit$
\end{remark}

\begin{remark}
{\rm The techniques developed in the present paper make it possible to prove that, for any quasi-square configuration of size $\ell_1\times \ell_2$ in ${\cal X}_{\Delta^+}$, the cluster exits any finite box centered around the cluster with a volume that does not depend on $\beta$, within a time of order $\ee^{r(\ell_1,\ell_2)\beta}$. This is the reason why we speak of a \emph{wandering cluster}. We will not state nor use this result as a formal theorem.}\hfill$\spadesuit$
\end{remark}


\subsection{Outline}
\label{sec:discout}

Section \ref{sec:tools} collects certain key tools that are needed throughout the paper. In particular, in Section \ref{sec:notation} we introduce key notation, in Section \ref{sec:enest} we formulate certain regularity properties for the initial configuration that we can impose because their failure is extremely unlikely, while in Section~\ref{sec:rec} we group the configurations into a sequence of subsets of  configurations of increasing regularity and prove a recurrence property to these sets on an increasing sequence of time scales. In Section \ref{sec:lemmas} we state three key propositions (Propositions \ref{prp:lbtipico}--\ref{prp:uppbound}) that are needed along the way. The proof of Theorems \ref{thm:qsrec}, \ref{thm:res}, \ref{thm:cambiotipico} and \ref{thm:cambioatipico} are given in Sections \ref{sec:proof1}, \ref{sec:proof2}, \ref{sec:proof3} and \ref{sec:proof4}, respectively, subject to these propositions.

The three propositions are proved in Sections \ref{sub:prp1}--\ref{sub:prp3}. The proof is based on a number of key lemmas (Lemmas \ref{lmm:exitempty}, \ref{lmm:exitgen}, \ref{lmm:exit2x2}) whose proof is given in Section~\ref{sub:lmmgen}. This section, which uses two more key lemmas (Lemmas \ref{gex}--\ref{gorgonzola}), is long and difficult because it contains the main technical hurdles of the paper. These hurdles are organised into what we call the \emph{deductive approach}: the tube of typical trajectories  leading to nucleation is described via a series of events, whose complements have negligible probability, on which the evolution of the gas consists of \emph{droplets wandering around on multiple space-time scales} in a way that can be captured by a coarse-grained Markov chain on a space of droplets. 

Appendices \ref{sec:appa} and \ref{sec:appb3} provide additional computations that are needed in the paper: environment estimates that exclude non-regular configurations, respectively, large deviation estimates for certain events that come up in the deductive approach.

	
\section{Key tools}
\label{sec:tools}
	
In this section we provide some tools that are needed to prove the theorems. These tools rely on the notion of QRWs (Quasi-Random Walks). In \cite{GdHNOS09} it was shown that the active particles of a two-dimensional lattice gas under Kawasaki dynamics at low density evolve in a way that is close to an \emph{ideal gas}. The results in  \cite{GdHNOS09} are formulated in the general context of QRWs for  a large class of initial conditions having no anomalous concentration of particles for time horizons that are much larger than the typical collision time. More precisely, the process of QRWs used to describe the ideal gas approximation consists of $N$ labelled particles that can be coupled to a process of $N$ Independent Random Walks (IRWs) in such a way that the two processes follow the same paths outside rare time intervals, called pause intervals, in which the paths of the QRWs remain confined to small regions. 

For the definition of QRWs and their construction, we refer to \cite[Sections 2.2-2.4]{GdHNOS09}. We note that for the notion of sleeping and active particles to be well defined, we need to label the {\it particles} and not work with a dynamics of {\it configurations} $\eta\in{\cal X}_\beta$ only, as defined in \eqref{gendef}. There is flexibility in associating a particle dynamics with a configuration dynamics. In particular, as in \cite{GdHNOS09} we can allow instantaneous permutation of particles inside a given cluster. Later we will use this flexibility by specifying a local permutation rule (see Section \ref{subsec:color}). For now we only assume that such a rule has been chosen. We encourage the reader to inspect the main properties of QRWs, which will be a key tool in the remaining part of the paper. In particular, we refer to \cite[Theorems 3.2.3, 3.2.4, 3.2.5, 3.3.1]{GdHNOS09} for the \emph{non-superdiffusivity} property and for upper and lower bounds on the \emph{spread-out property}, respectively.
	
	
\subsection{Definitions and notations}
\label{sec:notation}
	
In this section we introduce some definitions and notations that will be needed throughout the sequel.

\begin{definition}
$\mbox{}$
{\rm 
\begin{itemize}
\item[1.] 
As in \cite{GdHNOS09}, $\alpha$ and $d$ are two positive parameters that can be chosen as small as desired, and $\lambda(\b)$ is an unbounded but slowly increasing function of $\beta$ that satisfies \eqref{def:lambda}. Moreover, $C^{\star}$ is a positive parameter that can be chosen as large as desired. Once chosen, $\alpha$, $d$, $\lambda$ and $C^{\star}$ are fixed. We write $O(\delta)$, $O(\alpha)$ and $O(d)$ for quantities with an absolute value that can be bounded by a constant times $|\delta|$, $|\alpha|$ and $|d|$, for small enough values of these parameters. We write $O(\delta,\alpha,d)$ for the sum of three such quantities.
\item[2.] 
We use short-hand notation for a few quantities that depend on the old parameters $\Delta \in (\tfrac32 U,2U)$ and $\Theta \in (\Delta,2\Delta-U)$, and on the new parameters $\alpha$, $d$. Recall that
$$
\epsilon=2U-\Delta, \qquad \ell_c=\Big\lceil\frac{U}{\epsilon}\Big\rceil,
\qquad \gamma=\Delta-U-(\ell_c-2)\epsilon,
\qquad \theta=2\Delta-U-\gamma,
$$
and  
$$
D=U+d, \qquad \Delta^+=\Delta+\alpha,
$$
and abbreviate
\begin{equation}
\label{Spardef}
S=\dfrac{4\Delta-\theta}{3}-\alpha.
\end{equation}
For $C>0$, write $T_C$ for the time scale $T_C=\ee^{C\beta}$.
\item[3.]
For convenience we identify a configuration $\eta\in\cX_\beta$ with its support $\hbox{supp}(\eta)=\{z\in\Lambda_\beta: \eta(z)=1\}$ and write $z\in\eta$ to indicate that $\eta$ has a particle at $z$. For $\h\in\cX_\beta$, denote by $\h^{cl}$ the {\it clusterised part of} $\h$:
\begin{equation}
\label{def:cluster}
\h^{cl} = \{z\in\h\colon\,\parallel z-z'\parallel =1 \hbox{ for some } z'\in\h\}.
\end{equation}
Call {\it clusters of} $\h$ the connected components of the graph drawn on $\h^{cl}$ obtained by connecting nearest-neighbour sites
that are not a singleton.
\item[4.]
Denote by $B(z,r), z\in\R^2, r>0$, the open ball with center $z$ and radius $r$ in the norm defined in \eqref{def:norm}. The closure of $A\subset\R^2$ is denoted by $\overline{A}$.
\item[5.] 
For $A\subset\Z^2$, denote by $\partial^- A$ the internal boundary of $A$, i.e.,
\begin{equation}
\label{def:bordint}
\partial^- A = \{z\in A\colon\, \parallel z-z'\parallel=1 \hbox{ for some } z'\in\Z^2\setminus A\}.
\end{equation}
For $s>0$, put
\begin{equation}
[A,s] = \bigcup_{z\in A}\overline{B(z,\ee^{\frac{s}{2}\beta})}\cap\Z^2.
\end{equation}
Call $A$ a rectangle on $\Z^2$ if there are $a,b,c,d\in\R$ such that
\begin{equation}
A = [a,b] \times [c,d] \cap\Z^2.
\end{equation}
Write $\hbox{RC(A)}$, called the {\it circumscribed rectangle} of $A$, to denote the intersection of all the rectangles on $\Z^2$ containing $A$. Moreover, denote by ${\cal R}$ the set of all finite sets of rectangles on $\Z^2$.
\item[6.] 
Given $\sigma\geq0$ and $\bar S=\{R_1,\ldots,R_{|S|}\}\in{\cal R}$, two rectangles $R$ and $R'$ in $\bar S$ are said to be in the same equivalence class if there exists a finite sequence $R_1,\ldots,R_k$ of rectangles in $\bar S$ such that
$$
R=R_1, \quad R'=R_k, \quad  \dist(R_j,R_{j+1})<\sigma\,\, \forall\, 1 \leq j < k.
$$
Let $C$ indicate the set of equivalent classes, define the map
$$
\bar g_\sigma\colon\, \bar S\in{\cal R} \mapsto
\Bigg\{\hbox{RC}\Bigg(\bigcup_{j \in c}R_j\Bigg)\Bigg\}_{c\in C}\in{\cal R},
$$
and let $(\bar g_\sigma^{(k)})_{k\in\N_0}\in {\cal R}^{\mathbb{N}}$ be the sequence of iterates of $\bar g_\sigma$. Define
\begin{equation}
\label{def:mapg}
g_\sigma(\bar S)= \lim_{k\rightarrow\infty} \bar g_\sigma^{(k)}(\bar S).
\end{equation}
As discussed in \cite{G09}, the sequence $(\bar g_\sigma^{(k)}(\bar S))_{k\in\N_0}$ ends up being a constant, so the limit is well defined.
\end{itemize}
}\hfill$\spadesuit$
\end{definition}
	
	
\subsection{Environment estimates}
\label{sec:enest}
	
In this section we introduce a subset of configurations ${\cal X}^*\subset{\cal X}_\beta$, to which we refer as the \emph{typical environment}, with the property that if our system is started from the restricted ensemble, then it can escape from $\mathcal{X}^*$ within any time scale that is exponential in $\beta$ with a negligible probability only. Boxes are square boxes, and we require that $\Theta>\Delta$. Recall \eqref{Spardef} for the definition of the parameter $S$.

\begin{remark}
	\label{rmk:S}
	{\rm The choice of $S$ comes from the fact that we require the probability to have 4 particles anywhere in a box of volume $\ee^{S\beta}$ to tend to zero under the measure $\mu_{\cal R}$ as $\beta\to\infty$.}\hfill$\spadesuit$
\end{remark}

\begin{definition}
{\rm Define
\begin{equation}
{\cal X}^*=\bigcap_{j=1}^5{\cal X}_j^*,
\end{equation}
where, for $\lambda$ satisfying \eqref{def:lambda} and $S$ given by \eqref{Spardef},
\begin{align}
{\cal X}_1^*& = \left\{\eta\in{\cal X}_\beta\colon
\begin{array}{ll}
\hbox{in any box of volume } \ee^{\theta\beta} \hbox{ the number of clusters is at most } \lambda(\beta)\\
\end{array}
\right\}, \notag \\[0.2cm]
{\cal X}_2^*&=\left\{\eta\in{\cal X}_\beta\colon
\begin{array}{ll}
\hbox{in any box of volume } \ee^{\theta\beta} \hbox{ the number of 4-tuples of particles in different} \\
\hbox{connected components with diameter smaller than } \sqrt{\ee^{S\beta}} \hbox{ is at most } \lambda(\beta)
\end{array}
\right\}, \notag \\[0.2cm]
{\cal X}_3^*&=\left\{\eta\in{\cal X}_\beta\colon
\begin{array}{ll}
\hbox{in any box of volume } \ee^{\theta\beta} \hbox{ the number of 4-tuples of particles in different} \\
\hbox{connected components with diameter smaller than } \sqrt{\ee^{A\beta}} \\ 
\hbox{is at most } \ee^{(3A-4\Delta+\theta+4\alpha)\beta} \hbox{ for any } S<A<\Delta^+
\end{array}
\right\}, \notag \\[0.2cm]
{\cal X}_4^*&=\left\{\eta\in{\cal X}_\beta\colon
\begin{array}{ll}
\hbox{in any box of volume } \ee^{(\Delta+\alpha)\beta} \hbox{ the number of particles is at most } \ee^{\frac{3}{2}\alpha\beta} \\\hbox{and at least } \ee^{\frac{1}{2}\alpha\beta}
\end{array}
\right\}, \notag \\[0.2cm]
{\cal X}_5^*&=\left\{\eta\in{\cal X}_\beta\colon
\begin{array}{ll}
\hbox{in any box of volume } \ee^{(\Delta-\frac{\alpha}{4})\beta}
\hbox{ the number of particles is at most } \tfrac14 \lambda(\beta)
\end{array}
\right\} \notag.
\end{align}
}\hfill$\spadesuit$
\end{definition}

\begin{remark}
{\rm The exit time of ${\cal X}^*_5$ coincides with the first time ${\cal T}_{\alpha,\lambda}$ when an anomalous concentration event occurs (\cite{GdHNOS09}). Since the QRW-estimates of \cite{GdHNOS09} hold up to this time, we can use them as long as the system stays in ${\cal X^*}\subset{\cal X}^*_5$.} \hfill$\spadesuit$
\end{remark}

\begin{remark}
{\rm The reason why we define ${\cal X}^*$ for any $\Theta>\Delta$ (i.e., without the restriction $\Theta\leq\theta$) is that in \cite{BGdHNOS} we will need Proposition \ref{prp:ambito}, which says that starting from $\mu_{{\cal R}}$ the system exits ${\cal X}^*$ within any given exponential time with $\SES(\beta)$ probability only. The main theorems of the present paper, which hold for the dynamics in a box of volume at most $\ee^{\theta\beta}$, with {\em periodic} boundary conditions, immediately extend to the case where such a small box is embedded into a larger box of volume $\ee^{\Theta\beta}$, with {\em open} boundary conditions, as long as the system remains in the typical environment ${\cal X}^*$.} \hfill$\spadesuit$
\end{remark}
    
Recall the definition of the set ${\cal R}$ given in \eqref{def:R}, of the set $\cal R'\supset\cal R$ given in \eqref{def:R'} and of the time ${T^{\star}}$ given in \eqref{def:Tstella}.
	
\begin{proposition}
\label{prp:ambito}
$P_{\mu_{\cal R}}(\tau_{{\cal X}_\beta\setminus{\cal X}^*} \leq T^{\star})=\SES(\beta)$.
\end{proposition}

\begin{proof}
Denote by $A_t$ the event that the dynamics exits ${\cal X}^*$ at time $t$, and by $A_t^{\cal R'}$ the event $A_t$ when the dynamics is restricted to ${\cal R'}$ (by ignoring jumps that would lead the dynamics out of ${\cal R'}$). Given a Poisson process on $\mathbb{R}_+$ with rate $\ee^{\Theta\beta}$, denote by $M(t)$ the number of times the clock rings up to time $t\geq0$, and write $P$ to denote its law. Let $(\check{X}_k)_{k\in\mathbb{N}}$ be the the embedded discrete-time process such that the original process $(X(t))_{t\geq0}$ can be written as $X(t)=\check{X}_{M(t)}$. Estimate, for $\delta>0$,
$$
\begin{array}{ll}
P_{\mu_{\cal R}}(\exists \ t<T^{\star},A_t)
&=\displaystyle\sum_{\eta\in{\cal R}}\frac{\mu(\eta)}{\mu({\cal R})} P_\eta(\exists \ t<T^{\star},A_t) \\
&\displaystyle\leq\dfrac{\mu({\cal R'})}{\mu({\cal R})}
\sum_{\eta\in{\cal R'}}\dfrac{\mu(\eta)}{\mu({\cal R}')} P_\eta(\exists \ t<T^{\star},A_t)\\
&\leq\displaystyle\frac{\mu({\cal R'})}{\mu({\cal R})}
P_{\mu_{\cal R'}}\Bigl(\exists \ t<\ee^{C^{\star}\beta},A_t^{\cal R'}\Bigr) \\
&\displaystyle\leq \frac{\mu({\cal R'})}{\mu({\cal R})} 
\Big[P\Bigl(M(\ee^{C^{\star}\beta})\geq \ee^{(\Theta+C^{\star}+\delta)\beta}\Bigr)
+\sum_{1 \leq k<\ee^{(\Theta+C^{\star}+\delta)\beta}}P_{\mu_{\cal R'}}
(\check{X}_k\in{\cal X}_\beta\setminus{\cal X}^*)\Big]\\
&\displaystyle\leq \frac{\mu({\cal R'})}{\mu({\cal R})}
\Bigl[\SES(\beta)+\ee^{(\Theta+C^{\star}+\delta)\beta}\sum_{i=1}^5 \mu_{\cal R'}(({{\cal X}_i}^*)^c)\Bigr],
\end{array}
$$
where the term $\SES(\beta)$ comes from the Chernoff bound for a Poisson random variable, and $\mu_{{\cal R}'}$ stands for
the grand-canonical Gibbs measure conditioned to ${\cal R}'$. To get the claim, note that, because $\mu({\cal R'})/\mu({\cal R})\leq \ee^{\bar{C}\beta}$ for some $\bar{C}>0$, it suffices to prove that $\mu_{\cal R'}(({{\cal X}^*_i})^c)=\SES(\beta)$ for any $i$. This is done in Appendix \ref{sec:appa}.
\end{proof}

\begin{remark}
{\rm Proposition~\ref{prp:ambito} allows us to work with configurations in ${\cal X}^*$. Replacing the original dynamics by the dynamics restricted to ${\cal X}^*$, we can couple the two dynamics in such a way that they have the same trajectories up to any time that is exponential in $\beta$ with probability $1-\SES(\beta)$.} \hfill$\spadesuit$
\end{remark}

	
\subsection{Recurrence properties}
\label{sec:rec}
	
In this section we group the configurations in $\cal X_\beta$ into a sequence of subsets of  configurations of \emph{increasing regularity}, and we prove a \emph{recurrence property} of the associated Markov processes restricted to these sets on an \emph{increasing sequence of time scales}. To that end, denote by
$$
\bar H_{i}(\bar\eta_i)=-U
\sum_{\{x,y\}\in\bar\Lambda_i(t)^*}\eta(x)\eta(y)
+\Delta\sum_{x\in\bar\Lambda_i(t)}\eta(x)
$$
the \emph{local energy} of $\bar\eta_i=\eta_{|\bar\Lambda_i}$ at time $t$ inside the box $\bar\Lambda_i(t)$, where $\bar\Lambda_i(t)^*$ denotes the set of bonds in $\bar\Lambda_i(t)$. We emphasise that, alongside the local model, we need to introduce two additional sets, ${\cal X}_D$ and ${\cal X}_S$, to control the regularity of the gas surrounding the droplets.
    
\begin{remark}
{\rm With each particle $i$ we can associate, at any time $t\geq0$, the time
\begin{equation}
\label{def:sleeping}
s_i(t) = \inf\big\{s\in{[0,t)}\colon\, \hbox{particle } i \hbox{ is not free during the entire time interval } [s,t]\big\},
\end{equation}
 so that $s^*_i(t)=\ee^{D\beta}-(t-s_i(t))$ is the time that particle $i$ needs to remain not free in order to fall asleep. By convention, for a sleeping (respectively, free) particle at time $t$ we put $s^*_i(t)=0$  (respectively, $s^*_i(t)=\infty$). In this way we are able to characterise active and sleeping particles at any time $t$. In addition, the process $Y=(X(t),(s^*_i(t))_{i=1}^N,\bar\Lambda(t))_{t\geq0}$ is Markovian. In the sequel we will simply refer to this process as the original process $X=(X(t))_{t\geq0}$. In Section \ref{subsec:color} we will consider a slight generalisation of the process $Y$ in which more information about particles is included, again referring to it as the original process $X$.}\hfill$\spadesuit$
\end{remark}
	
\begin{definition}
\label{def:recurrence}
{\rm For any time $t\geq0$, given a configuration $\eta_t=X(t)\in{\cal X}_{\beta}$ and the collection $\bar\Lambda(t) = (\bar\Lambda_i(t))_{1 \leq i \leq k(t)}$ of finite boxes in $\Lambda_\beta$ as in Definition \ref{def:boxes}, we say that $\eta_t$ is 0-reducible (respectively, $U$-reducible) if for some $i$ the local energy of $\bar\eta_i$ can be reduced along the dynamics with constant $\bar\Lambda(t)$ without exceeding the energy level $\bar H_{i}(\bar\eta_i)+0$ (respectively, $\bar H_{i}(\bar\eta_i)+U$). If $\eta_t$ is not $0$-reducible or $U$-reducible, then we say that $\eta_t$ is $0$-irreducible or $U$-irreducible, respectively. We define
$$
\begin{array}{ll}
{\cal X}_0
&=\{\eta_t\in{\cal X}^*\colon\,\eta_t \hbox{ is $0$-irreducible}\}, \\[0.2cm]
{\cal X}_U
&=\{\eta_t\in{\cal X}_0\colon\,\eta_t \hbox{ is $U$-irreducible}\}, \\[0.2cm]
{\cal X}_D 
&=\{\eta_t\in{\cal X}_U\colon\,\hbox{all the particles in } \Lambda(t) \hbox{ are sleeping}\}, \\[0.2cm]
{\cal X}_S
&=\{\eta_t\in{\cal X}_D\colon\,\hbox{each box of volume } \ee^{S\beta} \hbox{ contains three active particles at most}\}, \\[0.2cm]
{\cal X}_{\Delta^+}
&=\left\{\eta_t\in{\cal X}_S\colon\,
\begin{array}{ll}
\bar\eta_t \hbox{ is a union of at most } \lambda(\beta) \hbox{ quasi-squares with} \\
\hbox{no particle inside } \bigcup_i[\bar\Lambda_i(t),\Delta-\alpha] 
\hbox{ except for those} \\
\hbox{in the quasi-squares, one for each local box } \bar\Lambda_i(t) 
\end{array}
\right\},
\end{array}
$$
where $[\bar\Lambda_i(t),\Delta-\alpha]$ are the boxes of volume $\ee^{(\Delta-\alpha)\beta}$ with the same center as $\bar\Lambda_i(t)$.}\hfill$\spadesuit$
\end{definition}

\begin{remark}
\label{rmk:qsdim}
{\rm Note that if $\eta\in{\cal X}_{\Delta^+}$ and $\ell_2=2$, then $\ell_1=2$. Indeed, a $1\times2$ dimer does not belong to ${\cal X}_U$ and therefore is not in ${\cal X}_{\Delta^+}$.}\hfill$\spadesuit$
\end{remark}
     
Recall that $T_A=\ee^{A\beta}$ for $A\in\{0,U,D,S,\Delta^+\}$. Note that we have used the index $\Delta^+$ to define the set ${\cal X}_{\Delta^+}$, despite the fact that it explicitly depends on the quantity $\Delta-\alpha$. This is needed to provide an upper bound for the return times in ${\cal X}_{\Delta^+}$, namely, the recurrence property stated in the following proposition, which uses the usual shift operator $\vartheta_s$, $s\geq0$, defined by $\vartheta_s(X)=X(s+\cdot)$, so that $s+\tau_{{\cal X}_A}\circ\vartheta_s= \min\{t\geq s\colon\, X(t)\in{\cal X}_A\}$.

\begin{proposition}
\label{prop:rec}
For all $A\in\{0,U,D,S,\Delta^+\}$, all $\delta>0$ and any stopping time $\tau$,
$$
P_{\mu_{\cal R}}\big(\tau_{{\cal X}_A}\circ\vartheta_\tau \geq T_A \ee^{\delta\beta}, \
\tau+\tau_{{\cal X}_A}\circ\vartheta_\tau\leq T^{\star}\big) = \SES(\beta).
 $$	
 \end{proposition}
 
\noindent
To prove Proposition \ref{prop:rec}, we need the following lemma, whose proof is postponed until after the proof of Proposition \ref{prop:rec}.
 
\begin{lemma}
\label{lmm:sleeping}
Let $t\geq0$ be the time at which an active particle $p$ joins a cluster ${\cal C}$ with at most two particles, and let $t^*=t+(t_1\wedge t_2)\circ\vartheta_t$, where $t_1$ (respectively, $t_2$) is the first time when the cluster ${\cal C}$ contains at least four particles (respectively, does not contain particle $p$ anymore). The probability that particle $p$ falls asleep during the time interval $[t,t^*]$ is $\SES(\beta)$.
 \end{lemma}
     
\noindent {\it Proof of Proposition \ref{prop:rec}.}
Let $A\in\{0,U,D,S,\Delta^+\}$. Divide the time interval $[0,T_A\ee^{\delta\beta}]$ into $\ee^{\frac{3}{4}\delta\beta}$ intervals $I_j$ of length $T_A\ee^{\frac{\delta}{4}\beta}$. We have 
\begin{equation}
\label{eq:up1}
\begin{array}{ll}
\displaystyle\sup_{\eta\in{\cal X}^*} P_\eta(\tau_{{\cal X}_A}\wedge 
\tau_{{\cal X}_\beta\setminus{\cal X}^*}>T_A \ee^{\delta\beta})
&\leq\displaystyle\prod_{1 \leq j<\ee^{\frac{3}{4}\delta\beta}}
\sup_{\eta\in{\cal X}^*} P_\eta(\tau_{{\cal X}_A}, \tau_{{\cal X}_\beta\setminus{\cal X}^*}\notin I_j) \\
&=\displaystyle \Big(1-\inf_{\eta\in{\cal X}^*} P_\eta(\tau_{{\cal X}_A}\wedge \tau_{{\cal X}_\beta\setminus{\cal X}^*}
\leq T_A \ee^{\frac{\delta}{4}\beta})\Big)^{\ee^{\frac{3}{4}\delta\beta}},
\end{array}
\end{equation}
where we use the strong Markov property for the stopping time $\tau_{{\cal X}_A}$. By Proposition \ref{prp:ambito}, it suffices to prove that 
\begin{equation}
\label{eq:rec}
\inf_{\eta\in{\cal X}^*}P_\eta(\tau_{{\cal X}_A}\wedge \tau_{{\cal X}_\beta\setminus{\cal X}^*}
\leq T_A \ee^{\frac{\delta}{4}\beta})\geq \ee^{-\frac{\delta}{4}\beta}.
\end{equation}
In other words, for each $\eta$ in ${\cal X}^*$ we have to build a dynamical event on time scale $T_A\ee^{(\delta/4)\beta}$ and with probability at least $\ee^{-(\delta/4)\beta}$ such that the final configuration is in ${\cal X}_A$, provided our system does not exit ${\cal X}^*$. This is a standard estimate for metastable systems at low temperature, which has been carried out in full detail for a simplified version of our model \cite{dHOS00a}. Here we indicate the differences with respect to the earlier work.

To build ${\cal X}_A$, we used $\Lambda(t)=\cup_{1 \leq i \leq k(t)}\bar\Lambda_i(t)$, the connected components of which form our box collection $\bar\Lambda(t)$. For $A\leq S$ we use another box collection $\bar\Lambda'(t)$ such that $\Lambda'(t)=\cup_{1 \leq i \leq k'(t)}\bar\Lambda_i'(t)$, for which $\bar\Lambda_i'(t)$, $1 \leq i \leq k'(t)$, are the connected components of $\Lambda'(t)$, and such that $\Lambda(t)\subset\Lambda'(t)$ for all $t$. As a consequence, the associated ${\cal X}'_A$ is contained in ${\cal X}_A$. We need to consider this new collection $\bar\Lambda'(t)$ in order to avoid the creation of new local boxes when some particle outside of $\bar\Lambda(t)$ falls asleep before time $T_Ae^{\delta\beta}$. The construction of $\bar\Lambda'(t)$ is analogous to that in Definition \ref{def:boxes}, but now without removing the collection $\bar\Lambda^*(t)$, $t\geq0$, i.e., the boxes without sleeping particles, and by redefining the boxes at time $t$ when at least one of the conditions B1', B3 and B4 is violated by $\bar\Lambda'(t^-)$, with B1' being defined as B1 but referring to clusterised particles instead of sleeping particles. In this way the new collection satisfies conditions B1', B3 and B4 for any $t\geq0$.
     
\medskip\noindent
$\bullet$ \underline{Case $A=0$:}
Consider $\Lambda'(0)$, which contains all the clusterised particles at time 0 and is such that all particles outside $\Lambda'(0)$ are initially free. Let $\tau_c$ be the first time when two of these free particles collide, or one of them enters $\Lambda'(0)$. By \cite[Proposition 3.1.1 and Theorem 1]{G09}, the probability that $\tau_c>\ee^{\delta\beta}$ when starting from a configuration in ${\cal X}^*$ is larger than $\ee^{-\delta\beta}$ for $\beta$ large enough. Conditionally on this event, and as long as no clusterised particle in $\Lambda'(0)$ is at distance one from the internal border of $\Lambda'(0)$, the dynamics inside $\Lambda'(0)$ is independent from that outside. By construction, there are at least two particles in each $\bar\Lambda_i(0)$. By grouping them we can perform within time $\ee^{\delta\beta}$ the 0-reduction in $\bar\Lambda'(0)$ with a non-exponentially small probability, as in \cite{dHOS00a}: the only difference is that we are not working with a box $\bar\Lambda(0)$ of a bounded size but of a slowly growing size. However, we can deal with this box as in \cite[Appendix A]{GdHNOS09}.
      
\medskip\noindent
$\bullet$ \underline{Case $A=U$:}
We can proceed in the same way, except for the fact that to reach a $U$-irreducible configuration we may have to move some particle outside $\Lambda'(0)$. This happens for example when starting with a protuberance on a quasi-square (see \cite{dHOS00a}). We set
$$
\tilde\Lambda=\{\zeta\in\Lambda_\beta\colon\,
\exists\, x\in\Lambda'(0), ||\zeta-x||\leq2\}
$$
and to build the reduction event we ask that each free particle in $\Lambda_\beta\setminus\tilde\Lambda$ remains free, without entering $\tilde\Lambda$, for a time $T_U\ee^{\delta\beta}$. We also ask that each free particle in $\tilde\Lambda\setminus\Lambda'(0)$ moves to $\Lambda_\beta\setminus\tilde\Lambda$ without visiting $\Lambda'(0)$ or forming a new cluster. Conditionally on this event, the local configuration in $\Lambda'(0)$ can be $U$-reduced, with respect to $\Lambda'(0)$, within time $T_U\ee^{\delta\beta}$ with a non-exponentially small probability. Since no more than $\lambda(\beta)$ particles can leave $\Lambda'(0)$ on the described event,
\cite[Theorem 1]{G09} gives the desired bound.
     
\medskip\noindent
$\bullet$ \underline{Case $A=D$:}
We can simply use the same event as in the case $A=U$, but built on the slightly longer time scale $T_D\ee^{\delta\beta}$: all clusterised particles in our $U$-reduced configuration fall asleep.
     
\medskip\noindent
$\bullet$ \underline{Case $A=S$:}
We use again the same event, but built on the longer time scale $T_S\ee^{\delta\beta}$. By the coupling of QRWs and IRWs, the probability that a given quadruple of free particles at time $T_D\ee^{\delta\beta}$ has a diameter at most $\ee^{S\beta/2}$ at time $T_S\ee^{\delta\beta}$ is smaller than $\ee^{-\delta\beta/2}$, as a consequence of the spread-out property of the simple random walk given by the difference between the position of two of the four particles. By the non-superdiffusivity property, assuming that our process is in ${\cal X}_D\subset{\cal X}^*$ at time $T_D\ee^{\delta\beta}$, we only have to consider $\lambda(\beta)$ quadruples to check that by time $T_S\ee^{\delta\beta}$ we have reached ${\cal X}_S$: the probability that a particle exits $\Lambda'(0)$ within time $T_S\ee^{\delta\beta}\ll \ee^{2U\beta}$, and before the entrance of a new particle in $\Lambda'(0)$, is exponentially small in $\beta$. Consequently,
$$
P_\eta(X(T_S\ee^{\delta\beta})\in{\cal X}_S \hbox{ or } \tau_{{\cal X}^*}<T_S\ee^{\delta\beta})
\geq \ee^{-\frac{\delta}{4}\beta}-\ee^{-(2U-S-\delta)\beta}
- \lambda(\beta)\ee^{-\frac{\delta}{2}\beta}\geq \ee^{-\frac{\delta}{3}\beta}
$$
for $\beta$ large enough.
      
\medskip\noindent
$\bullet$ \underline{Case $A=\Delta+\alpha$:}
We have shown that within time $T_S \ee^{\delta\beta}$ the dynamics reaches ${\cal X}_S$ or exits ${\cal X}^*$ with probability $1-\SES(\beta)$. To build the event, we let particles enter the local boxes in order to form quasi-squares, before emptying the annulus between $[\Lambda(t),\Delta-\alpha]$ and $\Lambda(t)$ for a large enough $t$, while going to ${\cal X}_S$, all without the occurrence of a box creation. To control this event, we provide an upper bound for the probability that a box creation occurs after reaching ${\cal X}_S$. A box creation can occur with a non-$\SES(\beta)$ probability only when four particles are in a box of volume $\ee^{(D+\delta)\beta}$ at the same time $t<T_{\Delta^+}\ee^{(\delta/4)\beta}$. Indeed, starting from a cluster consisting of two or three particles only, the probability that a particle falls asleep is $\SES(\beta)$ by Lemma \ref{lmm:sleeping}. We estimate from above the probability to have four particles in a box of volume $\ee^{(D+\delta)\beta}$ at the same time $t<T_{\Delta^+}\ee^{(\delta/4)\beta}$. For $S<A'<\Delta^+$ we can estimate the probability that a given quadruple of particles with diameter $\ee^{A'\beta/2}$ arrives in a box of volume $\ee^{(D+\delta)\beta}$ within time $T_{\Delta^+}\ee^{\delta\beta}$, as follows. Divide the time interval $[\ee^{A'\beta},T_{\Delta^+}\ee^{\delta\beta}]$ into intervals of length $\ee^{D\beta}$, and divide at each initial time $i\ee^{D\beta}$ of such a time interval the volume $i\ee^{(D+\delta)\beta}$ centered at one of the particles into boxes of volume $\ee^{(D+\delta)\beta}$. Then, by the non-superdiffusivity property and the spread-out property of the QRWs, we get that the required probability is at most
$$
\ee^{\delta\beta}\displaystyle\sum_{\ee^{A'\beta}\leq i\ee^{D\beta} \leq \ee^{(\Delta+\alpha+\delta)\beta}}
\Bigg(\dfrac{i\ee^{(D+\delta)\beta}}{\ee^{(D+\delta)\beta}}\Bigg)
\Bigg(\dfrac{\ee^{(D+\delta)\beta}\ee^{\delta\beta}}{i\ee^{(D+\delta)\beta}}\Bigg)^4
\leq \ee^{2(D-A')\beta+O(\delta)\beta}.
$$
When $X(T_S\ee^{\delta\beta})\in{\cal X}^*$, this implies that the probability to have four particles in a box of volume $\ee^{(D+\delta)\beta}$ within time $T_{\Delta^+}\ee^{\delta\beta}$ is at most $\ee^{(A'+2D-4\Delta+\theta+4\alpha)\beta}\ee^{O(\delta)\beta}$, which is an increasing function of $A'$. Since $A'<\Delta^+$, we have that the required probability is less than $\ee^{\theta\beta}\ee^{(2D-3\Delta)\beta}\ee^{(4\alpha+O(\delta))\beta}$, which implies that 
\begin{equation}
\label{eq:newbox}
P\big(\hbox{a box creation occurs within time } T_{\Delta^+}\ee^{\delta\beta}\big)
\leq \ee^{-(3\Delta - 2U -\theta -2d)\beta}.
\end{equation}
This is exponentially small, so that we can work with a constant number of boxes.
     
We can now proceed as in \cite{dHOS00a} to bring in particles from the gas in order to build quasi-squares. One additional difficulty and one additional simplification occurs. While in \cite{dHOS00a} the local box was fixed, which makes motion of large droplets inside impossible, here our local boxes move with the droplets, so that there are no lacunary configuration issues. However, we cannot use the simple random walk estimates to give lower bounds on the probability of bringing particles from the gas into the local boxes: these have to be replaced by the strong lower bounds of \cite[Theorem 3.3.1]{GdHNOS09}. Once we have obtained quasi-squares only in $\Lambda(\tau)$ for some stopping time $\tau\leq T_{\Delta^+}\ee^{\delta\beta/2}$, we can build the same event that was used to deal with $A=S$ to empty the annulus $[\Lambda(\tau),\Delta-\alpha]\setminus\Lambda(\tau)$ without moving the boxes anymore while going back to ${\cal X}_S$.
\qed
     
\medskip
\noindent {\it Proof of Lemma \ref{lmm:sleeping}.}
Suppose that two active particles join together. Divide the time interval $[0,\ee^{D\beta}]$ into $\ee^{\frac{3}{4}d\beta}$ intervals of length $\ee^{(U+\frac{d}{4})\beta}$. We have
$$
P\big(\hbox{a particle is detached within time } \ee^{(U+\frac{d}{4})\beta}\big)
\geq \ee^{-\frac{d}{4}\beta}
$$
and so by the Markov property the probability to have a particle falling asleep is at most $(1-\ee^{-\frac{d}{4}\beta})^{\ee^{\frac{3}{4}d\beta}}=\SES(\beta)$. The case in which three active particles join together can be treated similarly.
\qed

    
\section{Proof of theorems}
\label{sec:proofs}
   
Section~\ref{sec:lemmas} lists three key propositions that provide bounds on the probability of transitions between configurations consisting of a single quasi-square and free particles. The proofs of these propositions are deferred to Section~\ref{sec:lemmone}. Sections~\ref{sec:proof1}--\ref{sec:proof4} use the propositions to prove Theorems~\ref{thm:qsrec}--\ref{thm:cambioatipico}, respectively.

The pure gas state is defined as
\begin{equation}
\label{def:vuoto}
{\cal X}_E:=\{\eta\in{\cal X}_{\Delta^+}\colon\, \eta \hbox{ has no quasi-square}\}.
\end{equation}
     
     
\subsection{Key propositions: Propositions \ref{prp:lbtipico}--\ref{prp:uppbound}}
\label{sec:lemmas}
    
Recall the definition of $\pi(\eta_0)\in\bar{\cal X}_{\Delta}$, $\eta_0\in{\cal X}_{\Delta^+}$, given in Section \ref{sec:mainres}. Denote by $(\ell_1,\ell_2) \in QS$ with $\ell_1 \geq 2$ the dimensions of the smallest quasi-square, if any, otherwise set $\ell_1=\ell_2=0$. Define the projections $\pi',\pi''\in\bar{\cal X}_{\Delta}$ similarly to the projections $\pi_i',\pi_i''$ defined in Section \ref{sec:mainres}. 

We start by giving a lower bound for the probability that the dynamics, starting from $\eta_0\in{\cal X}_{\Delta^+}$, has a projection that is distinct from $\pi(\eta_0)$ at time $\bar\tau_1$ without exiting the environment ${\cal X}^*$.
     
\begin{proposition}
\label{prp:lbtipico}
Assume that $\Delta<\Theta\leq\theta$. If $\eta_0\in{\cal X}_{\Delta^+}$, then for any $\delta>0$,
$$
P_{\eta_0}\left(\begin{array}{ll} 
\pi(X(\bar\tau_1))\neq\pi(\eta_0) \hbox{ or a coalescence}\\
\hbox{occurs before } \bar\tau_1 \hbox{ or } \bar\tau_1>\tau_{{\cal X}_\beta\setminus{\cal X^*}}
\end{array}
\right)
\geq \ee^{-[r(\ell_1,\ell_2)-\Delta+O(\alpha,d,\delta)]\beta}.
$$
\end{proposition}
     
\noindent
The proof of Proposition \ref{prp:lbtipico} is given in Section \ref{sub:prp1}.
     
We next give a lower bound on the probability that the dynamics, starting from $\eta_0\in{\cal X}_{\Delta^+}$ consisting of a single subcritical quasi-square, at time $\bar\tau_1$ reaches a configuration $X(\bar\tau_1)$ such that $\pi(X(\bar\tau_1))=\pi''$ without exiting the environment ${\cal X}^*$ and no box creation occurs before $\bar\tau_1$.
     
\begin{proposition}\
\label{prp:subgrow}
Assume that $\Delta<\Theta\leq\theta$. If $\eta_0\in{\cal X}_{\Delta^+}$ consists of a single $\ell_1\times \ell_2$ quasi-square with $2\leq \ell_1<\ell_c$, then for any $\delta>0$,
$$
P_{\eta_0}\left( 
\begin{array}{ll}
\pi(X(\bar\tau_1))=\pi'' \hbox{ and no box creation} \\
\hbox{occurs before } \bar\tau_1, \hbox{ or } \bar\tau_1>\tau_{{\cal X}_\beta\setminus{\cal X^*}}
\end{array}
\right)
\geq \ee^{-[\Delta-U+O(\alpha,d,\delta)]\beta}.
$$
\end{proposition}

\noindent
The proof of Proposition \ref{prp:subgrow} is given in Section \ref{sub:prp2}.
     
We finally provide upper bounds on the probability that typical and atypical transitions occur.

\begin{proposition}
\label{prp:uppbound}
Assume that $\Delta<\Theta\leq\theta$.\\
(1) If $\eta_0\in{\cal X}_{\Delta^+}$, then 
\begin{equation}
\label{eq:exit1}
\displaystyle\limsup_{\beta\rightarrow\infty}
\sup_{\pi(\eta_0)}\frac{1}{\beta}\log P_{\eta_0}\left(
\begin{array}{ll}
\pi(X(\bar\tau_1))\neq\pi(\eta_0)\hbox{ and a} \\
\hbox{coalescence does not occur} \\
\hbox{before }\bar\tau_1,
\hbox{ or } \bar\tau_1>\tau_{{\cal X}_\beta\setminus{\cal X^*}}
\end{array}
\right)
\leq -[r(\ell_1,\ell_2)-\Delta-O(\alpha,d)].
\end{equation}
(2) If $\eta_0\in{\cal X}_{\Delta^+}\setminus{\cal X}_E$, then
\begin{equation}
\label{eq:exit2}
\displaystyle\limsup_{\beta\rightarrow\infty} \sup_{\pi(\eta_0)}\frac{1}{\beta}\log P_{\eta_0}\left(
\begin{array}{ll}
\pi(X(\bar\tau_1))\notin\{\pi(\eta_0)\}\cup\pi' \hbox{ and}\\
\hbox{a coalescence does not occur} \\
\hbox{before }\bar\tau_1,
\hbox{ or } \bar\tau_1>\tau_{{\cal X}_\beta\setminus{\cal X^*}}
\end{array}
\right)
< -[r(\ell_1,\ell_2)-\Delta-O(\alpha,d)].
\end{equation}
\end{proposition}     
     
\noindent
The proof of Proposition \ref{prp:uppbound} is given in Section \ref{sub:prp3}.

     
\subsection{Proof of Theorem \ref{thm:qsrec}}
\label{sec:proof1}

Fix $\delta>0$. From Proposition \ref{prop:rec} we deduce that the event $\{\bar\tau_0\geq \ee^{(\Delta+\alpha+\delta)\beta},\bar\tau_0\leq T^{\star}\}$ has probability $\SES(\beta)$. Consider any $i\in\N_0$ such that $\bar\tau_{i+1}\leq T^{\star}$. The event $\{\bar\tau_{i+1}-\bar\tau_i>\ee^{(\Delta+\alpha+\delta)\beta}\}$ has probability $\SES(\beta)$. Indeed, this event would imply that either $\bar\sigma_{i+1}$ or $\bar\tau_{i+1}$ exceed $T_{\Delta^+} \ee^{\delta\beta}$, and both have probability $\SES(\beta)$. Indeed, in the former case, we have to control the probability that none of the particles inside the volume $[\bar\Lambda,\Delta+\alpha]$ enters $\bar\Lambda$ within a time $T_{\Delta^+}\ee^{\delta\beta}$. These particles are at least $\ee^{\frac{1}{2}\alpha\beta}$ in number, since the dynamics is in ${\cal X^*}$ because of the condition $\bar\tau_{i+1}\leq T^{\star}$. Hence this probability is $\SES(\beta)$ by the strong lower bounds associated with the spread-out property of QRWs (see \cite[Theorem 3.3.1]{GdHNOS09}). In the latter case, we conclude by using Proposition \ref{prop:rec}. Also the event $\{\bar\tau_{i+1}-\bar\tau_i<\ee^{(\Delta-\alpha-\delta)\beta}\}$ has probability $\SES(\beta)$. Indeed, this event would imply that $\bar\sigma_{i+1}$ is at most $\ee^{(\Delta-\alpha-\delta)\beta}$. This event has probability $\SES(\beta)$ by the non-superdiffusivity property if the configuration at time $\bar\tau_i$ is in ${\cal X}_{\Delta^+}\setminus{\cal X}_E$, otherwise it has probability zero by the condition $\bar\sigma_{i+1}=\ee^{\Delta\beta}$.
\qed


\subsection{Proof of Theorem \ref{thm:res}}
\label{sec:proof2}
     
For $i\in N_0$, define
\begin{equation*}
K_i = \min\big\{k\in\N\colon\, \pi(X(\bar\tau_{i+k})) \neq \pi(X(\bar\tau_i))\big\}.
\end{equation*}
Up to coalescence and exit from ${\cal X}^*$, Proposition \ref{prp:lbtipico} and the first part of Proposition \ref{prp:uppbound} show that $K_i$ dominates and is dominated by a geometric random variable with success probability of order $\ee^{-(r(\ell_1,\ell_2)-\Delta)\beta}$. Together with Theorem \ref{thm:qsrec}, which gives uniform lower and upper bounds on the return times $\bar\tau_{j+1}-\bar\tau_j$, $j\in\N_0$, this proves Theorem \ref{thm:res}: the $\SES$ error in \eqref{eq1} is related to an anomalously large realisation of a geometric random variable, while an anomalously small realisation leads to an error that is only exponentially small in \eqref{eq2}. 
\qed

     
\subsection{Proof of Theorem \ref{thm:cambiotipico}}
\label{sec:proof3}

Proposition \ref{prp:lbtipico} and the second part of Proposition \ref{prp:uppbound} prove Theorem \ref{thm:cambiotipico} for any $i\in\N_0$ such that $X(\bar\tau_i)\in{\cal X}_{\Delta^+}\setminus{\cal X}_E$: these propositions provide the necessary lower and upper bounds on the denominator and numerator of the conditional probability. Otherwise, if $X(\bar\tau_i)\in{\cal X}_E$, then instead of using Proposition \ref{prp:uppbound} we conclude by using Remark \ref{rmk:qsdim} and arguing as in \eqref{eq:newbox} to show that the probability to have more than 4 particles in a box with volume of order $\ee^{D\beta}$ is exponentially smaller than the bound obtained in Proposition \ref{prp:lbtipico}.
\qed
     

\subsection{Proof of Theorem \ref{thm:cambioatipico}}
\label{sec:proof4}

Proposition \ref{prp:subgrow} and the first part of Proposition \ref{prp:uppbound} prove Theorem \ref{thm:cambioatipico}: they give the necessary upper and lower bounds on the denominator and numerator of the conditional probability.
\qed


\section{Proof of propositions}
\label{sec:lemmone}
  
In Section~\ref{sub:prp1}--\ref{sub:prp3} we prove Propositions~\ref{prp:lbtipico}--\ref{prp:uppbound}, respectively. The proof of Proposition~\ref{prp:uppbound} relies on three additional lemmas, whose proof is deferred to Section~\ref{sub:lmmgen}.      
 

\subsection{Proof of Proposition \ref{prp:lbtipico}}
\label{sub:prp1}

Fix $\delta>0$. Since 
$$
P_{\eta_0}(\pi(X(\bar\tau_1))\neq\pi(\eta_0)) \geq P_{\eta_0}(\pi(X(\bar\tau_1))=\pi'),
$$
we need to bound from below the probability that a typical transition of the dynamics on ${\cal X}_{\Delta^+}$ occurs.

\medskip\noindent
{\bf 1.}    
We start by considering the supercritical case $\ell_1>\ell_c$. Since in this case $r(\ell_1,\ell_2)=2\Delta-U$, it suffices to exhibit a mechanism to grow within time $T_{\Delta^+}\ee^{\delta\beta}$ with probability at least $\ee^{-(\Delta-U+O(\alpha,d,\delta))\beta}$. Within time $T_{\Delta^+}\ee^{\delta\beta/2}$ bring two particles from the gas inside one of the volumes $[\bar\Lambda_i, D-\delta]$. Attach the two particles in time $\ee^{(D+\delta)\beta}$. Complete the quasi-square with particles from the gas. Let $\tau$ be the first time at which there are two active particles inside one of these volumes. On the time scale we are interested in, particles can arrive inside the box $\Lambda$, but before time $\tau$ only one can be active. Thus, by using the recurrence property to ${\cal X}_U$, we know that this active particle can attach itslef to the quasi-square inside $\Lambda$, but it does not feel asleep with probability $1-\SES$. Moreover, via the interaction with this active particle the cluster can move, but in such a way that $\Lambda(t)\subset[\Lambda(0), D-\delta]$ for any $t$. Indeed, any redefinition of the local box, implied by the movement of the cluster, is related to a free particle that moves in $\Lambda$. We show that the probability that the number of these box special times exceeds $\ee^{O(\alpha,\delta)\beta}$ is $\SES$.

Since the dynamics belongs to the environment ${\cal X}^*$, by the non-superdiffusivity property of the QRWs we know that at most $\ee^{3\alpha\beta/2}$ particles can interact with $\Lambda$ within time $T_{\Delta^+}\ee^{\delta\beta}$. Each particle no longer visits $\Lambda$ after each box special time associated with it with a probability lat least $1/(\log \exp (\Delta+O(\alpha,\delta))\beta)$.Thus,
$$
P(\hbox{there are more than } \ee^{O(\alpha,\delta)\beta} \hbox{ visits in } \Lambda)
\leq \Biggl(1-\frac{1}{(\Delta+O(\alpha,\delta))\beta}\Biggr)^{\ee^{O(\alpha,\delta)\beta}}
= \SES(\beta).
$$
Thus, up to an event of probability $\SES$, we deal with the fixed target volume $[\Lambda(0), D-\delta]$. In addition, we deal with a constant number of local boxes, since we can control the probability that a box creation occurs  within time $T_{\Delta^+}\ee^{\delta\beta}$ via the estimate derived in the proof of Proposition \ref{prop:rec}. To check that the resulting order of probability is correct, we proceed as follows. Divide the time interval $[0,T_{\Delta^+}\ee^{\delta\beta}]$ into intervals $[t_i,t_i+\ee^{(D+\delta)\beta}]$, with $1 \leq i < \ee^{(\Delta+\alpha-D)\beta}$. By considering $T_i=i \ee^{(D+\delta)\beta}$, and using the non-superdiffusivity property and the lower bound associated with the spread-out property of the QRWs (see \cite[Theorem 3.2.5(ii)]{GdHNOS09}), we get 
\begin{equation}
P(\tau<\ee^{(\Delta+\alpha+\delta)\beta}) \geq
\sum_{\ee^{(\Delta+\alpha-\delta)\beta}\leq i\ee^{(D+\delta)\beta}\leq \ee^{(\Delta+\alpha+\delta)\beta}}
\Biggl(\frac{\ee^{(D-\delta)\beta}}{i\ee^{(D+\delta)\beta}\ee^{\delta\beta}}\Biggr)^2
\geq \ee^{-[\Delta-U+O(\alpha,d,\delta)]\beta}.
\end{equation}
Let these two active particles inside $[\Lambda(0), D-\delta]$ at time $\tau$ attach themselves to the quasi-square. By using the non-superdiffusivity property and the stronger, higher resolution, lower bounds associated with the spread-out property of the QRWs, we get that this probability is at least $\ee^{-O(\delta)\beta}$. Arguing in the same way, we obtain an analogous lower bound for the probability to complete the quasi-square with particles from the gas in time $T_{\Delta^+}\ee^{\delta\beta/2}$. We conclude by using the strong Markov property at times $\tau$ and those corresponding to each attachment of the particles to the cluster in $\Lambda$.
    
\medskip\noindent
{\bf 2.} 
Next consider the subcritical case $\ell_1<\ell_c$. We start with $\eta_0\in{\cal X}_E$. Since in this case $r(\ell_1,\ell_2)=4\Delta-2U-\theta$, it suffices to exhibit a mechanism to create a $2\times2$ droplet within time $T_{\Delta^+}\ee^{\delta\beta}$ with probability at least  $\ee^{-(3\Delta-2U-\theta+O(\alpha,d,\delta))\beta}$. Within time $T_{\Delta^+}\ee^{\delta\beta/2}$ bring four particles from the gas inside a box of volume $\ee^{(D-\delta)\beta}$. Attach two of these particles within time $\ee^{(D+\delta)\beta}$. Move the other two particles at a finite distance from the dimer within time $\ee^{(U-\delta/2)\beta}$. Given a fixed site $x\in\Lambda_\beta$, let $\tau$ be the first time at which there are four active particles in a box of volume $\ee^{(D-\delta)\beta}$ centered at $x$. To check that the resulting order of probability is correct, we proceed as follows. Divide the time interval $[0,T_{\Delta^+}\ee^{\delta\beta}]$ into intervals $[t_i,t_i+\ee^{(D+\delta)\beta}]$ with $1 \leq i<\ee^{(\Delta+\alpha-D)\beta}$. By considering $T_i=i \ee^{(D+\delta)\beta}$, and using the non-superdiffusivity property and the lower bound associated with the spread-out property of the QRWs (see \cite[Theorem 3.2.5(ii)]{GdHNOS09}), we get 
\begin{equation}
P(\tau<\ee^{(\Delta+\alpha+\delta)\beta}) \geq
\sum_{\ee^{(\Delta+\alpha-\delta)\beta}\leq i\ee^{(D+\delta)\beta}\leq \ee^{(\Delta+\alpha+\delta)\beta}}
\Biggl(\frac{\ee^{(D-\delta)\beta}}{i\ee^{(D+\delta)\beta}\ee^{\delta\beta}}\Biggr)^4
 \geq \ee^{-3[\Delta-U+O(\alpha,d,\delta)]\beta}.
\end{equation}
Let $\sigma$ be the first time at which two among these four active particles form a dimer for the first time at a finite distance from the site $x$. By using the non-superdiffusivity property and the stronger lower bounds associated with the spread-out property of the QRWs,
we get
\begin{equation}
\label{eq:distfinita}
P(\sigma<\ee^{(D+\delta)\beta})\geq	
\int_{\ee^{(D-\delta)\beta}}^{\ee^{(D+\delta)\beta}}\Biggl(\frac{1}{t\ee^{\delta\beta}}\Biggr)^2 dt
\geq \ee^{-[U+O(\delta,d)]\beta}.
\end{equation}
Now let the other two active particles attach themselves to the dimer formed at time $\sigma$ within time $\ee^{(U-\delta/2)\beta}$, so that the dimer is still present with probability $1-\SES$. Arguing as before, we deduce that this probability is at least $\ee^{-O(\delta)\beta}$. Finally we observe that these creations of a first cluster of sleeping particles around a given site $x$ are disjoint events up to an event with negligible probability, the probability of which is controlled as in \eqref{eq:newbox}. By summing over all the sites $x\in\Lambda_\beta$ and applying the strong Markov property at the times $\tau$, $\sigma$ and those corresponding to the attachment of the third particle to the dimer, we get the claim.

\medskip\noindent
{\bf 3.}    
Finally, consider the case $\ell_1\geq 2$. It suffices to exhibit a mechanism to shrink within time $\ee^{(\Delta-\alpha+\delta)\beta}$ with a probability at least $\ee^{-(r(\ell_1,\ell_2)-\Delta+O(\alpha,d,\delta))\beta}$. The mechanism to shrink is the following: detach a row of $\ell_1$ particles and bring each particle outside the volume $[\Lambda,\Delta-\alpha]$ within time $\ee^{(\Delta-\alpha/2)\beta}$. Note that at time $t=0$ there are at most $\lambda(\beta)/4$ particles inside the volume $[\Lambda,\Delta-\alpha/4]$ because the dynamics starts in ${\cal X}^*$. Thus, by the non-superdiffusivity property it follows that, up to an event of probability $\SES$, these are the only particles that can enter $[\Lambda,\Delta-\alpha]$ within time $\ee^{(\Delta-\alpha/2)\beta}$. We can therefore argue as in the proof of Proposition \ref{prop:rec} for $A=U$ with the following differences. For the first $\ell_1-1$ particles we obtain that the probability for each one of them to be detached is at least $\ee^{(-(2U-\Delta)-O(\alpha,\delta))\beta}$. Indeed, divide the time interval $[0,\ee^{(\Delta-\alpha/2)\beta}]$ into intervals $S_i$ of length $\ee^{D\beta}$, with $1 \leq i<\ee^{(\Delta-D-\alpha/2)\beta}$. Then the probability to detach one of these particles is at least
$$
\ee^{-\delta\beta}\sum_{1 \leq i < \ee^{(\Delta-D-\alpha/2)\beta}}
P(\hbox{there is a move of cost } 2U 
\hbox{ between } i\ee^{D\beta} \hbox{ and } (i+1)\ee^{D\beta})
\geq \ee^{[-(2U-\Delta)-O(\alpha,\delta)]\beta}.
$$
After applying the strong Markov property at each of the detaching times and observing that the probability of detaching the last particle at cost $U$ within time $\ee^{(\Delta-\alpha/2)\beta}$ is at least $\ee^{-O(\delta)\beta}$, and also the probability that no particle is inside the annulus $[\Lambda,\Delta-\alpha]\setminus\Lambda$ because of the lower bounds associated with the spread-out property of the QRWs.
\qed
   

\subsection{Proof of Proposition \ref{prp:subgrow}}
\label{sub:prp2}

Fix $\delta>0$. Since in this case $\pi''=(\ell_2\times(\ell_1+1))$, in order to get the claim it suffices to exhibit a mechanism to grow with a probability at least $\ee^{-[\Delta-U+O(\alpha,d,\delta)]\beta}$. The mechanism is the same as for the supercritical case used in the proof of Proposition \ref{prp:lbtipico}. Since now we are interested in not having a box creation before time $\bar\tau_1$, we obtain the desired lower bound after using the estimate in \eqref{eq:newbox}.
\qed
    
     
\subsection{Proof of Proposition \ref{prp:uppbound}}
\label{sub:prp3}
	
Since we need to control all the possible mechanisms to grow and shrink, the proof of Proposition \ref{prp:uppbound} is much more involved than the proofs of Propositions~\ref{prp:lbtipico}--\ref{prp:subgrow}, and is organised into steps. We start by considering the case $\eta_0\in{\cal X}_E$. We assume that there is a single finite box for the starting configuration $\eta_0$, namely, $\eta_0\in{\cal X}_{\Delta^+}$ consisting of a single quasi-square of size $\ell_1 \times \ell_2$. Abusing notation, we refer to the current box $\Lambda=\bar\Lambda_0$ as $\bar\Lambda$ instead of $\bar\Lambda_0$. This is needed in order to make the proof clearer. We will see later how to derive the statement for general boxes. 

The key steps in the proof are the following:
\begin{description}
\item[Step 1:] Introduce coloration and permutation rules (Section \ref{subsec:color}).
\item[Step 2:] Consider the case $\h_0\in{\cal X}_E$ (Section \ref{sub:step2}).
\item[Step 3:] Consider the case $\h_0\in{\cal X}_{\Delta^+} \setminus {\cal X}_E$ and $\ell_2 \geq 3$ (Section \ref{sub:step3}).
\item[Step 4:] Consider the case $\h_0\in{\cal X}_{\Delta^+} \setminus {\cal X}_E$ and $\ell_2 = 2$ (Section \ref{sub:step4}).
\item[Step 5:] Derive the statement for a general collection of finite boxes $\bar\Lambda = (\bar\Lambda_i)_{i \in I}$ (Section \ref{sub:step5}).
\end{description}

\noindent
In step 1 we introduce the notion of \emph{colours for particles and their permutation rules}, which are needed in steps 2--5. In each of steps 2--4 we state a key lemma and explain how to derive the statement of interest from it. The proofs of the lemmas are deferred to Section \ref{sub:lmmgen}, which is the technical core of the present paper. 

Recall that we are considering the case in which there is a single finite local box $\bar\Lambda$. We call ${\cal I}(n)$ the set of configurations $\eta$ such that $\bar\eta$ is of size $|\bar\eta|=n$ and is the solution of the associated isoperimetric problem. We use the notation ${\cal I}(n)^{fp}$ to indicate the presence of a free particle in $\bar\Lambda$. Moreover, we call ${\cal I}(0)$ the set of configurations for which there is no local box $\bar\Lambda$. We introduce the sequence $(\tau_k)_{k\in\N_0}$ of return times in ${\cal X}_{D}$ after seeing an active particle in $\bar\Lambda$ as follows. Put $\tau_0=0$ and, for $i \in \N_0$, define
\begin{equation}
\label{activereturn}
\sigma_{i + 1} = \inf\left\{t > \tau_i\colon\, \hbox{there is an active particle inside $\bar\Lambda(t)$ at time $t$} \right\}
\end{equation}
and
\begin{equation}
\label{deltareturn2}
\tau_{i + 1} = \inf\left\{t > \sigma_{i + 1}\colon\, X(t) \in {\cal X}_{D} \right\}.
\end{equation}
Note the difference between \eqref{activereturn}-\eqref{deltareturn2} and \eqref{activereturn1}-\eqref{deltareturn1}. Let $\varphi^k$ be the finite-time Markov chain $\varphi^k = (X(\tau_i))_{0 \leq i \leq k}$, and put
$$
n = \max\left\{k \geq 0\colon\,\tau_k < T_{\Delta^+} \ee^{\delta \beta}\right\}.
$$
Finally, set $\iota=\ell_c-U/\epsilon \in (0,1)$.

    
\subsubsection{Step 1: Coloration and permutation rules}
\label{subsec:color}

Divide the particles into \emph{active} particles and \emph{sleeping} particles: a notion that is related to {\it free} particle. Define
$$
\hat{\cal X}_N := \{(z_1,\dots,z_N)\colon\,z_i\neq z_j \, \forall\, i,j\in\{1,...,N\},i\neq j\},
$$
a set of $N$ labelled particles. We say that a particle $i \in \{1,\ldots,N\}$ is free at time $t_0\geq0$ if there exists a trajectory $\hat\eta\colon\,t \in [t_0,t_0+T] \mapsto \hat\eta(t) \in \hat{\cal X}_N$ that respects the rules of the dynamics and satisfies (see the construction carried out in \cite[Section 2.2]{GdHNOS09} and recall that $T_{\alpha}=\ee^{(\Delta-\alpha)\beta}$ with $\alpha>0$)
\begin{itemize}
\item[(i)] $||\hat\h_i(t_0+T)-\hat\eta_i(t_0)||_2>T_{\alpha}^{1/2}$.
\item[(ii)] $\forall \ t\in[t_0,t_0+T]\colon\, {\cal U}(\hat\eta(t))^{cl} = {\cal U}(\hat\eta(t_0))^{cl}$.
\end{itemize}
For $t>\ee^{D\beta}$, a particle is said to be sleeping at time $t$ if it was not free during the entire time interval $[t - \ee^{D \beta}, t]$. A non-sleeping particle is said to be active. By convention, prior to time $\ee^{D\beta}$ all particles are active.
    
\medskip\noindent
$\bullet$ {\bf Coloration rules.} These are for active particles only: sleeping particles have no color.
\begin{enumerate}
\item 
All particles in $[\bar\Lambda, \Delta - \delta]^c$ are green and remain green when entering $[\bar\Lambda, \Delta - \delta]$. Any particle that leaves $[\bar\Lambda, \Delta - \delta]$ is colored green.
\item 
When a particle wakes up in $\bar \Lambda$ at some time $t$ it is colored red if the following rules are satisfied:
\begin{itemize}
\item[(i)] $t = \sigma_i$ for some $i > 0$.
\item[(ii)] The particle is the only one that is active in $\bar \Lambda$ at time $t$.
\item[(iii)] There was a move of cost $2U$ or two ``$\delta$-close moves'' of cost $U$, i.e., both in the time interval $[t - \ee^{\delta \beta}, t]$.
\end{itemize}
\item Color yellow any particle that wakes up without being colored red.
\end{enumerate}
It follows from these rules that at time $t = 0$ all clusterized particles are without color, all active particles are green, a green particle cannot change color but can only loose color, any particle can loose its color by falling asleep, an awaking particle cannot be colored green at a wake-up time, and a colored particle can change color (from red or yellow to green) only when leaving $[\bar\Lambda, \Delta - \delta]$.
    
\medskip\noindent
$\bullet$ {\bf Permutation rules.} We couple the color rules with labelling rules by building a hierarchy on clusterized particles in the same cluster. The higher particles in this hierarchy are the sleeping ones, followed by yellow, then red, and finally green particles. To compare two sleeping particles or two particles with the same color, we say that the lower one in the hierarchy is the last aggregated particle in their shared cluster, and we break ties by some random rule. At each time $t$ when some particle has to be freed from a cluster, we set particle positions to ensure that this particle is the lowest one in the cluster hierarchy at time $t^-$. This is compatible with the local permutation rule associated with quasi-random walks.
    
The reason why we prefer to release green and red particles rather than yellow particles is that we have much less control on the latter. We also want to have to control the smallest possible number of active particles, which is why we place sleeping particles at the highest rank in the hierarchy, and we introduce the time aggregation rule to give more chance to fall asleep to any particle that was about to do so.
    
    
\subsubsection{Step 2: Starting configuration has no square: Lemma \ref{lmm:exitempty}}
\label{sub:step2}
    
Consider the case in which the starting configuration $\eta_0\in{\cal X}_{\Delta^+}$ has no quasi-square, i.e., $\eta_0\in{\cal X}_E$ (recall Definition \ref{def:recurrence} and \eqref{def:vuoto}). Then we need to prove the first part of Proposition \ref{prp:uppbound} only. The following lemma controls the exit of the dynamics from the pure gas state, which corresponds to the creation of the first droplet and therefore to the creation of a new local box.

\begin{lemma}
\label{lmm:exitempty}
Assume that $\Delta<\Theta\leq\theta$. For $\eta_0$ in ${\cal X}_{E}$,
\begin{equation}
\limsup_{\beta \rightarrow \infty} {1 \over \beta} \log P_{\eta_0}\left(
\hbox{a box creation occurs within time }
\bar\tau_1\right) \leq -[3\Delta - 2U -\theta -O(\alpha,d)].
\end{equation}
\end{lemma}
    
\begin{remark}
\label{rmk:empty}
{\rm Starting from $\eta_0\in{\cal X}_E$, reaching at time $\bar\tau_1$ a configuration such that $\pi(X(\bar\tau_1))\neq\pi(\eta_0)$ implies that a box creation has occurred. Hence the first part of Proposition \ref{prp:uppbound} follows from Lemma \ref{lmm:exitempty}.}\hfill$\spadesuit$
\end{remark}
	

\subsubsection{Step 3: Starting configuration has a single large quasi-square: Lemma \ref{lmm:exitgen}}
\label{sub:step3}

Recall that we are considering a starting configuration $\h_0\in{\cal X}_{\Delta^+}$ consisting of a single quasi-square of size $\ell_1\times \ell_2$ with $\ell_1\leq \ell_2$ and $\ell_2 \geq 3$. Recall \eqref{activereturn}-\eqref{deltareturn2} and \eqref{def:res} for the definition of resistance for a quasi-square of size $\ell_1\times \ell_2$.

\begin{lemma}
\label{lmm:exitgen}
Assume that $\Delta<\Theta\leq\theta$. Let $\eta_0 \in {\cal X}_{\Delta^+}$ be such that its restriction $\bar\eta_0$ to $\bar\Lambda$ is a quasi-square of size $\ell_1 \times \ell_2$ with $\ell_1 \leq \ell_2$ and $\ell_2 \geq 3$. If $\eta_0$ is subcritical, i.e., $\ell_1 < \ell_c$, then we set $m = \ell_1 - 2$ and
$$
a =\gamma\Bigl(\frac{1}{2}\mathbb{1}_{\{\ell_1<\ell_c-1\}}
+\frac{1}{2}\mathbb{1}_{\{\ell_1=\ell_c-1,\iota<\frac{1}{2}\}}
+(1-\iota)\mathbb{1}_{\{\ell_1=\ell_c-1,\iota\geq\frac{1}{2}\}}\Bigr)>0.
$$
Let ${\cal G}_1$ be the graph represented in Fig.~\ref{fig:G1} and ${\cal G}_2$ the graph represented in Fig.~\ref{fig:G2sub}. If $\eta_0$ is supercritical,i.e., $\ell_1 \geq \ell_c$, instead set $m = \ell_c - 2$ and
$$
a = (\epsilon - \gamma)\mathbb{1}_{\{\iota<\frac{1}{2}\}}
+\gamma\mathbb{1}_{\{\iota\geq\frac{1}{2}\}}>0.
$$
Define the same ${\cal G}_1$ (associated with a different $m$), and let ${\cal G}_2$ be the graph represented in Fig.~\ref{fig:G2sup}. Then
\begin{equation}
\label{langres}
\displaystyle\limsup_{\beta \rightarrow \infty} {1 \over \beta} \log P_{\eta_0}\left(\varphi^n \hbox{ escapes from ${\cal G}_1$}\right) 
\leq -[r(\ell_1,\ell_2)-\Delta-O(\alpha,d)]
\end{equation}
and
\begin{equation}
\label{munster}
\limsup_{\beta \rightarrow \infty} {1 \over \beta} \log P_{\eta_0}\left(\varphi^n \hbox{ escapes from ${\cal G}_2$}\right) 
\leq -[r(\ell_1,\ell_2) - \Delta + a -O(\alpha,d)].
\end{equation}
\end{lemma}

\noindent
The proof of Lemma \ref{lmm:exitgen} is given in Section \ref{sub:lmmgen}.

\begin{figure}[h!]
	$$
	\begin{tikzpicture}
		\draw (14, 0) node {${\cal I}(\ell_1 \ell_2)$};
		\draw (13.5, .25) -- (12.5, .75);
		\draw (13, .6) node[right] {$2U$};
		\draw (12, 1) node {${\cal I}(\ell_1 \ell_2 - 1)$};
		\draw (11.5, 1.25) -- (10.5, 1.75);
		\draw (11, 1.6) node[right] {$2U$};
		\draw (10.25, 1.875) -- (10.15, 1.925);
		\draw (9.85, 2.075) -- (9.75, 2.125);
		\draw (9.5, 2.25) -- (8.5, 2.75);
		\draw (9, 2.6) node[right] {$2U$};
		\draw (8, 3) node {${\cal I}(\ell_1 \ell_2 - m + 1)$};
		\draw (7.5, 3.25) -- (6.5, 3.75);
		\draw (7, 3.6) node[right] {$2U$};
		\draw (6, 4) node {${\cal I}(\ell_1 \ell_2 - m)$};
	\end{tikzpicture}
	$$
	\caption{The graph ${\cal G}_1$ in both the subcritical and the supercritical case.}
	\label{fig:G1}
\end{figure}
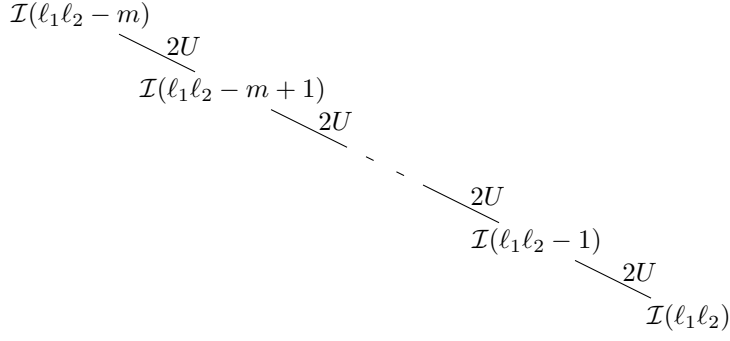
 
\begin{figure}[h!]
$$
\begin{tikzpicture}
\draw (12, 1) node {${\cal I}(\ell_1 \ell_2)$};
\draw (11.5, 1.25) -- (10.5, 1.75);
\draw (11, 1.6) node[right] {$2U$};
\draw (10.25, 1.875) -- (10.15, 1.925);
\draw (9.85, 2.075) -- (9.75, 2.125);
\draw (9.5, 2.25) -- (8.5, 2.75);
\draw (9, 2.6) node[right] {$2U$};
\draw (8, 3) node {${\cal I}(\ell_1 \ell_2 - m + 1)$};
\draw (7.5, 3.25) -- (6.5, 3.75);
\draw (7, 3.6) node[right] {$2U$};
\draw (6, 4) node {${\cal I}(\ell_1 \ell_2 - m)$};
\draw (5.5, 4.25) -- (4.5, 4.75);
\draw (5, 4.6) node[right] {$2U$};
\draw (4, 5) node {${\cal I}(\ell_1 \ell_2 - m - 1)$};
\draw (4, 4.75) -- (2.9, .3);
\draw (3, 0) node {${\cal I}(\ell_1 \ell_2 - \ell_1)$};
\draw (3.1, .3) -- (5.5, 3.75);
\end{tikzpicture}
$$
\caption{The graph ${\cal G}_2$ in the subcritical case.}
\label{fig:G2sub}
\end{figure}
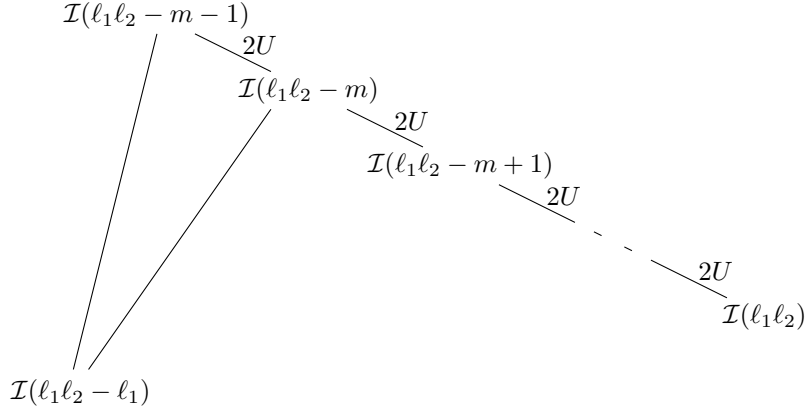

\begin{figure}[h!]
	$$
	\begin{tikzpicture}
		\draw (12, 1) node {${\cal I}(\ell_1 \ell_2 + \ell_2)$};
		\draw (11.5, 1.25) -- (10.5, 1.75);
		\draw (11, 1.6) node[right] {$2U$};
		\draw (10.25, 1.875) -- (10.15, 1.925);
		\draw (9.85, 2.075) -- (9.75, 2.125);
		\draw (9.5, 2.25) -- (8.5, 2.75);
		\draw (9, 2.6) node[right] {$2U$};
		\draw (8, 3) node {${\cal I}(\ell_1 \ell_2 + 3)$};
		\draw (7.5, 3.25) -- (6.5, 3.75);
		\draw (7, 3.6) node[right] {$2U$};
		\draw (6, 4) node {${\cal I}(\ell_1 \ell_2 + 2)$};
		\draw (5.5, 4.25) -- (4.5, 4.75);
		\draw (5, 4.6) node[right] {$2U$};
		\draw (4, 5) node {${\cal I}(\ell_1 \ell_2 + 1)$};
		\draw (4, 4.75) -- (2.9, 2.3);
		\draw (3.1, 2.3) -- (5.5, 3.75);
		\draw (3, 2) node {${\cal I}(\ell_1 \ell_2)$};
		\draw (2.5, 2.25) -- (1.5, 2.75);
		\draw (2, 2.6) node[right] {$2U$};
		\draw (1.25, 2.875) -- (1.15, 2.925);
		\draw (0.85, 3.075) -- (0.75, 3.125);
		\draw (0.5, 3.25) -- (-0.5, 3.75);
		\draw (0, 3.6) node[right] {$2U$};
		\draw (-1, 4) node {${\cal I}(\ell_1 \ell_2 - m)$};
	\end{tikzpicture}
	$$
	\caption{The graph ${\cal G}_2$ in the supercritical case.}
	\label{fig:G2sup}
\end{figure}
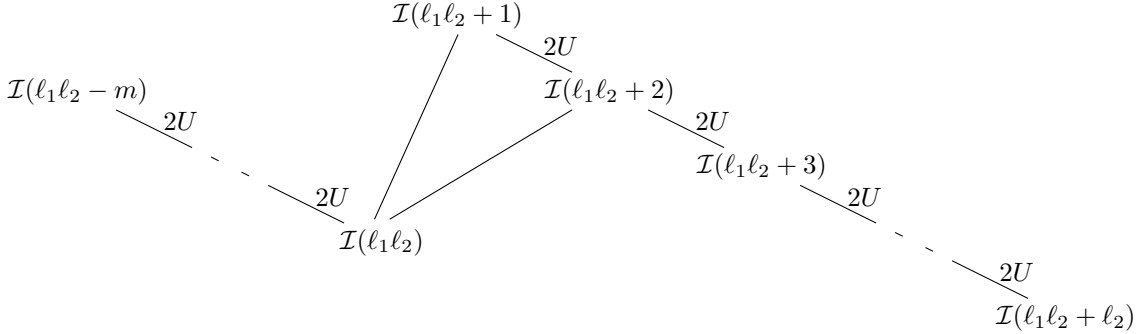

\begin{remark}
\label{rmk:implgen}
{\rm Proposition \ref{prp:uppbound} follows from Lemma \ref{lmm:exitgen} when $\h_0$ consists of a single quasi-square with size $\ell_1\times \ell_2$ and $\ell_2 \geq 3$. First, Lemma \ref{lmm:exitgen} gives us information at the return times to ${\cal X}_{\Delta^+}$ after seeing an active particle in $\bar\Lambda$. Indeed, note that such a return time can occur only in the time intervals of type $[\tau_k,\sigma_{k+1}]$, because during the time intervals of type $[\sigma_k,\tau_k]$ the configurations that are visited are not in ${\cal X}_D$ and therefore not even in ${\cal X}_{\Delta^+}$ (recall Definition \ref{def:recurrence}). It is clear that a return time in ${\cal X}_{\Delta^+}$ does not necessarily coincide with a time $\tau_k$, but during the time interval $[\tau_k,\sigma_{k+1}]$ the number of particles of the isoperimetric configuration is conserved, and so the system reaches ${\cal X}_{\Delta^+}$ in the same configuration visited at time $\tau_k$. Second, \eqref{langres} implies the first part of Proposition \ref{prp:uppbound}. Starting from $\eta_0$, if $\pi(X(\bar\tau_1)) \neq \pi(\eta_0)$, then $\varphi^n$ has escaped from ${\cal G}_1$. Hence the chain of inequalities holds due to \eqref{langres} and we get the claim. Finally, the second part of Proposition \ref{prp:uppbound} follows in the same way. Starting from $\eta_0$, if $\pi(X(\bar\tau_1))\notin\{\pi(\eta_0),\pi'\}$, then $\varphi^n$ has escaped from ${\cal G}_2$. Hence the chain of inequalities holds due to \eqref{munster}.}\hfill$\spadesuit$
\end{remark}

	
\subsubsection{Step 4: Starting configuration has a single small quasi-square: Lemma \ref{lmm:exit2x2}}
\label{sub:step4}
	
We recall that we are considering a starting configuration $\h_0\in{\cal X}_{\Delta^+}$ consisting of a single quasi-square of size $\ell_1 \times \ell_2$ with $\ell_1\leq \ell_2=2$. Thus, we need to consider only the case in which $\h$ contains a $2\times2$ square droplet (recall Remark \ref{rmk:qsdim}).

\begin{lemma}
\label{lmm:exit2x2}
Assume that $\Delta<\Theta\leq\theta$. Let $\eta_0\in{\cal X}_{\Delta^+}$ be such that its restriction $\bar\eta_0$ to $\bar\Lambda$ is a $2 \times 2$ square. Let ${\cal G}_1$ be the graph consisting of the vertex ${\cal I}(4)$ only, and define the graph ${\cal G}_2$ as
$$
\begin{array}{ccc}
& \ \ \ {\cal I}(4)\\
&/ \\
&{\cal I}(0).
\end{array}
$$
Then
\begin{equation}
\label{langres2x2}
\limsup_{\beta \rightarrow \infty} {1 \over \beta} \log P_{\eta_0}\left(\varphi^n \hbox{ escapes from ${\cal G}_1$}\right) 
\leq -[r(2,2) - \Delta - O(\alpha,d)]
\end{equation}
and
\begin{equation}
\label{munster2x2}
\limsup_{\beta \rightarrow \infty} {1 \over \beta} \log P_{\eta_0}\left(\varphi^n \hbox{ escapes from ${\cal G}_2$}\right) 
\leq -[r(2,2) - \Delta + \tfrac12\gamma - O(\alpha,d)].
\end{equation}
\end{lemma}

\noindent
The proof of Lemma \ref{lmm:exit2x2} is deferred to Section \ref{sub:lmm2x2}.

 \begin{remark}
{\rm In order to deduce Proposition \ref{prp:uppbound} from Lemma \ref{lmm:exit2x2} in case $\h_0$ consists of a single $2\times2$ square, we can argue as in Remark \ref{rmk:implgen}.}\hfill$\spadesuit$
\end{remark}


\subsubsection{Step 5: Result for a general collection of finite boxes}
\label{sub:step5}

We close by explaining how to derive Lemmas \ref{lmm:exitgen} and \ref{lmm:exit2x2} when the starting configuration is not such that 
$\bar\Lambda(0)=\bar\Lambda_0(0)$. First, we need to extend the definition of the set ${\cal I}(n)$. Given a collection $\bar\Lambda(t)=(\bar\Lambda_i(t))_{1 \leq i<k(t)}$ of finite boxes in $\Lambda_\beta$, we call ${\cal I}(n)$ the set of configurations $\eta$ such that $\bar\eta$ is of size $\sum_{1 \leq i<k(t)}|\bar\eta_i|=n$ and is the solution of the isoperimetric problem for a configuration with $n$ particles and $k(t)$ connected components. We use the notation ${\cal I}(n)^{fp}$ to indicate the presence of a free particle in one of the boxes. Moreover, in Lemma \ref{lmm:exit2x2} we need to replace the set ${\cal I}(0)$ by the set $\bar{\cal I}(n-4)$, defined as the set of configurations for which the collection $\bar\Lambda(t)$ has one local box less than $\bar\Lambda(t^-)$,  and there are $n$ particles inside $\bar\Lambda(t^-)$ and $n-4$ particles inside $\bar\Lambda(t)$. This set takes into account the dissolution of a $2\times2$ square droplet at time $t$ leading to the disappearance of one of the local boxes. Up to any coalescence between local boxes, we can argue as in the proof of Lemmas~\ref{lmm:exitgen} and \ref{lmm:exit2x2}.


\section{Proof of lemmas: from large deviations to deductive approach}
\label{sub:lmmgen}
 
Section~\ref{subsec:exitempty} shows that the proof of Lemma~\ref{lmm:exitempty} has already been achieved. Section~\ref{subsec:exitgen}, which is long and constitutes the main technical hurdle of the paper, contains the proof of Lemma \ref{lmm:exitgen} and is divided into several parts: Section~\ref{subsec:struct} outlines the structure of the proof, while Sections~\ref{subsec:escapeI}--\ref{subsec:escapeIII} work out the details of this proof for three cases. The latter rely on two further lemmas, whose proof is deferred to Sections~\ref{sec:appb1}--\ref{sec:appb2}. 

The structure of the argument used to achieve the proof of Lemma \ref{lmm:exitgen} and Lemma \ref{lmm:exit2x2} is common. Indeed, we follow a \emph{deductive approach}, in the sense that we consider a family of large deviation events and use their intricate interrelation to estimate their respective probabilities. In particular, starting from these large deviation events we will prove, by induction in $k$, a claim ${\cal P}(k)$ of the form ``if none of these events occurs, then the dynamics does not escape from the graph in the first $k$ steps''. This way of going about is inspired by the point of view that the tube of typical paths is the skeleton for the metastable crossover. Indeed, the role of the different graphs introduced below is that they describe the temporal configurational environment from which the dynamics cannot escape. We will control the evolution of the dynamics in this environment via large deviation a priori estimates, and we will need a detailed case study to be able to proceed.


\subsection{Proof of Lemma~\ref{lmm:exitempty}}
\label{subsec:exitempty}

The claim is the same as the one derived in \eqref{eq:newbox}.
\qed


\subsection{Proof of Lemma~\ref{lmm:exitgen}}
\label{subsec:exitgen} 


\subsubsection{Structure of the proof}
\label{subsec:struct}

By using the coloration and permutation rules introduced in Section~\ref{subsec:color}, we build a list of large deviation events, each having a cost 
$$
c(\cdot) = -\limsup_{\beta \rightarrow \infty} {1 \over \beta} \log P(\cdot),
$$
to prove by contradiction that if $\varphi^n$ escapes from ${\cal G}_1$, then the union $Z_1$ of these large deviation events has to occur. We define another event $Z_2$ by removing from $Z_1$ some of these large deviation events, resulting in a larger cost, and adding new large deviation events, which also have a larger cost than $Z_1$. While dealing with $Z_1$ and ${\cal G}_1$, we can consider the subcritical and the supercritical case simultaneously, but we must separate when dealing with $Z_2$ and ${\cal G}_2$. Finally, we prove that if $\varphi^n$ escapes from ${\cal G}_2$, then $Z_2$ has to occur.

In the sequel we consider three cases for escaping ${\cal G}_1$ and ${\cal G}_2$:
\begin{itemize}
\item[(I)] 
Escape from ${\cal G}_1$.
\item[(II)] 
Escape from ${\cal G}_2$ in the subcritical case.
\item[(III)] 
Escape from ${\cal G}_2$ in the supercritical case.
\end{itemize}


\subsubsection{Escape case (I)}
\label{subsec:escapeI}

	
\paragraph{$\bullet$ Large deviation events.}
Here is a list of bad events that can lead to $\varphi^n$ escaping from ${\cal G}_1$ or ${\cal G}_2$, together with a lower bound on their cost. We call {\it entrance time} and {\it exit time} all times $t$ at which a free particle enters or leaves $\bar \Lambda(t)$. A {\it special time} is an entrance time, an exit time, a wake up time, a return time $\tau_i$ or a boxes special time (recall \eqref{def:B}). Note that each $\sigma_i$ defined in \eqref{activereturn} is a special time, since it is either an entrance time or a wake-up time. As above, we say that two times $t_1 < t_2$ are {\it $\delta$-close} if $t_2 - t_1 < \ee^{\delta \beta}$.
	
\begin{description}
\item[$A:$] 
A recurrence or non-superdiffusivity property is violated within time $T_{\Delta^+} \ee^{\delta \beta}$. This event has an infinite cost, i.e., its probability is $\SES$.
\item[$B:$] 
There are more than $\ee^{(2\alpha+\delta)\beta}$ special times within time $T_{\Delta^+} \ee^{\delta \beta}$. This event has an infinite cost.
\item[$C:$] 
Within time $T_{\Delta^+} \ee^{\delta \beta}$ there is a time interval of length $\ee^{\delta \beta}$ that contains a special time followed by a move of cost larger than or equal to $U$. This event costs at least $U - O(\delta)$.
\item[$C':$] 
Within time $T_{\Delta^+} \ee^{\delta \beta}$ there is a time interval $I$ of length at most $\ee^{\delta \beta}$ that contains a move of cost larger than or equal to $U$ and ends with the entrance in $\bar\Lambda$ of a free particle that was outside $\bar\Lambda$ during $I$. This event costs at least $U - O(\delta)$.
\item[$D:$] 
Within time $T_{\Delta^+} \ee^{\delta \beta}$ there is a time interval of length $\ee^{D \beta}$ that contains a special time followed by a move of cost larger than or equal to $2U$ or two $\delta$-close moves of cost larger than or equal to $U$. This event costs at least $2U - D - O(\delta)$.
\item[$D':$] 
Within time $T_{\Delta^+} \ee^{\delta \beta}$ there is a time interval $I$ of length at most $\ee^{D \beta}$ that contains a move of cost larger than or equal to $2U$ or two $\delta$-close moves of cost larger than or equal to $U$, and ends with the entrance in $\bar\Lambda$ of a free particle that was outside $\bar\Lambda$ during $I$. This event costs at least $2U - D - O(\delta)$.
\item[$E:$] 
Within time $T_{\Delta^+} \ee^{\delta \beta}$ there is a time interval $[t_1, t_2]$ such that $|\bar X|$ is constant on $[t_1, t_2]$, the local energy difference $\bar H(\eta(t_2)) - \bar H(\eta(t_1))$ is larger than or equal to $3U$, and $t_1$ is $\delta$-close to some earlier special time. This event costs at least $3U - \Delta - \alpha -O(\delta)$.
\item[$F_{m + 1}:$]
There are $m + 1$ times $t_1 < \cdots < t_{m + 1} < T_{\Delta^+} \ee^{\delta\beta}$ at which some particle is colored red. This event costs at least $(m+1) (2U - \Delta-\alpha) -O(\delta)$.
\item[$G:$] 
There are two red particles at a same time $t < T_{\Delta^+} \ee^{\delta\beta}$ in $[\bar\Lambda, D + \delta]$. This event costs at least $U - d + \epsilon -\alpha - O(\delta)$.
\item[$G':$] 
There are a red and a green particles at a same time $t < T_{\Delta^+} \ee^{\delta\beta}$ in $[\bar\Lambda, D + \delta]$. This event costs at least $U - d -\alpha - O(\delta)$.
\item[$G_4':$] 
There are four active particles, red or green, at a same time $t<T_{\Delta^+}\ee^{\delta\beta}$ in a box of volume $\ee^{(D+\delta)\beta}$, or a particle that belongs to a cluster consisting of two or three active particles only falls asleep. This event costs at least $3\Delta-2U-\theta+3\alpha-2d-O(\delta)$.
\item[$H_2:$] 
There are two green particles at a same time $t < T_{\Delta^+} \ee^{\delta\beta}$ in $[\bar\Lambda, D + \delta]$. This event costs at least $\Delta - D +\alpha - O(\delta)$.
\end{description} 

Set 
\begin{equation}
\label{Z1}
Z_1 = A \cup B \cup C \cup C' \cup D \cup D' \cup E \cup F_{m + 1} \cup G \cup G' \cup G_4' \cup H_2,
\end{equation}
so that $Z_1^c$ implies $A^c$,$B^c$, \dots, $G_4'^c$, $H_2^c$. We will prove by induction that, for all $0 \leq k \leq n$,

\begin{ClmPk}
If $Z_1$ does not occur, then
\begin{itemize}
\item[(i)] 
$\varphi^k$ does not escape from ${\cal G}_1$.
\item[(ii)] 
A particle is painted red each time $\varphi^k$ climbs along an $2U$-edge of ${\cal G}_1$.
\item[(iii)] 
No particle is painted yellow within $\tau_k$.
\item[(iv)] 
No box creation occurs within $\tau_k$.
\end{itemize}
\end{ClmPk}
    
\noindent
Property (iv) avoids the creation of new boxes within time $t\leq\tau_k$. Since the cost of $Z_1$ is given by the smallest cost of its components $A$, $B$, $\dots$, we obtain
$$
c(Z_1) = \left\{
\begin{array}{ll}
c(F_{m + 1}) \geq r(\ell_1,\ell_2) - \Delta -O(\alpha) - O(\delta) & \hbox{if $\ell_1 < \ell_c$,}\\
c(H_2) \geq r(\ell_1,\ell_2) - \Delta - O(\alpha,d) - O(\delta) & \hbox{if $\ell_1 \geq lc$},
\end{array}
\right.
$$
and this will prove \eqref{langres}.
 
    
\paragraph{$\bullet$ Proof of $\boldsymbol{{\cal P}(k)}$, $\boldsymbol{0 \leq k \leq n}$.}
${\cal P}(0)$ obviously holds because $\tau_0 = 0$. We prove ${\cal P}(k + 1)$ by assuming ${\cal P}(k)$. Let us assume that $Z_1^c$ occurs. We have to control the process $X$ on the time interval
$$
[\tau_k, \tau_{k + 1}] = [\tau_k, \sigma_{k + 1}] \cup [\sigma_{k + 1}, \tau_{k + 1}].
$$
We analyse these two intervals separately.
    
\medskip\noindent
{\it The time interval $[\tau_k, \sigma_{k + 1}]$:}
Consider the process
$$
\Delta \bar H\colon\, t \in [\tau_k, \sigma_{k + 1}) \mapsto \bar H(X(t)) - \bar H(X(\tau_k)).
$$
It follows from the definition of $\sigma_{k + 1}$ that $|\bar X(t)|$ does not change during the time interval $[\tau_k, \sigma_{k + 1})$. ${\cal P}(k)$ implies, in particular,
\begin{equation}
\label{epoisse}
X(\tau_{k}) \in {\cal I}(\ell_1 \ell_2 - i)
\end{equation}
for some $1 \leq i \leq m$, so that $\bar X(\tau_k)$ is a solution of the isoperimetric problem, and this implies that $\Delta\bar H$ cannot go down below $0$. Then $E^c$ implies that $\Delta\bar H$ cannot go above $2U$, and it follows that
$$
\Delta \bar H(t) \in \{0, U, 2U\}, \qquad \tau_k \leq t < \sigma_{k + 1}.
$$
The process $\Delta \bar H$ can therefore be seen as a succession of increases and decreases of the local energy to some of these three values. We claim that $Z_1^c$ implies:
\begin{itemize}
\item[(i)] 
Each increase of $\Delta\bar H$ to $2U$ is followed by a $\delta$-close decrease to $U$ or $0$.
\item[(ii)] 
Each increase of $\Delta\bar H$ to $U$ is followed by a $\delta$-close decrease to $0$ or a $\delta$-close increase to $2U$.
\item[(iii)] 
After each decrease to $U$, $\Delta\bar H$ has to increase to $2U$ within a time $\ee^{(U + \delta)\beta}$ or to decrease to $0$ within a time $\ee^{\delta\beta}$. 
\end{itemize}
Indeed, (i) and (ii) follow from the recurrence property to ${\cal X}_0$ implied by $A^c$, while (iii) follows from the recurrence properties to ${\cal X}_U$ and ${\cal X}_0$ implied by the same event.

Now, $\sigma_{k + 1}$ can be reached either via the entrance of a free particle in $\bar\Lambda$ or by freeing some particle in $\bar\Lambda$. We will refer to these as the entrance and wake-up case, and we analyse them separately.
\begin{description}
\item[{\it \underline{Entrance case:}}]
In this case properties (i)--(iii), $C'^c$ and $D'^c$ imply that $\Delta\bar H(\sigma_{k + 1}^-) = 0$, hence $X(\sigma_{k + 1}^-) \in {\cal I}(\ell_1 \ell_2 -i)$ with $1 \leq i \leq m$ defined by \eqref{epoisse}. $\bar X(\sigma_{k +1})$ is then made up of an isoperimetric configuration of size $\ell_1 \ell_2 -i$ and a free particle, for which we use the short-hand notation $X(\sigma_{k + 1}) \in {\cal I}(\ell_1 \ell_2 - i)^{fp}$.
\item[{\it \underline{Wake-up case:}}]
Recall \eqref{epoisse} again, and use that $E^c$ and $i \leq m < \ell_1 - 1$ imply
$$
\bar H(X(\sigma_{k + 1})) \leq \bar H({\cal I}(\ell_1 \ell_2 - i)) + 2U = \bar H({\cal I}(\ell_1 \ell_2 - i -1)) + \Delta.
$$
Since a free particle has perimeter $4$, we also have the reverse inequality 
$$
\bar H(X(\sigma_{k + 1})) \geq \bar H({\cal I}(\ell_1 \ell_2 - i -1)) + \Delta,
$$
and so we conclude that    	
\begin{equation}
\label{petit_gaugry}
\bar H(X(\sigma_{k + 1})) = \bar H({\cal I}(\ell_1 \ell_2 - i -1)) + \Delta.
\end{equation}
Together with properties (i)--(iii) this implies that the waking-up particle is colored red: the requested move of cost $2U$, or two $\delta$-close move of cost $U$, do not have to be $\delta$-close to $\sigma_{k + 1}$, and it is not possible that a particle wakes up from a  $U$-reducible configuration that is reached without waking up from a configuration in ${\cal X}_D$. Indeed, it is impossible to obtain an isoperimetric configuration with a free particle by detaching a particle from an isoperimetric configuration in ${\cal X}_0 \setminus {\cal X}_U$: if the free particle is detached from the external boundary of the configuration, then the starting configuration is not isoperimetric, while if the particle is detached from the internal boundary, then it is not in ${\cal X}_0$. Equation~\eqref{petit_gaugry} also implies that $X(\sigma_{k + 1}) \in {\cal I}(\ell_1 \ell_2 - i - 1)^{fp}$.
\end{description}
    
The above analysis of the time interval $[\tau_k, \sigma_{k + 1}]$ requires a few concluding remarks. First, we proved that no yellow particle can be produced during this time interval. Second, $F_{m + 1}^c$ together with ${\cal P}(k)$ and $X(0) \in {\cal I}(\ell_1 \ell_2)$ imply that the wake-up case has to be excluded when $i = m$. Third, we can conclude 
\begin{equation}
\label{maroilles}
X(\sigma_{k + 1}) \in \left\{\begin{array}{ll}
{\cal I}(\ell_1 \ell_2 - j)^{fp} \hbox{ for some $j \in \{i, i + 1\}$}
&\hbox{if $i < m$,} \\
\noalign{\vskip 3pt}
{\cal I}(\ell_1 \ell_2 - j)^{fp} \hbox{ with $j = i$}
&\hbox{if $i = m$.} 
\end{array}
\right.
\end{equation}

\medskip\noindent
$\blacktriangleright$ {\it The time interval $[\sigma_{k + 1}, \tau_{k + 1}]$:}
$A^c$ implies that $\tau_{k + 1} - \sigma_{k + 1} < \ee^{(D + \delta / 2)\beta}$. From ${\cal P}(k)$ and our previous analysis  we also know that we have a red or a green particle in $\bar \Lambda$ and that no yellow particle was produced during the time interval $[0, \sigma_{k + 1}]$. Therefore the non-superdiffusivity property, $G^c$, $G'^c$ and $H_2^c$ imply that no other (colored) particle can enter $\bar \Lambda$ before time $\tau_{k + 1}$.
    
Let us next consider the process
$$
\Delta \bar H\colon\, t \in [\sigma_{k + 1}, \tau_{k + 1}] \mapsto \bar H(X(t)) - \bar H(X(\sigma_{k + 1}))
$$
and make two preliminary observations:
\begin{itemize}
\item[(i)]
Since there is a free particle in $\bar \Lambda$, the recurrence property to ${\cal X}_0$, $C^c$, $C'^c$ and the fact that no other active particle can enter $\bar\Lambda$ before $\tau_{k + 1}$ imply that $\Delta\bar H$ first has to decrease within a time $\ee^{\delta\beta}$.
\item[(ii)]
The recurrence property to ${\cal X}_U$, $D^c$ and $D'^c$ imply that before time $\tau_{k + 1}$ there will be neither a move of cost larger than or equal to $2U$, nor a succession of $\delta$-close moves of cost larger than or equal to $U$.    
\end{itemize}
We now separate two complementary events, to which we will refer as the {\it good attachment} and the {\it exit}.
\begin{description}
\item[$\blacktriangleright$ \it Good attachment:]
This occurs when $\Delta\bar H$ reaches the level $-2U$ before the free particle leaves $\bar \Lambda$. With $1 \leq j \leq m$ defined in \eqref{maroilles}, the local energy is equal to
$$
\bar H({\cal I}(\ell_1 \ell_2 - j)^{fp})- 2U = \bar H({\cal I}(\ell_1 \ell_2 - (j - 1)))
$$
because $j - 1 \leq m - 1 < \ell_1 - 1$ and $j > 0$: good attachment is excluded when $j = 0$ because 
$$
\bar H({\cal I}(\ell_1 \ell_2)^{fp}) - 2U < \bar H({\cal I}(\ell_1 \ell_2 + 1)).
$$
The recurrence property to ${\cal X}_0$, observation (ii) and the fact that no other free particle can enter $\bar\Lambda$ before time $\tau_{k + 1}$ imply that $\Delta\bar H$ can only oscillate between the levels $-2U$ and $-U$. This excludes any possibility for the active particle to leave $\bar\Lambda$ before time $\tau_{k + 1}$, and $X$ has to reach ${\cal X}_D$ by reaching ${\cal X}_U$ and making the active particle fall asleep. Since they are reached from level $-2U$, configurations at level $-U$ are $U$-reducible. It follows that $X$ reaches ${\cal X}_D$ at the level $-2U$, i.e., in ${\cal I}(\ell_1 \ell_2 - (j - 1))$.
\item[$\blacktriangleright$ \it Exit:]
This occurs when $\Delta\bar H$ does not reach the level $-2U$ before the free particle leaves $\bar \Lambda$. Observation (i) implies that $\Delta\bar H$ first decreases to $-\Delta$ or $-U$. In the first case $X$ reaches ${\cal X}_D$ in ${\cal I}(\ell_1 \ell_2 - j)$ with $j$ defined in \eqref{maroilles}. In the second case the recurrence property to ${\cal X}_0$, observation (ii) and the fact that no other free particle can enter $\bar\Lambda$ before $\tau_{k + 1}$ imply that $\Delta\bar H$ can only oscillate between the levels $-U$ and $0$ before possibly going down to $-\Delta$. Since configurations at levels $-U$ and $0$ are all $U$-reducible (consider the reverse path to $X(\sigma_{k + 1}^-)$), $\Delta\bar H$ must eventually go down to $-\Delta$: $X$ reaches ${\cal X}_D$ in ${\cal I}(\ell_1 \ell_2 - j)$.
 \end{description}
Our permutation rules now imply that no yellow particle can be produced during the time interval $[\sigma_{k + 1}, \tau_{k + 1}]$, and we conclude that 
$$
X(\tau_{k + 1}) \in \left\{
\begin{array}{ll}
{\cal I}(\ell_1 \ell_2 - i') \hbox{ for some $i' \in \{j - 1, j\}$}
&\hbox{if $j > 0$,} \\
\noalign{\vskip 3pt}
{\cal I}(\ell_1 \ell_2 - i') \hbox{ with $i' = j$}
&\hbox{if $j = 0$.} 
\end{array}
\right.
$$
Combined with \eqref{maroilles} and the fact that a red particle was produced if $j = i + 1$, it remains to prove ${\cal P}(k+1)$-iv). But    this follows from the event $G_4'^c$ and ${\cal P}(k+1)$(iii), and ends our induction.


\paragraph{$\bullet$ Cost estimates.}

To complete the proof of \eqref{langres}, we only need to check the lower bounds for the cost of each event that makes up $Z_1$, for which we refer to Appendix \ref{sec:appb3}. This concludes Case (I).
\qed


\subsubsection{Escape case (II): Lemmas \ref{gex}--\ref{gorgonzola}}
\label{subsec:escapeII}


\paragraph{$\bullet$ Special times and large deviation events in the subcritical case.}

In the subcritical case, the cost of $Z_1$ equals the cost of $F_{m + 1}$. We build $Z_2$ by removing $F_{m + 1}$ from $Z_1$, before adding new large deviation events. With
$$
\ell_1' = \ell_2 - 1, \qquad \ell_2' = \ell_1,
$$
the proof of \eqref{langres} shows that $(Z_1 \setminus F_{m + 1})^c$ implies that either $\varphi^n$ does not escape from ${\cal G}_1$, or there is a first return time $\tau_{k_0}$ such that $X(\tau_{k_0}) \in {\cal I}(\ell_1' \ell_2' + 2)$, and an $(m + 1)^{th}$ particle is colored red at time $\sigma_{k_0 + 1}$. The following formula is a definition of $k_0$:
\begin{equation}
\label{bleu}
\hbox{\it $\sigma_{k_0 + 1}$ is the $(m + 1)^{th}$ attribution time of the red color.}
\end{equation}
Note that before time $\tau_{k_0}$ no particle can be colored yellow and there are at least $\ell_1'\ell_2'$ sleeping particles for any $t\in[0,\tau_{k_0}]$. In proving \eqref{munster} we will therefore have to deal with yellow particles. These cannot be controlled by their too low energetic cost, but they are closely related to the notion of $U$-reducibility. A careful analysis of the possible trajectories between $U$-reducible clusterized configurations and configurations in ${\cal X}_D$ will be the key tool to control the yellow particles. To that end we set $\tilde\tau_{k_0} = \tau_{k_0}$ and, for $k \geq k_0$,
$$
\tilde\sigma_{k + 1} = \inf\left\{t > \tilde\tau_k\colon\,
\hbox{there is a free particle inside $\bar\Lambda$ at time $t$}\right\},
$$
and
$$
\tilde\tau_{k + 1} = \inf\left\{t > \tilde\sigma_{k + 1}\colon\,\hbox{$X(t) \in {\cal X}_D$ 
or $X(t) \in {\cal I}(\ell_1' \ell_2' + 1) \setminus {\cal X}_U$}\right\}.
$$
Note the difference between these definitions  and those of the special times $\sigma_{i + 1}$ and $\tau_{i + 1}$: they are related to {\it free} particles and ${\cal X}_D \cup ({\cal I}(\ell_1' \ell_2' + 1) \setminus {\cal X}_U)$, rather than to {\it active} particles and ${\cal X}_D$. However, $(Z_1 \setminus F_{m + 1})^c$ implies that $\tilde\sigma_{k_0 + 1} = \sigma_{k_0 + 1}$. To prove \eqref{munster} we must analyze the time intervals $[\tilde\tau_k, \tilde\sigma_{k + 1}]$ and $[\tilde\sigma_{k + 1}, \tilde\tau_{k + 1}]$, just like we analyzed the time intervals $[\tau_i, \sigma_{i + 1}]$ and $[\sigma_{i + 1}, \tau_{i + 1}]$  to prove \eqref{langres}. We needed such an analysis for all $1 \leq i < n$, but now it will turn out that it will be enough to consider $1 
\leq k < \tilde n$ with
\begin{equation}
\label{fourme}
\tilde n = \min\{\tilde n_1, \tilde n_2\}
\end{equation}
and
\begin{align*}
\tilde n_1
&= \max\{k \geq k_0\colon\, \tilde\tau_k \leq T_{\Delta^+} \ee^{\delta\beta}\}, \\
\tilde n_2
&= \min\{k > k_0\colon\, X(\tilde\tau_k) \in {\cal I}(\ell_1' \ell_2' + 2)\}. 
\end{align*}
We will add $\tilde\sigma_k$ and $\tilde\tau_k$, $1 \leq k \leq \tilde n$ to our set of special times.  

\medskip\noindent
$\blacktriangleright$ {\it \underline{The main obstacle:}}
With a pair of particles $\{i, j\}$ we associate a family of special times $\theta^{ij}_k$, $k \in \N_0$. Before giving the definition of these stopping times, let us explain what they will be used for. In proving \eqref{langres}, we could exclude the simultaneous presence of two {\it free} particles in $\bar\Lambda$. This was done by excluding the simultaneous presence of two {\it active} particles in $[\bar\Lambda, D + \delta]$ by the means of large deviation events to control red and green particles and the inductive hypothesis to control yellow particles. In proving \eqref{munster}, we still need to exclude the simultaneous presence of two {\it free} particles in $\bar\Lambda$, but we have to allow the simultaneous presence of two {\it active} particles in $\bar \Lambda$. We will face this obstacle by using large deviation events and some inductive hypothesis to exclude, on the one hand, the simultaneous presence of {\it three} active particles in $[\bar\Lambda, D+\delta]$, and showing, on the other hand, that the first simultaneous presence of two free particles $i$ and $j$ in $\bar\Lambda$ at a time $T^{ij}$ would imply some large deviation event $J^{ij}$ that involves the two particles $i$ and $j$ during a time interval $[\theta^{ij}_k, T^{ij}]$ in which $i$ and $j$ are the only active particles in $[\bar\Lambda, D + \delta]$.
    
Let us now give the precise definitions for $\theta^{ij}_k$ and $J^{ij}$. We call $\theta^{ij}_0 < \theta^{ij}_1 < \cdots$ the ordered sequence of times $t$ such that one of the following events occurs:
\begin{itemize}
\item[(i)] $i$ is clusterized in $\bar\Lambda$, $j$ is freed inside $\bar\Lambda$, and there was at $t^-$ a single cluster in $\bar\Lambda$ that contained $i$ and $j$.
\item[(ii)] 
$i$ enters $[\bar \Lambda, D + \delta]$ and $j$ is in $[\bar \Lambda, D + \delta]$, so that $i$ was outside $[\bar \Lambda, D + \delta]$ at time $t^-$.
\item[(iii)] 
$i$ is clusterized in $\bar\Lambda$, $j$ is free in $[\bar\Lambda, D + \delta]$, a third particle $k$ leaves $[\bar\Lambda, D + \delta]$ and there is no other free particle in $[\bar\Lambda, D + \delta]$, so that $k$ was inside $[\bar\Lambda, D + \delta]$ at time $t^-$.
\end{itemize} 
We call $T^{ij}$ the first time when particles $i$ and $j$ are both free in $\bar\Lambda$. We say that $J^{ij}$ occurs if $T^{ij} \leq T_{\Delta^+} \ee^{\delta\beta}$ and there is some $\theta^{ij}_k < T^{ij}$ such that any active particle in $[\bar\Lambda, D + \delta]$ during the time interval $[\theta^{ij}_k, T^{ij}]$ is either $i$ or $j$. The following lemma expresses one of the main properties of the large deviation event $J^{ij}$.

\begin{lemma}
\label{gex}
If $T^{ij} \leq T_:{\Delta^+} \ee^{\delta\beta}$, then either $J^{ij}$ occurs or there is a time $t \leq T^{ij}$ at which there are at least three active particles inside $[\bar\Lambda, D + \delta]$.
\end{lemma}

\noindent
The proof of Lemma \ref{gex} is deferred to Section \ref{sec:appb1}.


\paragraph{$\bullet$ Large deviations events.}
The event $\tilde B$ in the following list contains $B$ because we enlarge our set of {\it special times} by adding the $\tilde\sigma_k$, $\tilde\tau_k$ and $\theta^{ij}_k$. In the same way, $\tilde C$ and $\tilde D$ contain $C$ and $D$. The event $\tilde F_{m + 1}$ is instead contained in $F_{m + 1}$ and has a larger cost.
\begin{description}
\item[$\tilde B:$] 
There are more than $\ee^{(2\alpha+\delta) \beta}$ special times within time $T_{\Delta^+} \ee^{\delta \beta}$. This event has an infinite cost.
\item[$\tilde C:$] 
Within time $T_{\Delta^+} \ee^{\delta \beta}$ there is a time interval of length $\ee^{\delta \beta}$ that contains a special time followed by a move of cost larger than or equal to $U$. This event costs at least $U - O(\delta)$.
\item[$\tilde D:$] 
Within time $T_{\Delta^+} \ee^{\delta \beta}$ there is a time interval of length $\ee^{D \beta}$ that contains a special time followed by a move of cost larger than or equal to $2U$ or two $\delta$-close moves of cost larger than or equal to $U$. This event costs at least $2U - D - O(\delta)$.
\item[$G_3':$]
There are three active particles, red or green, together with a particle from a cluster at a same time $t<T_{\Delta^+}\ee^{\delta\beta}$ in a box of volume $\ee^{(D+\delta)\beta}$, or a particle that belongs to a cluster consisting of two or three active particles only falls asleep. This event costs at least $3\Delta-2U-\theta+3\alpha-2d-O(\delta)$.
\item[$\tilde F_{m + 1}:$]
Within time $T_{\Delta^+} \ee^{\delta\beta}$ there are $m + 1$ attributions of red color and there are either an extra move of cost larger than or equal to $2U$ or two $\delta$-close moves of cost larger than or equal to $U$, or else the occurrence of one of the events $J^{ij}$. Note that $\tilde F_{m + 1} = F_{m + 1} \cap (F_1\cup \bigcup_{i,j}J^{ij})$. This event costs at least $(m + \frac{3}{2}) (2U - \Delta-\alpha) - O(\delta)$.
\end{description}

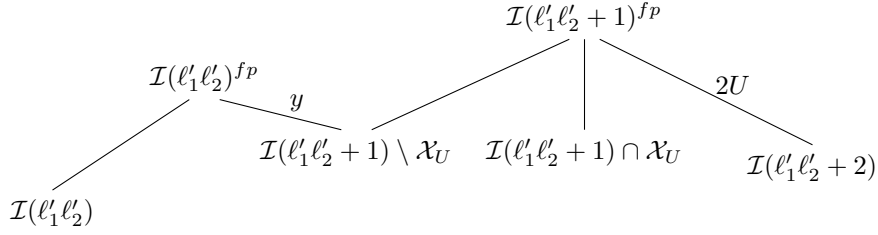
\begin{figure}[htbp]
$$
\begin{tikzpicture}
\draw (2, -.8) node {${\cal I}(\ell_1'\ell_2')$};
\draw (2, -.5) -- (3.8, .7);
\draw (4, 1) node {${\cal I}(\ell_1'\ell_2')^{fp}$};
\draw (4.2, .7) -- (5.8, .3);
\draw (5, .65) node [right] {$y$};
\draw (6, 0) node {${\cal I}(\ell_1'\ell_2' + 1) \setminus {\cal X}_U$};
\draw (6.2, .3) -- (8.8, 1.5);
\draw (9, 0) node {${\cal I}(\ell_1'\ell_2' + 1) \cap {\cal X}_U$};
\draw (9, .3) -- (9, 1.5);
\draw (9, 1.8) node {${\cal I}(\ell_1'\ell_2' + 1)^{fp}$};
\draw (9.2, 1.5) -- (12, .1);
\draw (10.6, .9) node [right] {$2U$};
\draw (12, -.2) node {${\cal I}(\ell_1'\ell_2' + 2)$};
\end{tikzpicture}
$$
\caption{The graph $\tilde{\cal G}$.}
\label{fig:tildecalG}
\end{figure}

Set 
\begin{equation}
\label{Z2sub}
Z_2 = A \cup \tilde B \cup \tilde C \cup C' \cup \tilde D \cup D' \cup E
\cup \tilde F_{m + 1} \cup G \cup G' \cup G_3' \cup G_4' \cup H_2.
\end{equation}
Let $\tilde{\cal G}$ be the graph in Fig.~\ref{fig:tildecalG}.

Recall \eqref{bleu} and \eqref{fourme}, and set
$$
\tilde\varphi^k
= (X(\tilde\tau_{k_0}), X(\tilde\sigma_{k_0 + 1}), X(\tilde\tau_{k_0 + 1}), \dots, X(\tilde\sigma_k), X(\tilde\tau_k)),
\qquad k_0 \leq k \leq \tilde n. 
$$
We will prove by induction that, for all $k_0 \leq k \leq \tilde{n}$,
\begin{ClmPkt}
If $Z_2$ does not occur, then
\begin{itemize}
\item[(i)] 
$\tilde\varphi^k$ does not escape from $\tilde {\cal G}$.
\item[(ii)] 
Some particle can be colored yellow during the time interval $[\tilde\tau_{k_0}, \tilde\tau_k]$, but only during the climbing of the $y$-edge of $\tilde {\cal G}$.
\item[(iii)] 
There is at most one yellow particle at each time $t \leq \tilde\tau_k$.
\item[(iv)] 
Each time $0 \leq t \leq \tilde\tau_k$ a particle falls asleep there is no yellow particle at the first $\tilde\tau_j$, $1 \leq j \leq k$, larger than or equal to $t$.
\item[(v)] 
For all $k_0<j \leq k$, if $X$ visits ${\cal X}_U$ during the time interval $[\tilde\sigma_j, \tilde\tau_j)$, then there is no red or green particle in $\bar\Lambda$ at time $\tilde\tau_j$.
\item[(vi)] 
At each time $0 \leq t \leq \tilde\tau_k$ there are at least $\ell_1' \ell_2'$ sleeping particles.
\item[(vii)] No box creation occurs within time $\tau_k$.
\end{itemize}
\end{ClmPkt}

\noindent
Property (i) is the main one we are interested in. Property (iv) implies that if a particle falls asleep when there is a yellow particle, then it is the yellow particle that falls asleep. Property (vi) is easy to check and simplifies a few steps of the proof. We will use properties (ii)--(iv) to control inductively the yellow particles, in particular, property (iii) will be used to prove property (vii). Property (v) will be used to prove property (iv) with the help of the following lemma, whose proof is deferred to Section \ref{sec:appb2}.

\begin{lemma}
\label{gorgonzola}
If $Z_2$ does not occur, then, for all $k \leq \tilde n$, either $X(\tilde\tau_k) \in {\cal X}_U$ or $X(t) \not\in {\cal X}_U$ for all $t \in [\tilde\tau_k, \tilde\sigma_{k + 1})$.
\end{lemma}

\medskip\noindent
$\trianglerighteq$
Before proving $\tilde{\cal P}(k)$, $k_0 \leq k \leq \tilde n$, let us show that $\tilde{\cal P}(\tilde n)$ implies for both cases $\tilde n = \tilde n_1$ and $\tilde n = \tilde n_2$ that if $Z_2^c$ occurs, then $\varphi^n$ cannot escape from ${\cal G}_2$. 

For $\tilde n = \tilde n_1$, since $Z_2^c$ implies that $\tilde\varphi^{\tilde n}$ does not escape from $\tilde{\cal G}$, it suffices to prove, for all $k_0 \leq l \leq n$, that $\tau_l = \tilde\tau_k$ for some $k \leq \tilde n$. We prove prove by induction on $l \geq k_0$. The claim is obvious for $l = k_0$. If this is true for some $l < n$, then $\tilde\sigma_{k + 1} = \sigma_{l + 1}$ and, since $\tilde n = \tilde n_1$, there is a last time $\tilde\tau_{m^*} > \sigma_{l + 1}$ before $\tau_{l + 1}$:
$$
\tilde\tau_{m^*} = \max\{\tilde\tau_m \leq \tau_{l + 1}\colon\, \tilde\tau_m > \sigma_{l + 1}\}.
$$
If $X(\tilde\tau_{m^*}) \not\in {\cal X}_U$, then, by Lemma~\ref{gorgonzola}, $\tau_{l + 1} \geq \tilde\sigma_{m^* + 1}$ and $\tilde\tau_{m^*}$ cannot be the {\it last} time $\tilde\tau_m$ smaller than or equal to $\tau_{l + 1}$. It follows that $X(\tilde\tau_{m^*}) \in {\cal X}_D$ and, since $\tilde\tau_{m^*} > \sigma_{l + 1}$, $\tilde\tau_{m*} \geq \tau_{l + 1} \geq \tilde\tau_{m*}$: the two times coincide.

For $\tilde n = \tilde n_2$, like for $\tilde n = \tilde n_1$, we prove that there is some $k_1 > k_0$ such that $\tau_{k_1} = \tilde\tau_{\tilde n}$ and
$$
\{\tau_{k_0}, \tau_{k_0 + 1}, \dots, \tau_{k_1}\}
\subset \{\tilde\tau_{k_0}, \tilde\tau_{k_0 + 1}, \dots, \tilde\tau_{\tilde n}\}.
$$
It follows that $Z_2^c$ implies that $\varphi^{k_1}$ does not escape from ${\cal G}_2$. Since $\tilde n = \tilde n_2$, $X$ reaches ${\cal X}_D$ at time $\tilde\tau_{\tilde n}$ and, since $\tilde{\cal P}(\tilde n)$-i) implies that it does so by making some particle fall asleep, $\tilde{\cal P}(\tilde n)$(iv) implies that there is no yellow particle at time $\tau_{k_1} = \tilde\tau_{\tilde n}$. Using that $\tilde F_{m + 1}^c$ excludes any $(m + 2)^{th}$ attribution of the red color, we can show by induction, as in the proof of \eqref{langres}, that $\varphi^k$, for $k \geq k_1$, cannot escape anymore from ${\cal G}_1$, the subgraph of ${\cal G}_2$.

Since the cost of $Z_2$ is given by the smallest cost of its components, we obtain
$$
c(Z_2) = \left\{
\begin{array}{ll}
c(\tilde F_{m + 1}) \geq r(\ell_1,\ell_2) - \Delta + \frac{\gamma}{2} -O(\alpha,d) - O(\delta) & \hbox{if $\ell_1 < \ell_c - 1$,}\\
c(\tilde F_{m + 1})\geq r(\ell_1,\ell_2) - \Delta + \frac{\gamma}{2} -O(\alpha,d) - O(\delta) & \hbox{if $\ell_1=\ell_c-1$ and $\iota<1/2$}, \\
c(H_2)\geq r(\ell_1,\ell_2) - \Delta + (1-\iota)\gamma -O(\alpha,d) - O(\delta) & \hbox{if $\ell_1=\ell_c-1$ and $\iota\geq1/2$}.
\end{array}
\right.
$$
To prove \eqref{munster} in the subcritical case, it only remains to prove $\tilde{\cal P}(\tilde n)$ and check the given cost estimates .

\medskip\noindent
$\bullet$ {\bf Proof of $\boldsymbol{\tilde{\cal P}(k)}$, $\boldsymbol{k_0 \leq k \leq \tilde n}$.}
$\tilde{\cal P}(k_0)$(iii) and $\tilde{\cal P}(k_0)$(vi) follow from the argument explained below \eqref{bleu}, while the other items are obvious. For $k \geq k_0$, we assume $\tilde{\cal P}(k)$ to prove $\tilde{\cal P}(k + 1)$. We consider four cases, depending on the configuration at time $\tilde\tau_k$ in one of the sets of $\tilde{\cal G}$ ordered from right to left.


\paragraph{Case 1: $X(\tilde\tau_k) \in {\cal I}(\ell_1' \ell_2' + 2)$.}

If $k \neq k_0$, then $\tilde n = \tilde n_2 = k$ and there is nothing to prove. We only need to consider the case $k = k_0$, for which $\tilde\sigma_{k + 1} = \sigma_{k_0 + 1}$ and the definition of $\sigma_{k_0 + 1}$ gives $X(\sigma_{k_0 + 1}) \in {\cal I}(\ell_1' \ell_2' + 1)^{fp}$. The analysis of the time intervals $[\tau_k, \sigma_{k + 1}]$ we gave to prove \eqref{langres} also shows that in this case no yellow particle can be produced during the time interval $[\tilde\tau_{k_0}, \tilde\sigma_{k_0 + 1}]$, and that there are $\ell_1' \ell_2' + 2$ sleeping particles all along $[\tilde\tau_{k_0}, \tilde\sigma_{k_0 + 1}]$, and $\ell_1' \ell_2' + 1$ sleeping particles at time $\tilde\sigma_{k_0 + 1}$.

Since the free particle is colored red at time $\tilde\sigma_{k_0 + 1}$ and no yellow particle was produced during the time interval $[0, \tilde\sigma_{k_0 + 1}]$, the analysis of the time intervals $[\sigma_{k + 1}, \tau_{k + 1}]$ we gave to prove \eqref{langres} can be reproduced to prove $\tilde{\cal P}(k_0 + 1)$. There are two differences. One difference is that we have to distinguish between two cases at the end of the ``exit case'', when reaching an isoperimetric configuration of sleeping particles: if this configuration is $U$-irreducible, then $X$ reaches ${\cal X}_D$ in ${\cal I}(\ell_1' \ell_2' + 1) \cap {\cal X}_U$, while if not, then $X$ reaches ${\cal I}(\ell_1' \ell_2' + 1) \setminus {\cal X}_U$. Still, no yellow particle was produced during the time interval $[\tilde\sigma_{k_0 + 1}, \tilde\tau_{k_0 + 1}]$, in which we always have $\ell_1' \ell_2' + 1$ sleeping particles at least. The other difference is that we have to check ${\cal P}(k_0 + 1)$(v). To do so it suffices to note that the only case for which $X(\tilde\tau_{k_0 + 1}) \not\in {\cal X}_D$ is the ``exit case'' for which $X$ does not visit ${\cal X}_U$ during the whole time interval $[\tilde\sigma_{k_0 + 1}, \tilde\tau_{k_0 + 1})$.  Property $\tilde{\cal P}(k+1)$(vii) follows from the events $G_4'^c$, $G_3'^c$ and $\tilde{\cal P}(k+1)$-iii).


\paragraph{Case 2: $X(\tilde\tau_k) \in {\cal I}(\ell_1' \ell_2' + 1) \cap {\cal X}_U$.}

In this case the main part of the analysis is that of the time interval $[\tilde\tau_k, \tilde\sigma_{k + 1}]$. In particular, we will prove that 
$X(\tilde\sigma_{k + 1})$ belongs to ${\cal I}(\ell_1' \ell_2' + 1)^{fp}$, with a cluster made up of sleeping particles only, and there is no yellow particle at time $\tilde\sigma_{k + 1}$. After that we can conclude as in Case 1.

We first note that, by the definition of $\tilde\tau_k$, there are only sleeping particles in $\bar\Lambda$ at time $\tilde\tau_k$. Therefore we study once again the process
$$
\Delta\bar H\colon\, t \in [\tilde\tau_k, \tilde\sigma_{k + 1}) \mapsto \bar H(X(t)) - \bar H(X(\tilde\tau_k)).
$$
Similarly to the analysis we gave to prove \eqref{langres}, the events $\tilde F_{m + 1}^c$ and $A^c$ imply that the process can only oscillate between the energy levels $0$ and $U$, and has to go back to $0$ within a time $\ee^{\delta\beta}$ after each increase to $U$. Since $X(\tilde\tau_k) \in {\cal X}_U$, there is no way to free any particle without going above the energy level $U$. We therefore only have to consider the entrance case. The event $C'^c$ implies that $X$ reaches ${\cal I}(\ell_1' \ell_2' + 1)^{fp}$, with a cluster made up of sleeping particles only.

Now, if there were some yellow particle at time $\tilde\sigma_{k + 1}$, then by $\tilde{\cal P}(k)$(ii) this should have been produced at some earlier time $\tilde\sigma_{k'} < \tilde\tau_k$, leaving $\ell_1' \ell_2'$ sleeping particles. Since at time $\tilde\tau_k$ there are $\ell_1' \ell_2' + 1$ sleeping particles, we would get a contradiction with $\tilde{\cal P}(k)$(iv). It therefore remains to prove $\tilde{\cal P}(k+1)$(vii), for which we can argue as before.


\paragraph{Case 3: $X(\tilde\tau_k) \in {\cal I}(\ell_1' \ell_2' + 1) \setminus {\cal X}_U$.}

In this case, the same analysis for the time interval $[\tilde\tau_k, \tilde\sigma_{k + 1}]$ can be reproduced with a different conclusion.
On the one hand, it is now possible to free some particle with a move of cost $U$, leading to ${\cal I}(\ell_1' \ell_2')^{fp}$ at time $\tilde\sigma_{k + 1}$, with a cluster of $\ell_1' \ell_2'$ sleeping particles. One yellow particle, {\it but no more than one}, can subsequently be produced. On the other hand, it is still possible to reach ${\cal I}(\ell_1' \ell_2' + 1)^{fp}$ at time $\tilde\sigma_{k + 1}$, {\it without producing any new yellow particle}, but in this case too there is a difference with respect to Case 2: it is not true anymore that all the clusterized particles in $\bar\Lambda$ are necessarily sleeping at time $\tilde\sigma_{k + 1}$. Indeed, we cannot exclude anymore the presence of an active particle in $\bar\Lambda$ at time $\tilde\tau_k$. Also we cannot exclude with the same argument the possibility of having, at time $\tilde \sigma_{k + 1}$, $l'_1 l'_2 + 1$ sleeping particles together with a yellow free particle.  We will first prove that $Z_2^c$ implies that an eventual red or green particle at time $\tilde\tau_k$ cannot fall asleep during the time interval  $[\tilde\tau_k, \tilde\sigma_{k + 1}]$. Afterwards we will study the time interval $[\tilde\sigma_{k + 1}, \tilde\tau_{k + 1}]$ in the two cases  $X(\tilde\sigma_{k + 1}) \in {\cal I}(\ell_1' \ell_2' + 1)^{fp}$ and $X(\tilde\sigma_{k + 1}) \in {\cal I}(\ell_1' \ell_2')^{fp}$, with a cluster of $\ell_1' \ell_2'$ sleeping particles.

{\it A red or a green particle cannot fall asleep in the first time interval.}
We only have to consider the case when there is some red or green particle $i$ in $\bar\Lambda$ at time $\tilde\tau_k$. Let us call $\tilde\tau_{l^*}$ the last time $\tilde\tau_l$ before $\tilde\tau_k$ such that $X(\tilde\tau_l) \in {\cal X}_D$. Lemma~\ref{gorgonzola} implies that $X$ could not visit ${\cal X}_U$ during any time interval $[\tilde\tau_j, \tilde\sigma_{j + 1})$ for $1 \leq l^* < j \leq k$. Let us call $[\tilde\sigma_{j^*}, \tilde\tau_{j^*})$ the last time interval $[\tilde\sigma_{j}, \tilde\tau_{j})$ after $\tilde\tau_{l^*}$ and before $\tilde\tau_k$ in which $X$ visited ${\cal X}_U$.  We consider separately the cases in which such an index $j^*$ exists or not.  If $j^*$ exists, then by $\tilde{\cal P}(k)$-v) there was no red or green particle at time $\tilde\tau_{j^*}$, in particular, $j^* < k$ and, by construction, $X$ did not visit ${\cal X}_U$ during the time interval $[\tilde\tau_{j^*}, \tilde\sigma_{k + 1})$. The recurrence property to ${\cal X}_U$, which is described by $A^c$, then implies
\begin{equation}
\label{soumaintrain}
\tilde\sigma_{k + 1} - \tilde\tau_{j^*} \leq T_U \ee^{\delta\beta}.
\end{equation}
Since at time $\tilde\tau_{j^*}$ there was no red or green particle in $\bar\Lambda$, if our red or green particle $i$ at time $\tilde\tau_k$
was already in $\bar\Lambda$ at time $\tilde\tau_{j^*}$, then it was sleeping and there must have been some time $t_f$ in $[\tilde\tau_{j^*}, \tilde\tau_k)$ at which $i$ was free. If $i$ was not in $\bar\Lambda$ at time $\tilde\tau_k$, then it had to enter $\bar\Lambda$ during the time interval $[\tilde\tau_{j^*}, \tilde\tau_k)$ and, in this case too, it had to be free at some time $t_f$ in $[\tilde\tau_{j^*}, \tilde\tau_k)$. Inequality~\eqref{soumaintrain} implies that
$$
\tilde\sigma_{k + 1} - t_f \leq T_U \ee^{\delta\beta} < \ee^{D\beta},
$$
so that $i$ cannot fall asleep before time $\tilde\sigma_{k + 1}$.

If $j^*$ does not exist, then by construction we deduce that $l^*<k$ and
\begin{equation}
\label{2soumaintrain}
\tilde\sigma_{k + 1} - \tilde\sigma_{l^*+1} \leq T_U \ee^{\delta\beta}.
\end{equation}
Since all the clusterized particle in $\bar\Lambda$  at time $\tilde\sigma_{l^*+1}$ were sleeping particles, if $i$ was among them, then there was some time $t_f$ between $\tilde\sigma_{l^*+1}$ and $\tilde\tau_k$ when $i$ was free. The same conclusion obviously holds if $i$ was the free particle at time $\tilde\sigma_{l^*+1}$. Finally, if $i$ was not in $\bar\Lambda$ at time $\tilde\sigma_{l^*+1}$, then it had to enter $\bar\Lambda$ between times $\tilde\sigma_{l^*+1}$ and $\tilde\tau_k$. But in this case also it had to be free at some time $t_f$ between $\tilde\sigma_{l^*+1}$ and $\tilde\tau_k$. It follows from \eqref{2soumaintrain} that
$$
\tilde\sigma_{k + 1} - t_f \leq T_U \ee^{\delta\beta} < \ee^{D\beta},
$$
and $i$ cannot fall asleep before time $\tilde\sigma_{k+1}$.

\medskip\noindent
$\blacktriangleright$ {\it The case $X(\tilde\sigma_{k + 1}) \in {\cal I}(\ell_1' \ell_2' + 1)^{fp}$.}
If all the clusterized particles in $\bar\Lambda$ are sleeping at time $\tilde\sigma_{k + 1}$, then we can conclude as in Case 2: the entrance at $\tilde \sigma_{k + 1}$ of a yellow particle would imply either the presence of another yellow particle in $\bar \Lambda$ at time $\tilde \tau_k$, which would contradict $\tilde {\cal P}(k)$(iii), or the fact that there were only sleeping particles at $\tilde \tau_k$,
which as before would contradict either $\tilde{\cal P}(k)$(ii) or $\tilde {\cal P}(k)$(iv). Let us therefore assume that the isoperimetric cluster at time $\tilde\sigma_{k + 1}$ contains an active particle. Since there is also a free particle at time $\tilde\sigma_{k + 1}$ in $\bar\Lambda$, we have two active particles in $\bar\Lambda$. The events $G^c$, $G'^c$ and $H_2^c$ imply that at least one of them has to be yellow. Since at time $\tilde\tau_k$ there was one yellow particle at most and we did not produce any new yellow particle during the time interval $[\tilde\tau_k, \tilde\sigma_{k + 1}]$, there is at most one yellow particle. The events $G^c$, $G'^c$ and $H_2^c$ imply that, among the two active particles in $\bar\Lambda$, one is yellow and the other is either red or green, there is no other yellow particle in $\bar\Lambda^c$, and no other active particles in $[\bar\Lambda, D + \delta] \setminus \bar\Lambda$. In particular, as a consequence of $A^c$, no other particle can enter $\bar\Lambda$ before time $\tilde\tau_{k + 1}$.

Let us consider the process
$$
\Delta\bar H\colon\,t \in [\tilde\sigma_{k + 1}, \tilde\tau_{k + 1}]
\mapsto \bar H(X(t)) - \bar H(X(\tilde\sigma_{k + 1})).
$$
As a consequence of $A^c$, $\tilde C^c$ and the fact that no other particle can enter $\bar\Lambda$, this process has to decrease within a time $\ee^{\delta\beta}$. We then have a flow of alternatives organised as follows. We consider three distinct cases $a$, $b$, $c$: the first two will be conclusive, while the last can either be conclusive in three different ways or bring us to a similar but simpler and binary alternative $b'$/$c'$. Once again the first case will be conclusive, while the last case can either be conclusive in three different ways or bring us back to the same binary alternative $b'$/$c'$. It will be clear later that $Z_2^c$ will prevent us from running into an infinite loop.
\begin{description}
\item[$(a)$] 
{\it The free particle at time $\tilde\sigma_{k + 1}$ leaves $\bar\Lambda$ without interacting with any other particle in $\bar\Lambda$.}
In this case $\Delta\bar H$ first decreases to $-\Delta$, which occurs at time $\tilde\tau_{k + 1}$: $X$ reaches ${\cal I}(\ell_1' \ell_2' + 1) \setminus {\cal X}_U$ without having time to make the other active particle fall asleep. Indeed, with the same argument as used before, it is possible to prove that the eventual red or green particle cannot fall asleep during the time interval $[\tilde\tau_k,\tilde\sigma_{k+1}+\ee^{\delta\beta}]$. If the yellow particle was free at time $\tilde\sigma_{k+1}$, then at time $\tilde\tau_{k+1}$ it is outside $\bar\Lambda$. If the yellow particle was clusterized at time $\tilde\sigma_{k+1}$, then at time $\tilde\tau_{k+1}$ it is in $\bar\Lambda$. In this case the system does not visit ${\cal X}_U$ during the time interval $[\tilde\sigma_{k + 1}, \tilde\tau_{k + 1}]$.
\item[$(b)$] 
{\it $\Delta\bar H$ reaches the energy level $-2U$ before a free particle leaves $\bar\Lambda$.}
In this case we can reproduce the analysis of the good attachment case described to prove \eqref{langres}. $X$ reaches ${\cal X}_D$ in ${\cal I}(\ell_1' \ell_2' + 2)$ at time $\tilde\tau_{k + 1}$ by making fall asleep the two active particles of time $\tilde\sigma_{k + 1}$.
\item[$(c)$] 
{\it The free particle at time $\tilde\sigma_{k + 1}$ interacts with the clusterized particles and $\Delta\bar H$ does not reach the energy level $-2U$ before a free particle leaves $\bar\Lambda$.} In this case we can reproduce the analysis of the exit case described to prove \eqref{langres}, $\Delta\bar H$ will reach the energy level $-\Delta$ with the exit of a free particle from $\bar\Lambda$ and an isoperimetric configuration in $\bar\Lambda$. We note that our permutation rules ensure that at each time $t$ whenever there is a free particle after the first interaction time and before reaching the energy level $-\Delta$, it cannot be yellow. At the time $t$ of the red or green particle exit we distinguish between three cases.
\begin{itemize}
\item[(i)] 
If $X(t) \in {\cal I}(\ell_1' \ell_2' + 1) \setminus {\cal X}_U$, then $\tilde\tau_{k + 1} = t$. If some particle fell asleep before time $t$, then it was the yellow one and there is no yellow particle anymore at time $t$. If there is still some active particle in $\bar\Lambda$ at time $t$,
then it is the yellow one: there is no green or red particle in $\bar\Lambda$ at time $t$.
\item[(ii)] 
If $X(t) \in {\cal I}(\ell_1' \ell_2'+ 1) \cap {\cal X}_U$ and all particles in $\bar\Lambda$ are sleeping at time $t$, then $\tilde\tau_{k + 1} = t$. There is no yellow particle anymore at time $t$. There is no green or red particle in $\bar\Lambda$ at time $t$.
\item[(iii)] 
If $X(t) \in {\cal I}(\ell_1' \ell_2' + 1) \cap {\cal X}_U$ and the yellow particle is still active at time $t$, then $\tilde\tau_{k + 1} > t$. As in the good attachment case studied to prove \eqref{langres}, where $\Delta\bar H$ could eventually only oscillate between the two energy levels $-2U$ and $-U$, $\Delta\bar H$ can only oscillate between the energy levels $-\Delta$ and $-\Delta + U$ until the first time $t' > t$ when either the yellow particle falls asleep or the red or green particle comes back in $\bar\Lambda$. In the former case, to which we will refer as the {\it conclusive} case, $\tilde\tau_{k + 1} = t'$, there is no yellow particle anymore at time $t'$, and there is no red or green particle in $\bar\Lambda$ at time $t'$. In the latter case, considering in the same way 
$$
\Delta\bar H\colon\, s \in [t', \tilde\tau_{k + 1}] \mapsto \bar H(X(s)) - \bar H(X(t')),
$$
we are led to repeat the same kind of analysis, with one more hypothesis with respect to time $\tilde\sigma_{k + 1}$: we know that the free particle at time $t'$ is either red or green and that the clusterized active particle is the yellow one. We can then define a single alternative $(c')$ to a similar case $(b')$.
\begin{description}
\item[$(b')$] 
{\it $\Delta\bar H$ reaches the energy level $-2U$ before a free particle leaves $\bar\Lambda$.}
There is no difference in this case with the previous case $(b)$. 
\item[$(c')$] 
{\it $\Delta\bar H$ does not reach the energy level $-2U$ before a free particle leaves $\bar\Lambda$.}
This case includes a possible absence of interaction between the clusterized particles in $\bar\Lambda$ and the green or red free particle before it exits. The same conclusions hold as in the previous case $c$, with the possibility of going back to the same alternative $(b')$/$(c')$ after a similar time $t'$ when the green or red particle comes back in $\bar\Lambda$.
\end{description}
\end{itemize}
\end{description}
Since each time we go back to the alternative $(b')$/$(c')$ the green or red particle enters again $\bar\Lambda$, $\tilde B^c$ implies that it can happen a finite number of times only. Ultimately, no yellow particle can be produced during the time interval $[\tilde\sigma_{k + 1}, \tilde\tau_{k + 1}]$: if the red or the green particle falls asleep (cases $(b)$ and $(b')$), then so does the yellow one, and if the yellow particle falls asleep (cases $(b)$, $(b')$, $(c)$(ii), $(c')$(ii), conclusive $(c)$(iii) and $(c')$(iii), or $(a)$, $(c)$(i) and $(c')$(i)), then there is no yellow particle anymore at time $\tilde\tau_{k + 1}$, while if $X$ visited ${\cal X}_U$, then $a$ is excluded, which is the only case with a possible green or red particle in $\bar\Lambda$ at time $\tilde\tau_k$. We also had at least $\ell_1' \ell_2'$ sleeping particles in the whole time interval. For the proof of $\tilde{\cal P}(k+1)$(vii) we can argue as before.

\medskip\noindent
$\blacktriangleright$ {\it The case $X(\tilde\sigma_{k + 1}) \in {\cal I}(\ell_1' \ell_2')^{fp}$, with a cluster of $\ell_1' \ell_2'$ sleeping particles.}
Let us first show by contradiction that there cannot be two yellow particles at time $\tilde\sigma_{k + 1}$. Indeed, =in this case $\tilde{\cal P}(k)$(iii) would imply that we just reached ${\cal I}(\ell_1' \ell_2')^{fp}$ by producing a yellow particle $i$ during the time interval $[\tilde\tau_k, \tilde\sigma_{k + 1}]$. This is possible only if we had $\ell_1' \ell_2' + 1$ sleeping particles at time $\tilde\tau_k$. We note that we could not produce more than one yellow particle in this time interval. Hence there should have been another yellow particle $j$ produced at an earlier time $t < \tilde\tau_k$, and we can assume that $t$ was the last emission time of a yellow particle before time $\tilde\tau_k$. Since our hypothesis $\tilde{\cal P}(k)$(ii) implies that there were at most $\ell_1' \ell_2'$ sleeping particles at time $t$, some particle fell asleep between times $t$ and $\tilde\tau_k$ and this would contradict $\tilde{\cal P}(k)$(iv). 

Note that $G^c$, $G'^c$ and $H_2^c$ imply that there is either 0 or 1 particle in $[\bar\Lambda, D + \delta] \setminus \bar\Lambda$. We also note that the sleeping particles in $\bar\Lambda$ at time $\tilde\sigma_{k + 1}$ form a quasi-square: this is the only isoperimetrical configuration of size $\ell_1' \ell_2'$.

If there is no particle in $[\bar\Lambda, D + \delta] \setminus \bar\Lambda$, then, once again, $A^c$ and $\tilde C^c$ imply that the local energy first has to decrease within a time $\ee^{\delta\beta}$.This can be realized in two ways only: waiting either for the attachment of the free particle to the cluster or for the free particle to leave $\bar\Lambda$ at some time $t$. In both cases $\tilde\tau_{k + 1} = t$. In the former case $X$ goes back to ${\cal I}(\ell_1' \ell_2' + 1)$ without making any particle fall asleep and without visiting ${\cal X}_U$. In the latter case $X$ reaches ${\cal X}_D$ in ${\cal I}(\ell_1' \ell_2')$.

If there is another active particle in $[\bar\Lambda, D + \delta] \setminus \bar\Lambda$, then $A^c$ and $\tilde C^c$ together with Lemma~\ref{gex} and $\tilde F_{m + 1}^c$ lead to the same conclusion. The free particle at time $\tilde\sigma_{k + 1}$ indeed has to either leave $\bar\Lambda$ or join the cluster before the second active particle can enter $\bar\Lambda$. For the proof of $\tilde{\cal P}(k+1)$(vii) we can argue as before.


\paragraph{Case 4: $X(\tilde\tau_k) \in {\cal I}(\ell_1' \ell_2')$.}

In this case we have a quasi-square of sleeping particles at time $\tilde\tau_k$, and any move before the entrance of a free particle would cost $2U$ at least. Such a move is excluded by $\tilde F_{m + 1}^c$. It follows that $X$ reaches ${\cal I}(\ell_1' \ell_2')^{fp}$
with a cluster made up of sleeping particles only at time $\tilde\sigma_{k + 1}$, and we conclude like in the previous case. This ends our induction.


\paragraph{$\bullet$ Cost estimates.}

To complete the proof of \eqref{langres} in the subcritical case, we only need to check the given lower bounds for the cost of each event that compounds $Z_2$, for which we refer to Appendix \ref{sec:appb3}. This concludes Case (II).
\qed


\subsubsection{Escape case (III)}
\label{subsec:escapeIII}


\paragraph{$\bullet$ Large deviation events in the supercritical case.}

In the supercritical case, the cost of $Z_1$ is that of $H_2$. We will build $Z_2$ by removing $H_2$ from $Z_1$ before adding new large deviation events. The event $\tilde H_2$ in the following list is contained in $H_2$ and has a larger cost. 

\begin{description}
\item[$H_3:$] 
There are three green particles at a same time $t < T_{\Delta^+} \ee^{\delta\beta}$ in $[\bar\Lambda, D + \delta]$. This event costs at least $2(\Delta-D+\alpha)-O(\delta)$.
\item[$H_3':$] 
There are two times $t_1 < t_2 < T_{\Delta^+} \ee^{\delta\beta}$ at which there is a pair of green particles in $[\bar\Lambda, D + \delta]$ at time $t_1$ and a different pair of green particles in $[\bar\Lambda, D + \delta]$ at time $t_2$. This event costs at least $2(\Delta-D+\alpha) - O(\delta)$.
\item[$I:$] 
Within time $T_{\Delta^+} \ee^{\delta\beta}$ there are two green particles at a same time in $[\bar{\Lambda},D+\delta]$, and there is one attribution of the red color, or else the occurrence of one of the events $J^{ij}$. (Note that $I=H_2\cap(F_1\cup \bigcup_{i,j} J^{ij})$.) This event costs at least $U - \frac{1}{2} \epsilon +\frac{1}{2}\alpha -d -O(\delta)$.
\item[$\tilde{H}_2:$] 
$H_3\cup H_3'\cup I$. This event costs at least $U - \frac{1}{2} \epsilon +\frac{1}{2}\alpha -d -O(\delta)$.
\end{description}

\medskip\noindent
$\bullet$ {\bf $\boldsymbol{Z_2}$ and the escape from $\boldsymbol{{\cal G}_2}$ in the supercritical case.}
Set 
\begin{equation}
\label{Z2sup}
Z_2 = A \cup \tilde B \cup \tilde C \cup C' \cup \tilde D \cup D' \cup E
\cup  F_{m + 1} \cup G \cup G' \cup G_4' \cup G'_3 \cup \tilde{H}_2.
\end{equation}
Define $Z_2'^c = Z_2^c\cap\{\hbox{no red particles are produced}\}$ and $Z_2''^c=Z_2^c\cap\{\hbox{red particles can be produced}\}$, 
so that $Z_2^c=Z_2'^c\dot{\cup}Z_2''^c$. If $Z_2^c$ occurs, then either $Z_2'^c$ or $Z_2''^c$ occurs. If $Z_2''^c$ occurs, then, by using the event $I^c$ and arguing in a similar way as in the proof of \eqref{langres}, we obtain that $\varphi^k$ does not escape from $\mathcal{G}_1$. If $Z_2'^c$ occurs, then we define $\tilde\tau_0=\tau_0$ and $\tilde\sigma_k$, $\tau_k$ with $k>0$, as before. If there exists $1 \leq k_1\leq\tilde n$ such that at time $\tilde\sigma_{k_1}$ there are two green particles in $\bar\Lambda$, then we define $k_0=k_1-1$, otherwise we put $k_0=\tilde n$. We will analyze separately the behavior of the process $X$ up to and after time $\tilde\tau_{k_0}$, because before the appearance of two green particles in $\bar\Lambda$ no particle can be painted yellow, otherwise this is possible.

Let $\tilde{\cal{G}}'$ be the graph in Fig.~\ref{fig:tildecalGprime}.

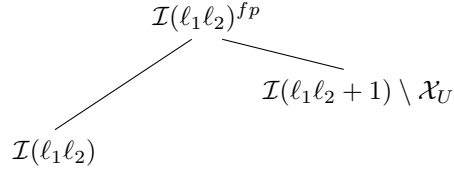
\begin{figure}[htbp]
$$
\begin{tikzpicture}
\draw (2, -.8) node {${\cal I}(\ell_1\ell_2)$};
\draw (2, -.5) -- (3.8, .7);
\draw (4, 1) node {${\cal I}(\ell_1\ell_2)^{fp}$};
\draw (4.2, .7) -- (5.8, .3);
\draw (6, 0) node {${\cal I}(\ell_1\ell_2 + 1) \setminus {\cal X}_U$};
\end{tikzpicture}
$$
\caption{The graph $\tilde{\cal{G}}'$.}
\label{fig:tildecalGprime}
\end{figure}

Recall \eqref{bleu} and \eqref{fourme}, and set
$$
\tilde\varphi'^k = (X(\tilde\tau_{0}), X(\tilde\sigma_{1}), X(\tilde\tau_{1}), \dots, X(\tilde\sigma_k), X(\tilde\tau_k)),
\qquad k \leq k_0.
$$
We will prove by induction that, for all $k \leq k_0$,
\begin{ClmPkt'}
If $Z_2$ does not occur, then
\begin{itemize}
\item[(i)] 
$\tilde\varphi'^k$ does not escape from $\tilde {\cal G}'$.
\item[(ii)] 
There is no yellow particle at each time $t \leq \tilde\tau_k$.
\item[(iii)] 
For all $0<j \leq k$, if $X$ visited ${\cal X}_U$ during the time interval $[\tilde\sigma_j, \tilde\tau_j)$, then there is no green particle in $\bar\Lambda$ at time $\tilde\tau_j$.
\item[(iv)] 
At each time $t \leq \tilde\tau_k$ there are $\ell_1 \ell_2$ sleeping particles.
\item[(v)] 
No box creation occurs within time $\tau_k$.
\end{itemize}
\end{ClmPkt'}

\medskip\noindent
$\bullet$ {\bf Proof of $\boldsymbol{\tilde{\cal P}'(k)}$, $\boldsymbol{0 \leq k \leq k_0}$.} 
Note that $\tilde{\cal P}'(0)$ is trivial. For $k\in\N_0$ we assume $\tilde{\cal P}'(k)$ to prove $\tilde{\cal P}'(k+1)$. If $k=k_0$, then there is nothing to prove, so assume that $k\neq k_0$. We separate two cases, depending on the configuration at time $\tilde\tau_k$ in one of the bottom sets of $\tilde {\cal G}'$, ordered from left to right.


\paragraph{Case 1: $X(\tilde\tau_k) \in {\cal I}(\ell_1 \ell_2)$.} 

In this case we have a quasi-square of sleeping particles at time $\tilde\tau_k$, and any move before the entrance of a free particle would cost $2U$ at least. Such a move is excluded by the fact that no red particles can be produced. It follows that $X$ reaches ${\cal I}(\ell_1 \ell_2)^{fp}$ with a cluster made up of sleeping particles only at time $\tilde\sigma_{k + 1}$. By the fact that no red particles are created and by the event $H_3^c$, we know that there are at most two active particles in $[\bar\Lambda, D+\delta]$. In particular, the active particles can be green only. 

If there is no particle in $[\bar\Lambda,D+\delta]\setminus\bar\Lambda$, then by the events $A^c$ and $\tilde{C}^c$ we know that 
the local energy must decrease within a time $\ee^{\delta\beta}$. This can be realized in two ways only: waiting either for the attachment of the free particle to the cluster or for the free particle to leave $\bar\Lambda$ at some time $t$. In both cases $\tilde\tau_{k + 1} = t$. In the former case  $X$ goes back to ${\cal I}(\ell_1 \ell_2 + 1)$ without making any particle fall asleep and without visiting ${\cal X}_U$. In the latter case $X$ reaches ${\cal X}_D$ in ${\cal I}(\ell_1 \ell_2)$.

If there is one active particle in $[\bar\Lambda,D+\delta]\setminus\bar\Lambda$, then we argue as in the subcritical case by using the events $A^c$, $\tilde{C}^c$, $I^c$ and Lemma \ref{gex}, and the fact that no red particles can be produced. Indeed, the free particle at time $\tilde\sigma_{k+1}$ has to either leave $\bar\Lambda$ or join the cluster before the second active particle enters $\bar\Lambda$.  Property $\tilde{\cal P}'(k+1)$(v) follows from the event $G_4'^c$ and $\tilde{\cal P}'(k+1)$(ii).


\paragraph{Case 2: $X(\tilde\tau_k) \in {\cal I}(\ell_1 \ell_2+1)\setminus{ \cal X}_U$.} 

We can repeat the analysis given in the subcritical case. In particular, with the same arguments we prove that the possible green particle at time $\tilde\tau_k$ cannot feel asleep during the time interval $[\tilde\tau_k,\tilde\sigma_{k+1}]$. Note that $\tilde{\cal P}'(k)$ implies that at time $\tilde\tau_k$ there is a green particle in $\bar\Lambda$. We have to analyze the time interval $[\tilde\sigma_{k+1},\tilde\tau_{k+1}]$ in the case $X(\tilde\sigma_{k+1})\in{\cal I}(\ell_1\ell_2)^{fp}$, with a cluster of $\ell_1\ell_2$ sleeping particles: it is not possible that $X(\tilde\sigma_{k+1})\in {\cal I}(\ell_1\ell_2+1)^{fp}$ because $k\leq k_0$, and therefore two green particles cannot be in $\bar\Lambda$. We can therefore argue as in the subcritical case. For the proof of $\tilde{\cal P}'(k+1)$(v) we can argue as before.

Let $\tilde {\cal G}$ be the graph in Fig.~\ref{fig:tildecalGalt}.

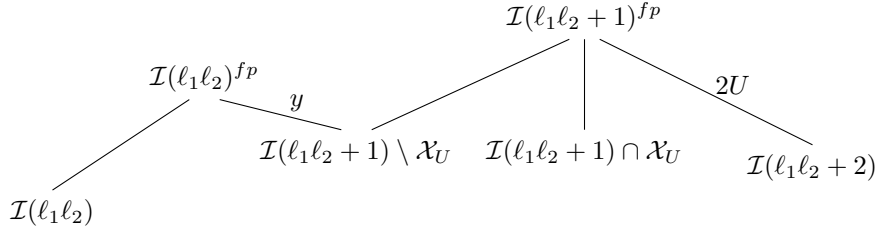
\begin{figure}[htbp]
$$
\begin{tikzpicture}
\draw (2, -.8) node {${\cal I}(\ell_1\ell_2)$};
\draw (2, -.5) -- (3.8, .7);
\draw (4, 1) node {${\cal I}(\ell_1\ell_2)^{fp}$};
\draw (4.2, .7) -- (5.8, .3);
\draw (5, .65) node [right] {$y$};
\draw (6, 0) node {${\cal I}(\ell_1\ell_2 + 1) \setminus {\cal X}_U$};
\draw (6.2, .3) -- (8.8, 1.5);
\draw (9, 0) node {${\cal I}(\ell_1\ell_2 + 1) \cap {\cal X}_U$};
\draw (9, .3) -- (9, 1.5);
\draw (9, 1.8) node {${\cal I}(\ell_1\ell_2 + 1)^{fp}$};
\draw (9.2, 1.5) -- (12, .1);
\draw (10.6, .9) node [right] {$2U$};
\draw (12, -.2) node {${\cal I}(\ell_1\ell_2 + 2)$};
\end{tikzpicture}
$$
\caption{The graph $\tilde {\cal G}$.}
\label{fig:tildecalGalt}
\end{figure}

Recall \eqref{bleu} and \eqref{fourme}, and set
$$
\tilde\varphi^k
= (X(\tilde\tau_{k_0}), X(\tilde\sigma_{k_0 + 1}), X(\tilde\tau_{k_0 + 1}), \dots, X(\tilde\sigma_k), X(\tilde\tau_k)),
\qquad k_0 < k \leq \tilde n. 
$$
We will prove by induction that, for all $k_0 < k \leq \tilde n$,
\begin{ClmPkt}
If $Z_2$ does not occur, then
\begin{itemize}
\item[(i)] 
$\tilde\varphi^k$ does not escape from $\tilde {\cal G}$.
\item[(ii)] 
Some particle can be colored yellow during the time interval $[\tilde\tau_{k_0}, \tilde\tau_k]$, but only during the climbing of the $y$-edge of $\tilde {\cal G}$.
\item[(iii)] 
There is at most one yellow particle at each time $t \leq \tilde\tau_k$.
\item[(iv)] 
At each time $1 \leq t \leq \tilde\tau_k$ when a particle falls asleep there is no yellow particle at the first $\tilde\tau_j$, $1 \leq j \leq k$, larger than or equal to $t$.
\item[(v)] 
For all $0<j \leq k$, if $X$ visited ${\cal X}_U$ during the time interval $[\tilde\sigma_j, \tilde\tau_j)$, then there is no red or green particle in $\bar\Lambda$ at time $\tilde\tau_j$.
\item[(vi)] 
At each time $t \leq \tilde\tau_k$ there are at least $\ell_1 \ell_2$ sleeping particles.
\item[(vii)] 
Each particle that is yellow at time $t_1\leq\tilde\tau_k$ was green at time $\tilde\sigma_{k_1}$ in $\bar\Lambda$.
\item[(viii)] 
If a green particle falls asleep at time $t\leq\tilde\tau_k$, then it was green at time $\tilde\sigma_{k_1}$ in $\bar\Lambda$.
\item[(ix)] 
No box creation occurs within $\tau_k$.
\end{itemize}
\end{ClmPkt}
Properties (i)-(vi) are the same as considered in the subcritical case, while we will use property (vii) to control inductively the yellow particles. In particular, we cannot exclude anymore the presence of two green particles, but we will exclude the simultaneous presence of two green particles and a yellow particle with the help of property (vii). Property (viii) will be used to prove property (vii). Property (iii) helps us to prove property (ix).

\medskip\noindent
$\trianglerighteq$
Before proving $\tilde {\cal P}(k)$, $k_0\leq k \leq \tilde{n}$, let us show that $\tilde {\cal P}(\tilde{n})$ implies that if $Z_2^c$ occurs, then $\varphi^n$ cannot escape from ${\cal G}_2$. We argue as in the subcritical case, but for $\tilde{n}=\tilde{n}_2$ there is one difference: the attribution of the red color is excluded by the event $Z_2'^c$,and therefore $\varphi^k$, $k\geq k_1$, cannot escape from ${\cal G}_2$. 

Since the cost of $Z_2$ is given by the smallest cost of its constituent components, we obtain
$$
c(Z_2) = \left\{
\begin{array}{ll}
c( F_{m + 1}) \geq r(\ell_1,\ell_2) - \Delta + \varepsilon -\gamma -O(\alpha,d) - O(\delta) & \hbox{if $\iota< 1/2$,}\\
c(G_4') \geq r(\ell_1,\ell_2) - \Delta + \gamma - O(\alpha,d) - O(\delta) & \hbox{if $\iota \geq 1/2$}.
\end{array}
\right.
$$
To prove \eqref{munster} in the supercritical case, it remains to prove $\tilde{\cal P}(\tilde n)$ and check the given cost estimates .

\medskip\noindent
$\bullet$ {\bf Proof of $\boldsymbol{\tilde{\cal P}(k)}$, $\boldsymbol{k_1 \leq k \leq \tilde n}$.} 
We have to prove ${\tilde{\cal P}(k_1)}$, and so we consider the time interval $[\tilde\tau_{k_0},\tilde\tau_{k_1}]$. By ${\tilde{\cal P}'(k_0)}$, either $X(\tilde\tau_{k_0})\in{\cal I}(\ell_1\ell_2)$ or $X(\tilde\tau_{k_0})\in{\cal I}(\ell_1\ell_2+1)\setminus{\cal X}_U$. We show by contradiction that $X(\tilde\tau_{k_0})\notin{\cal I}(\ell_1\ell_2)$. Indeed, if $X(\tilde\tau_{k_0})\in{\cal I}(\ell_1\ell_2)$, then repeating the analysis in the proof of ${\tilde{\cal P}'(k)}$ we obtain that $X(\tilde\sigma_{k_1})\in{\cal I}(\ell_1\ell_2)^{fp}$, with a cluster made up of sleeping particles only, and the free particle is green. This is in contradiction with the definition of the time $\tilde\sigma_{k_1}$. Hence $X(\tilde\tau_{k_0})\in{\cal I}(\ell_1\ell_2+1)\setminus{\cal X}_U$ with $\ell_1\ell_2$ sleeping particles and one active particle, which has to be green. We can repeat the analysis for in the subcritical case to prove that the green particle 
at time $\tilde\tau_{k_0}$ cannot fall asleep during the time interval $[\tilde\tau_{k_0},\tilde\sigma_{k_1}]$. By the definition of $\tilde\sigma_{k_1}$, we know that $X(\tilde\sigma_{k_1})\in{\cal I}(\ell_1\ell_2+1)^{fp}$, with $\ell_1\ell_2$ sleeping particles and two green particles. During the time interval $[\tilde\tau_{k_0},\tilde\sigma_{k_1}]$ no yellow particle is produced, an so there is no other particle in $[\bar\Lambda,D+\delta]$ at time $\tilde\sigma_{k_1}$. This implies that no other particle can enter $\bar\Lambda$ before time $\tilde\tau_{k_1}$. Property $\tilde{\cal P}'(k+1)$(ix) follows from the event $G_4'^c$, $\tilde{\cal P}'(k_0)$(ii) and the fact that no yellow particle is produced during the time interval $[\tilde\tau_{k_0},\tilde\tau_{k_1}]$. From now on we can argue as in the subcritical case with two differences only: we do not care about yellow particles and have to verify $\tilde{\cal P}(k_1)$(viii), which is trivial. For $k \geq k_1$ we assume $\tilde{\cal P}(k)$ to prove $\tilde{\cal P}(k + 1)$. We distinguish between four cases, depending on the configuration at time $\tilde\tau_k$ in one of the bottom sets of $\tilde{\cal G}$, ordered from left to right.


\paragraph{Case 1: $X(\tilde\tau_k) \in {\cal I}(\ell_1 \ell_2)$.}

In this case, as in the proof of ${\tilde{\cal P}'(k)}$, we have that $X(\tilde\sigma_{k+1})\in{\cal I}(\ell_1\ell_2)^{fp}$, with a cluster made up of sleeping particles only. Note that no yellow particle is produced during the time interval $[\tilde\tau_k,\tilde\sigma_{k+1}]$. By the fact that no red particle is produced and by the event $H_3^c$, we know that there are at most three active particles in $[\bar\Lambda,D+\delta]$. In particular, the free particle in $\bar\Lambda$ at time $\tilde\sigma_{k+1}$ is green.

If there is at most one particle in $[\bar\Lambda,D+\delta]\setminus\bar\Lambda$, then we can argue as in the subcritical case by using the events $A^c$, $\tilde{C}^c$, $I^c$ and Lemma \ref{gex}, and the fact that no red particles can be produced. If there are two particles in $[\bar\Lambda,D+\delta]\setminus\bar\Lambda$, then there are two green particles and one yellow particle $i$. Since no yellow particle is produced in $[\tilde\tau_k,\tilde\sigma_{k+1}]$, we know that particle $i$ was yellow at a time $t\leq\tilde\tau_k$. Thus, by ${\tilde{\cal P}(k)}$(vii) we know that $i$ was green at time $\tilde\sigma_{k_1}$ in $\bar\Lambda$. This is in contradiction with the event $H_3'^c$, so this case is not admissible.

${\tilde{\cal P}(k+1)}$(i)--(vi) follow by applying the same argument as in the subcritical case. We do not need to check ${\tilde{\cal P}(k+1)}$(vii)--(viii), because during the time interval $[\tilde\tau_k,\tilde\tau_{k+1}]$ no yellow particle is produced and no green particle falls asleep. Property $\tilde{\cal P}'(k+1)$(ix) follows from the events $G_4'^c$, $G_3'^c$ and $\tilde{\cal P}'(k+1)$(iii).


\paragraph{Case 2: $X(\tilde\tau_k) \in {\cal I}(\ell_1 \ell_2 + 1)\setminus {\cal X}_U$.}
 
We can repeat the analysis given for the subcritical case. In particular, with the same argument we are able to prove that the eventual green particle at time $\tilde\tau_k$ cannot fall asleep during the time interval $[\tilde\tau_k,\tilde\sigma_{k+1}]$, and we have to study 
the time interval $[\tilde\sigma_{k+1},\tilde\tau_{k+1}]$ in the two cases $X(\tilde\sigma_{k + 1}) \in {\cal I}(\ell_1 \ell_2)^{fp}$, with a cluster of $\ell_1 \ell_2$ sleeping particles, and $X(\tilde\sigma_{k + 1}) \in {\cal I}(\ell_1 \ell_2 + 1)^{fp}$.

\medskip\noindent
{\it The case $X(\tilde\sigma_{k + 1}) \in {\cal I}(\ell_1 \ell_2)^{fp}$, with a cluster of $\ell_1 \ell_2$ sleeping particles.}
As in the subcritical case, we can prove by contradiction that there cannot be two yellow particles at time $\tilde\sigma_{k + 1}$. By the fact that no red particle can be produced and by the event $H_3^c$, we know that there are at most three active particles in $[\bar\Lambda, D+\delta]$, and so we can conclude as in the previous case. ${\tilde{\cal P}(k+1)}$(i)--(vi) follow by applying the same argument carried out in the subcritical case. We do not need to check ${\tilde{\cal P}(k+1)}$(viii), because no particle falls asleep during the time interval $[\tilde\tau_k,\tilde\tau_{k+1}]$. To check ${\tilde{\cal P}(k+1)}$(vii), we may suppose that at time $\tilde\sigma_{k+1}$ a particle $i$ is colored yellow, because otherwise there is nothing to prove. By the permutation rules, it follows that particle $i$ was sleeping before being colored yellow. Since no particle falls asleep during the time interval $[\tilde\tau_{k},\tilde\sigma_{k+1}]$, particle $i$ was sleeping at $\tilde\tau_k^-$. By ${\tilde{\cal P}(k)}$ and the fact that $X(\tilde\sigma_{k+1})\in{\cal I}(\ell_1\ell_2+1)^{fp}$ 
with two green particles, we know that particle $i$ fell asleep when it was green. Thus, ${\tilde{\cal P}(k+1)}$(vii) follows by ${\tilde{\cal P}(k)}$(viii). For the proof of $\tilde{\cal P}'(k+1)$(ix) we can argue as before.

\medskip\noindent
{\it The case $X(\tilde\sigma_{k + 1}) \in {\cal I}(\ell_1 \ell_2 + 1)^{fp}$.}
We can argue as in the subcritical case with two differences only: we do not care about red particles and have to check ${\tilde{\cal P}(k+1)}$(viii): ${\tilde{\cal P}(k+1)}$(vii) is trivial because no particle is colored yellow during the time interval $[\tilde\sigma_{k+1},\tilde\tau_{k+1}]$. We distinguish between the two following cases: If at time $\tilde\sigma_{k+1}$ the two active particles in $\bar\Lambda$ are green, then ${\tilde{\cal P}(k+1)}$(viii) follows by the event $H_3'^c$. If at time $\tilde\sigma_{k+1}$ there is one green and one yellow particles in $\bar\Lambda$, then ${\tilde{\cal P}(k+1)}$(viii) follows by ${\tilde{\cal P}(k+1)}$(vii) and the event $H_3'^c$. For the proof of $\tilde{\cal P}'(k+1)$(ix) we can argue as before.


\paragraph{Case 3: $X(\tilde\tau_k) \in {\cal I}(\ell_1 \ell_2 + 1)\cap {\cal X}_U$.} 

In this case we can be repeated the analysis in the subcritical case with two differences only: no particle can be colored red and we do not need to check ${\tilde{\cal P}(k+1)}$(vii), because no yellow particle is produced during the time interval $[\tilde\tau_k,\tilde\tau_{k+1}]$. ${\tilde{\cal P}(k+1)}$(viii) can be checked as in the previous case. For the proof of $\tilde{\cal P}'(k+1)$(ix) we can argue as before.


\paragraph{Case 4: $X(\tilde\tau_k) \in {\cal I}(\ell_1 \ell_2 + 2)$.} 

In this case $k\neq k_0$, so $\tilde{n}=\tilde{n}_2=k$, and there is nothing to prove. This ends our induction.


\paragraph{$\bullet$ Cost estimates.}

To complete the proof of \eqref{munster} in the supercritical case, we only need to check the given lower bounds for the cost of each event that compounds $Z_2$, for which we refer to Appendix \ref{sec:appb3}. This concludes case (III).
\qed

	
\subsection{Proof of Lemma \ref{lmm:exit2x2}}
\label{sub:lmm2x2}
	
Recall the definition of the union of events $Z_1$ and $Z_2$ in the subcritical case given in $\eqref{Z1}$ and $\eqref{Z2sub}$. We can check that for the escape from ${\cal G}_1$ we can argue as in the general case $\ell_2\geq3$: the cost is given by $c(F_1)\geq 2U-\Delta-\alpha-O(\delta)$. For the escape from ${\cal G}_2$, again the proof of \eqref{langres} shows that $(Z_1 \setminus F_{1})^c$ implies that either $\varphi^n$ does not escape from ${\cal G}_1$ or there is a first return time $\tau_{k_0}$ such that $X(\tau_{k_0}) \in {\cal I}(4)$ and a particle is colored red at time $\sigma_{k_0 + 1}$. Set 
$$
\bar Z_2 = Z_2 \cup K_1 \cup K_2,
$$
where $K_1$ and $K_2$ are the following new large deviation events:
\begin{description}
\item[$K_1:$] 
There are three active particles, which can be green or red, together with one yellow particle in a box of volume $\ee^{D\beta}$ inside the box $[\bar\Lambda,\Delta+\delta]$ at a same time $t\in [t^*,T_{\Delta^+} \ee^{\delta \beta}]$ such that $X(t^*)\in{{\cal I}(0)}$ and at time $t^*$ the yellow particle is inside $[\bar\Lambda,D+\delta]$. This event costs at least $\Delta-D+\alpha-O(\delta)$.
\item[$K_2:$] 
There are two active particles, which can be green or red, together with two yellow particles in a box of volume $\ee^{D\beta}$ inside $[\bar\Lambda,\Delta+\delta]$ at a same time $t\in [t^*,T_{\Delta^+} \ee^{\delta \beta}]$ such that $X(t^*)\in{{\cal I}(0)}$ and at time $t^*$ the two yellow particles are inside $[\bar\Lambda,D+\delta]$. This event costs at least $\Delta-D+\alpha-O(\delta)$.
\end{description}
By defining 
\begin{equation}
\bar n = \min\{ n,  n^*\}
\end{equation}
with
\begin{equation}
n^* = \min\{k > k_0 : X(\tau_k) \in {\cal I}(4)\}
\end{equation}
and
$$
\bar\varphi^k = (X(\tau_{k_0}), X(\tau_{k_0 + 1}), \dots, X(\tau_k)), \qquad k_0 \leq k \leq \bar n,
$$
we will prove by induction that, for all $k_0\leq k \leq \bar n$,
\begin{ClmPktbar}
If $\bar Z_2$ does not occur, then
\begin{itemize}
\item[(i)] 
$\bar\varphi^k$ does not escape from ${\cal G}_2$.
\item[(ii)] 
There are at most three yellow particles at each $t\leq\tau_k$.
\item[(iii)] 
If $X(\tau_k)\in{\cal I}(4)$, then at time $\tau_k$ there is no yellow particle.
\item[(iv)] 
If $X(\tau_k)\in{\cal I}(0)$, then at time $\tau_k$ there are three yellow and one red particles in $[\bar\Lambda,D+\delta]$, with two yellow particles at distance two from each other. 
\end{itemize}
\end{ClmPktbar}
Property (i) is the main property we are interested in. We will use properties (ii)--(iv) to control inductively the yellow particles, in particular, property (iii) implies that $X$ reaches ${\cal I}(4)$ by putting  tosleep all the yellow particles created during the time interval $[\tau_{k_0},\tau_k]$, while property (iv) implies that ${\cal I}(0)$ is reached by breaking a dimer.

\medskip\noindent
$\trianglerighteq$
Before proving $\bar{\cal P}(k)$, $k_0 \leq k \leq \bar n$, let us show that $\bar{\cal P}(\bar n)$ implies in the two cases $\bar n = n$ and $\bar n = n^*$ that if $\bar Z_2^c$ occurs, then $\varphi^n$ cannot escape from ${\cal G}_2$. If $\bar n = n$, then the claim is trivial. If $\bar n = n^*$, then $\bar{\cal P}(n)$(iii) implies that there is no yellow particle at time $\tau_{n^*}$. Using $\tilde F_{1}^c$, which excludes any $2^{nd}$ attribution of the red color, we can show by induction, as in the proof of \eqref{langres}, that $\varphi^k$, for $k \geq n^*$, cannot escape anymore from ${\cal G}_1$, subgraph of ${\cal G}_2$.
	
\medskip\noindent
$\bullet$ {\bf Proof of $\boldsymbol{\bar{\cal P}(k)}$, $\boldsymbol{k_0 \leq k \leq \bar n}$.} 
$\bar{\cal P}(k_0)$(i)-(iii) follow from the definition of $k_0$. Since $X(\tau_{k_0})\in{\cal I}(4)$,  we do not need to check $\bar{\cal P}(k_0)$(iv). For $k\geq k_0$ we assume $\bar{\cal P}(k)$ to prove $\bar{\cal P}(k+1)$. We distinguish between the two following cases.
	
\medskip\noindent
{\bf Case 1: $X(\tau_k) \in {\cal I}(4)$.} 
If $k\neq k_0$, then $\bar n=n^*=k$ and there is nothing to prove.  We only have to consider the case $k=k_0$. The definition of $\sigma_{k_0+1}$ gives $X(\sigma_{k_0 + 1})\in{\cal I}(3)^{fp}$ with the free particle colored red. Suppose that $X$ does not return to ${\cal I}(4)$ within time $\tau_{k_0+1}$. By arguing as in the general case, we deduce that the following moves occur: the red particle exits from $\bar\Lambda$, a particle is detached at cost $U$ and therefore is colored yellow, leading to the configuration ${\cal I}(2)^{fp}$. Since we are considering the time interval $[\sigma_{k_0+1},\tau_{k_0+1}]$ and the times $\tau_i$ are return times to ${\cal X}_D$, by the recurrence property to ${\cal X}_D$ implied by $A^c$ we deduce that no particle can exit from $[\bar\Lambda,D+\delta]$ before time $\tau_{k_0+1}$, in particular, this holds for the red particle. Thus, by the event $G'^c$, no green particle can enter $[\bar\Lambda,D+\delta]$. Afterwards, the free particle exits from $\bar\Lambda$ and two yellow particles are created  after breaking the dimer at time $t$: $X$ reaches ${\cal I}(0)$ at time $t=\tau_{k_0+1}$. By the previous observations it easy to check $\bar{\cal P}(k_0+1)$(i),(ii),(iv), while we do not need to check $\bar{\cal P}(k_0+1)$(iii). If $X$ returns in ${\cal I}(4)$ at time $t$, then we have to prove that $t=\tau_{k_0+1}$ because we are analyzing the time interval $[\tau_{k_0},\tau_{k_0+1}]$. By arguing as in the general case, we deduce that the only possibility, possibly after visiting ${\cal I}(2)$ and ${\cal I}(3)$ several times, is to reach ${\cal X}_D$ in ${\cal I}(4)$. Since no particle can enter and exit from $[\bar\Lambda,D+\delta]$ within time $\tau_{k_0+1}$, $\bar{\cal P}(k_0+1)$(ii),(iii) follow.

\medskip\noindent
{\bf Case 2: $X(\tau_k) \in {\cal I}(0)$.} 
This part of the proof is directly related to $\bar{\cal P}(k)$(iv) and the new events $K_1$ and $K_2$. Indeed, $\bar{\cal P}(k)$(iv) gives us control on the distance between the two nearest yellow particles in $[\bar\Lambda,D+\delta]$ and the green particles, which are outside the box $[\bar\Lambda,D+\delta]$ by the event $G'^c$. By $\bar{\cal P}(k)$(iv) and the events $K_1^c$ and $K_2^c$, we deduce that, if a cluster is formed, then it has to be created by attaching the three yellow particles together with one red or green particle, so $X(\tau_{k+1})\in{\cal I}(4)$ and properties (ii) and (iii) follow. The claim follows after checking the given cost estimate for the events $K_1$ and $K_2$, for which we refer to Appendix \ref{sec:appb3}.
\qed
	
	
\subsection{Proof of Lemma \ref{gex}}
\label{sec:appb1}
Let us assume that $T^{ij} \leq T_{\Delta^+} \ee^{\delta\beta}$ and there is no such time $t \leq T^{ij}$ with three active particles inside $[\bar\Lambda, D + \delta]$. Since at time $t = 0$ either both particles $i$ and $j$ belong to a same unique cluster in $\bar\Lambda$ or at least one is outside $[\bar\Lambda, D + \delta]$, by setting
$$
{\cal T}_0 = \sup\left\{t \leq T^{ij}\colon\,\hbox{$i$ or $j$ is outside $[\bar\Lambda, D + \delta]$ 
or both are in a same unique cluster in $\bar\Lambda$ at time $t$}\right\},
$$
we see that $0 \leq {\cal T}_0 \leq T^{ij}$. We distinguish between two cases.
\begin{itemize}
\item[(i)] 
If there are no active particles, but $i$ or $j$ are inside $[\bar\Lambda, D + \delta]$ during the whole time interval $[{\cal T}_0, T^{ij}]$, then ${\cal T}_0$ is the last $\theta^{ij}_k$ before $T^{ij}$.
\item[(ii)] 
If there is some other active particle inside $[\bar\Lambda, D + \delta]$ at some time $t$ in $[{\cal T}_0, T^{ij}]$, then we set 
$$
{\cal T}_1 = \sup\left\{t \leq T^{ij}\colon\,\hbox{there is an active particle distinct from $i$ and $j$ 
inside $[\bar\Lambda, D + \delta]$ at time $t$}\right\}.
$$
Since we assumed that there is no time $t \leq T^{ij}$ at which three active particles are inside $[\bar\Lambda, D + \delta]$, $i$ or $j$ must be sleeping at time ${\cal T}_1$ and ${\cal T}_1 \leq T^{ij}$. ${\cal T}_1$ is then the last $\theta^{ij}_k$ before $T^{ij}$.
\end{itemize}
In both cases there is a last $\theta^{ij}_k \geq 0$ before $T^{ij}$ such that any active particle in $[\bar\Lambda, D + \delta]$ during the time interval $[\theta^{ij}_k, T^{ij}]$ is either $i$ or $j$.
\qed
	
	
\subsection{Proof of Lemma \ref{gorgonzola}}
\label{sec:appb2}
	
We prove the claim by contradiction. Assume that $X(\tilde\tau_k) \not\in {\cal X}_U$ andthat  there is a $t \in [\tilde\tau_k, \tilde\sigma_{k + 1})$ such that $X(t) \in {\cal X}_U$. Then there is a constant-cluster-size path from $\bar X(t)$ to the isoperimetric configuration $\bar X(\tilde\tau_k)$. By the recurrence property to ${\cal X}_U$ implied by $A^c$, we may also assume that $t - \tilde\tau_k \leq T_U \ee^{\delta\beta}$. Then $\tilde D^c$ implies that the local energy along this path does not exceed $\bar H(X(\tilde\tau_k)) + U$. Since $\bar X(\tilde\tau_k)$ is isoperimetric, we also have $\bar H(X(t)) \geq \bar H(X(\tilde\tau_k))$. Since $\bar X(\tilde\tau_k)$ is $U$-reducible, we get a contradiction with the fact that $\bar X(t)$ is $U$-irreducible.
\qed
	

\begin{appendices}


\section{Environment estimates}
\label{sec:appa}
		
In this appendix we prove that $\mu_{{\cal R}'}(({\cal X}_i^*)^c) = \SES(\beta)$ for $i=1,\ldots,5$, where, for $\eta\in{\cal X}_\beta$,
$$
\mu_{{\cal R}'}(\eta) = \dfrac{\ee^{-\beta[H(\eta)+\Delta|\eta|]}}
{Z_{{\cal R}'}}\mathbb{1}_{{\cal R}'}(\eta), \qquad
Z_{{\cal R}'}=\displaystyle\sum_{\h\in{\cal R'}}\ee^{-\beta[H(\eta)+\Delta |\eta|]}.
$$
First, we consider the case $\Delta < \Theta \leq \theta$. Given a configuration $\eta\in{\cal X}_\beta$, we denote by ${\cal C}={\cal C}(\eta)$ its connected component with maximal volume when it is unique. Otherwise, we pick the component containing the highest particle in the lexicographic order. For $C \subset \Lambda_\beta$, we set $\bar{C} = {C}\cup\partial^+{C}$, where $\partial^+{C}$ denotes the external boundary of ${C}$. We start by showing that there exists a $c>0$ such that
$$
\mu_{{\cal R}'}({\cal X}_\beta\setminus{\cal X}_i^*)
\leq \ee^{c\beta}\mu_{{\cal R}}({\cal X}_\beta\setminus{\cal X}_i^*).
$$
To this end, given a finite set $\Lambda\subset\Lambda_\beta$ and two configurations $\eta_{\Lambda}\in\{0,1\}^{\Lambda}$ and $\eta_{\Lambda_\beta\setminus\Lambda} \in \{0,1\}^{\Lambda_\beta\setminus\Lambda}$, we denote by $\eta=\eta_{\Lambda}\cdot \eta_{\Lambda_\beta\setminus\Lambda}\in\{0,1\}^{\Lambda_\beta}$ the configuration defined by
$$
\eta(x) =
\begin{cases}
\eta_{\Lambda}(x) &\hbox{if } x\in\Lambda, \\
\eta_{\Lambda_\beta\setminus\Lambda}(x) 
&\hbox{if } x\in\Lambda_\beta\setminus\Lambda.
\end{cases}
$$
Given a configuration $\sigma\in\{0,1\}^{\Lambda_\beta}$, we introduce the measure $\mu_{{\cal R},\Lambda,\sigma}$ on $\{0,1\}^{\Lambda}$ defined by
$$
\mu_{{\cal R},\Lambda,\sigma}(\eta_{\Lambda})
= \dfrac{1}{Z_{{\cal R},\Lambda,\sigma}}\ee^{-\beta[H(\eta_{\Lambda}\cdot\sigma_{\Lambda_\beta\setminus\Lambda})
+ \Delta(|\eta_{\Lambda}|+|\sigma_{\Lambda_\beta\setminus\Lambda}|)]}
\mathbb{1}_{{\cal R}}(\eta_{\Lambda}\cdot\sigma_{\Lambda_\beta\setminus\Lambda}),
$$
where $Z_{{\cal R},\Lambda,\sigma}$ is the normalizing constant. For any finite $\Lambda\subset\Lambda_\beta$ and any configuration $\eta\in\{0,1\}^{\Lambda_\beta}$, the DLR equation for the measure $\mu_{{\cal R}}$ reads
$$
\mu_{{\cal R}}(\eta) = \sum_{\sigma\in\Lambda_\beta}
\mu_{{\cal R}}(\sigma)\mu_{{\cal R},\Lambda,\sigma}(\eta_{|\Lambda}).
$$
Since a cluster with volume at most $\lambda(\beta)/8$ has perimeter at most $\lambda(\beta)$, and therefore is contained in a box of volume $\lambda^2(\beta)$, it can be arranged inside $\Lambda_\beta$ in at most $2^{\lambda^2(\beta)}$ different ways and in at most $\ee^{\Theta\beta}$ different location. Hence
$$
\begin{aligned}
\mu_{{\cal R}'}({\cal X}_\beta\setminus{\cal X}_i^*)
&= \dfrac{\displaystyle\sum_{\eta\in{\cal R}'\setminus {\cal X}_i^*}\ee^{-\beta[H(\eta)+\Delta|\eta|]}}
{\displaystyle\sum_{\eta\in{\cal R}'}\ee^{-\beta[H(\eta)+\Delta|\eta|]}} \leq
\displaystyle\sum_{\substack{C\subset\Lambda_\beta \\ |C|\leq \lambda^2(\beta)}}
\dfrac{\displaystyle\sum_{\substack{\eta\in{\cal R}'
\setminus{\cal X}_i^* \\ {\cal C}=C}}\ee^{-\beta[H(\eta)+\Delta|\eta|]}}
{\displaystyle\sum_{\substack{\eta\in{\cal R}' \\ {\cal C}=C}} \ee^{-\beta[H(\eta)+\Delta|\eta|]}} \\
&\leq \displaystyle\sum_{\substack{C\subset\Lambda_\beta \\ 
|C|\leq \lambda^2(\beta)}} \dfrac{\ee^{-\beta[H(C)+\Delta|C|]}
\displaystyle\sum_{\eta\in{\cal R}\setminus{\cal X}_i^*}
\ee^{-\beta[H(\eta)+\Delta|\eta|]}}{\ee^{-\beta[H(C)+\Delta|C|]}
\displaystyle\sum_{\substack{\eta\in{\cal R} \\ 
|\eta_{|\bar C}|=0}}\ee^{-\beta[H(\eta)+\Delta|\eta|]}} \\
&\leq 2^{\lambda^2(\beta)}\ee^{\Theta\beta}
\dfrac{\mu_{{\cal R}}({\cal X}_\beta\setminus{\cal X}_i^*)}
{\displaystyle\min_{\substack{C\subset\Lambda_\beta \\ 
|C|\leq \lambda^2(\beta)}}\frac{1}{Z_{{\cal R}}}
\sum_{\substack{\eta\in{\cal R} \\ |\eta_{|\bar C}|=0}}
\ee^{-\beta[H(\eta)+\Delta|\eta|]}}
\leq \ee^{c\beta} \mu_{{\cal R}({\cal X}_\beta\setminus{\cal X}_i^*)},
\end{aligned}
$$
where in the last step we use the DLR equation and the fact that, for any configuration $\eta\in{\cal R}$, the probability of having $|\eta_{|\bar C}|=0$ is at least $1-\ee^{-(\Delta-\delta)\beta}$ for any $\delta>0$ and $\beta$ large enough, uniformly in the boundary conditions.

\medskip\noindent
$\bullet$ \underline{$i=1$}.		
Recall that, for $\eta\in{\cal X}_\beta$, $\eta^{cl}$ is the union of the connected components of size at least two, so that $|\eta\setminus\eta^{cl}|$ denotes the number of connected components
that are reduced to single particles. We get
\begin{align}
\mu_{{\cal R}}({\cal X}_\beta\setminus{\cal X}_1^*)
&\leq \dfrac{1}{Z_{{\cal R}}}\displaystyle\sum_{k=0}^{\ee^{\theta\beta}}
\sum_{\substack{\eta\in{\cal R}\setminus{\cal X}_1^* \\ 
|\eta\setminus\eta^{cl}|=k}}\ee^{-\beta[H(\eta)+\Delta|\eta|]}
\leq \dfrac{1}{Z_{{\cal R}}} \Bigl(\ee^{-(2\Delta-U)\beta}\ee^{\theta\beta}\Bigr)^{\lambda(\beta)}
\displaystyle\sum_{k=0}^{\ee^{\theta\beta}}
\sum_{\substack{\eta\in{\cal R} \\ \eta^{cl}=\emptyset, |\eta|=k}}
\ee^{-\beta[H(\eta)+\Delta|\eta|]} \notag \\ 
&\leq \dfrac{Z_{{\cal R}}}{Z_{{\cal R}}}\ee^{-(2\Delta-U-\theta)\beta\lambda(\beta)}
= \SES(\beta), \notag
\end{align}
where we use that $\theta<2\Delta-U$.

\medskip\noindent
$\bullet$ \underline{$i=2$}.
Note that ${\cal X}_\beta\setminus{\cal X}_2^*$ implies that the number of disjoint quadruples of particles with diameter smaller than $\sqrt{\ee^{S\beta}}$ is at least $(\lambda^{1/4}(\beta))/4$. Given $k=\lambda^{1/4}(\beta)/4$ and a collection $x=(x_i^j)_{i<4,j<k} \in \Lambda_\beta^{4\times k}$, we define the set 
$$
\Lambda_x=\bigcup_{\substack{i<4 \\ j<k}} B(x_i^j,\ell_c^2).
$$
Using the DLR equation, we obtain
$$
\begin{aligned}
\mu_{{\cal R}}({\cal X}_\beta\setminus{\cal X}_2^*)
&\leq \displaystyle\sum_{\substack{x_0^0,\ldots,x_3^0\in\Lambda_\beta \\ 
\hbox{diam}\{x_i^0,i<4\}<\ee^{S\beta/2}}} \cdots \sum_{\substack{x_0^{k-1},\ldots,x_3^{k-1} \in \Lambda_\beta \\
\hbox{diam}\{x_i^{k-1},i<4\}<\ee^{S\beta/2}}} \sum_{\sigma\in\{0,1\}^{\Lambda_\beta}}
\mu_{{\cal R}}(\sigma)\mu_{{\cal R},\Lambda_x,\sigma}
\left(\begin{aligned}
\hbox{the sites in } x \\
\hbox{are occupied} 
\end{aligned}
\right) \\
&\leq \Bigl(\ee^{(3S-4\Delta+\theta)\beta}\Bigr)^{\frac{\lambda^{1/4}(\beta)}{4}} = \SES(\beta),
\end{aligned}
$$
where $S=\frac{4\Delta-\theta}{3}-\alpha$.
		
\medskip\noindent
$\bullet$ \underline{$i=3$}.
Let $S<A<\Delta$ and divide the box $\Lambda_{\beta}$ into $\ee^{(3A-4\Delta+\Theta+3\alpha)\beta}$ boxes of volume $\ee^{(4\Delta-3A-3\alpha)\beta}$. Note that ${\cal X}_\beta\setminus{\cal X}_3^*$ implies that there exists one box containing at least $(\ee^{\alpha\beta/4})/4$ disjoint quadruples of particles with diameter smaller than $\sqrt{\ee^{A\beta}}$. Using the DLR equation and arguing as above, we get
$$
\begin{array}{ll}
\mu_{{\cal R}}({\cal X}_\beta\setminus{\cal X}_3^*)
&\leq \ee^{(3A-4\Delta+\Theta+3\alpha)\beta}
\Big(\ee^{-4\Delta\beta}\ee^{(4\Delta-3A-3\alpha)\beta}
\displaystyle\prod_{i=1}^3 (\ee^{A\beta}-5i)\Big)^{\frac{\ee^{\frac{\alpha\beta}{4}}}{4}}  \\
&\leq \ee^{(3A-4\Delta+\theta+3\alpha)\beta}
\ee^{-\frac{3}{4}\alpha\beta \ee^{\frac{\alpha\beta}{4}}} = \SES(\beta).
\end{array}
$$
		
\medskip\noindent
$\bullet$ \underline{$i=4$}.
Note that ${\cal X}_\beta\setminus{\cal X}^*_4$ implies that there exists a box of volume $\ee^{(\Delta+\alpha)\beta}$ containing either at least $\ee^{\frac{3}{2}\alpha\beta}$ or at most $\ee^{\frac{1}{2}\alpha\beta}$ particles. We consider these cases separately. Concerning the former case, by dividing the box of volume $\ee^{(\Delta+\alpha)\beta}$ into $\ee^{\frac{5}{4}\alpha\beta}$ boxes of volume $\ee^{(\Delta-\frac{\alpha}{4})\beta}$, we have that there exists a box containing at least $\ee^{\frac{\alpha}{4}\beta}$ particles. Concerning the latter case, by dividing the box of volume $\ee^{(\Delta+\alpha)\beta}$ into $\ee^{\frac{\alpha}{2}\beta}$ boxes of volume $\ee^{(\Delta+\frac{\alpha}{2})\beta}$, we have that there exists a box containing no particle. We proceed to estimate the denominator in this latter case by considering all the configurations with one particle in each box of volume $\ee^{(\Delta+\frac{\alpha}{4})\beta}$ inside a box of volume $\ee^{(\Delta+\frac{\alpha}{2})\beta}$, namely, these boxes are $\ee^{\frac{\alpha}{4}\beta}$. Using the DLR equation and arguing as above, we get
$$
\begin{array}{ll}
\mu_{{\cal R}}({\cal X}_\beta\setminus{\cal X}_4^*)
\leq \ee^{\frac{5}{4}\alpha\beta}\Big(\ee^{-\Delta\beta}
\ee^{(\Delta-\frac{\alpha}{4}\beta)}\Big)^{\ee^{\frac{\alpha}{4}\beta}} + \ee^{\frac{\alpha}{2}\beta} 
\dfrac{1}{\Bigl(\ee^{-\Delta\beta}\ee^{(\Delta+\frac{\alpha}{4})\beta}\Bigr)^{\ee^{\frac{\alpha}{4}}\beta}}
= \SES(\beta).
\end{array}
$$
	
\medskip\noindent
$\bullet$ \underline{$i=5$}.	
Using the DLR equation and arguing as above, we get
$$
\mu_{{\cal R}}({\cal X}_\beta\setminus{\cal X}_5^*) \leq \ee^{-\Delta\beta\frac{\lambda(\beta)}{4}}
\prod_{i<\frac{\lambda(\beta)}{4}} (\ee^{(\Delta-\frac{\alpha}{4})\beta}-5i)
\leq \ee^{-\frac{\alpha}{4}\beta\frac{\lambda(\beta)}{4}} = \SES(\beta).
$$
		
To conclude, consider the case $\Theta>\theta$. Dividing $\Lambda_\beta$ into boxes of volume $\ee^{\theta\beta}$ and arguing as above with the help of the DLR equation, we get the claim.

		
\section{Cost of large deviation events}
\label{sec:appb3}
		
\medskip\noindent
{\it Event $A$.} The cost of event $A$ follows from Proposition \ref{prop:rec} and \cite[Theorem 3.2.3]{GdHNOS09}.
		
\medskip\noindent
{\it Event $B$.} 
Each special time except $\tau_k$ is related to a free particle that moves in $\bar\Lambda$, but the number of special times $\tau_k$ is equal to the one of $\sigma_k$ by definition. The claim follows after arguing as in the proof of Proposition \ref{prp:lbtipico}: at each special time each free particle has a non exponentially small probability to avoid the box after leaving it, so that it visits this box $\ee^{\delta\beta}$ times with a super-exponentially small probability. Since, by non-superdiffusivity, the special times are associated with no more than $\ee^{(3\alpha/2+\delta)\beta}$ particles up to a $\SES(\beta)$-event, $B$ occurs with probability $1-\SES(\beta)$.
		
\medskip\noindent
{\it Event $C$.} 
Let $K$ denote the number of special times. By the event $B$, we have $K \leq \ee^{\delta\beta}$ with probability $1-\SES(\beta)$. Let $S_0,\ldots,S_{K-1}$ be the special times. Divide the time interval $[0,T_{\Delta^+}\ee^{\delta\beta}]$ into intervals $[t_i,t_i +\ee^{\delta\beta}]$ of length $\ee^{\delta\beta}$, with $1 \leq i<\ee^{(\Delta+\alpha)\beta}$. Introduce the following events: $C_1^i = \{\exists j \in \{0,\ldots,K-1\} \ \hbox{such that } S_j\in{[t_i,t_i+\ee^{\delta\beta}]}\}$ and $C_2^i=\{\hbox{there is a move of cost} \geq U \hbox{ in } [S_j,t_i+\ee^{\delta\beta}]\}$. Using the strong Markov property at the stopping time $S_j$, we obtain 
$$
\begin{array}{ll}
P(C)\leq\displaystyle\sum_{i<\ee^{(\Delta+\alpha)\beta}} 
P(C_2^i|C_1^i)P(C_1^i)\leq \ee^{-U\beta}\ee^{\delta\beta}\displaystyle
\sum_{i<\ee^{(\Delta+\alpha)\beta}} P(C_1^i) \leq \ee^{-U\beta}\ee^{O(\delta)\beta}
\end{array}
$$
and therefore $c(C)\geq U-O(\delta)$.
		
\medskip\noindent
{\it Event $C'$.} 
We control the cost of this event as for the event $C$ by using, instead of the strong Markov property, the independence of the dynamics of particles outside $\bar\Lambda$ from the marks used in $\bar\Lambda$.
		
\medskip\noindent
{\it Event $D$.} 
Divide the time interval $[0,T_{\Delta^+}\ee^{\delta\beta}]$ into intervals $[t_i,t_i+\ee^{D\beta}]$ of length $\ee^{D\beta}$, with $1 \leq i<\ee^{(\Delta+\alpha-D)\beta}\ee^{\delta\beta}$ and argue as for the event $C$.
		
\medskip\noindent
{\it Event $D'$.} 
Divide the time interval $[0,T_{\Delta^+}\ee^{\delta\beta}]$ into intervals $[t_i,t_i+\ee^{D\beta}]$ of length $\ee^{D\beta}$, with $1 \leq i<\ee^{(\Delta+\alpha-D)\beta}\ee^{\delta\beta}$ and argue as for the event $C$.
		
\medskip\noindent
{\it Event $E$.} 
Divide the time interval $[0,T_{\Delta^+}\ee^{\delta\beta}]$ into intervals $[t_i,t_i+\ee^{\delta\beta}]$ of length $\ee^{\delta\beta}$, with $1 \leq i<\ee^{(\Delta+\alpha)\beta}$. For $t_1$, $t_2$ and $k=|\bar\eta_0|$ fixed, by defining $\bar{\cal X}_k=\{\bar\eta\in\{0,1\}^{\bar\Lambda}| \ |\bar\eta|=k\}$ and using the reversibility of the measure $\mu$, we obtain
$$
\begin{array}{ll}
P_{\eta_0}(\bar{H}(\bar{X}(t_2-t_1))\geq\bar{H}(\bar{X}(\bar\eta_0))+3U)
&\leq P_{\bar\eta_0}(\bar{H}(\bar{X}(t_2-t_1))\geq\bar{H}(\bar{X}(\bar\eta_0))+3U) \\
&\leq \ee^{-3U\beta}\displaystyle\sum_{\bar\eta\in\bar{\cal X}_k \atop H(\bar\eta)
\geq H(\bar\eta_0)+3U} P_{\eta}(\bar{X}^k(t)=\bar\eta_0)\leq \ee^{-3U\beta}\ee^{\delta\beta}.
\end{array}
$$
Because the temporal entropy is $\ee^{\delta\beta}$ for $t_1$ and $\Delta^+$ for $t_2$, we get follows $c(E)\geq 3U-\Delta-\alpha-O(\delta)$.
		
\medskip\noindent
{\it Event $F_{m+1}$.} 
Divide the time interval $[0,T_{\Delta^+}\ee^{\delta\beta}]$ into intervals $[t_i,t_i+\ee^{D\beta}]$ of length $\ee^{D\beta}$, with $1 \leq i<\ee^{(\Delta+\alpha-D)\beta}\ee^{\delta\beta}$. First consider the event $F_1$, to obtain
$$
\begin{array}{ll}
P(F_1)
&\leq\displaystyle\sum_{1 \leq i<\ee^{(\Delta+\alpha-D)\beta}\ee^{\delta\beta}}
\Big[P(\hbox{move of cost } 2U \hbox{ in } [s_i,s_i+\ee^{D\beta}])\\
&\quad+P(\hbox{move of cost } U \hbox{ at time t} \in{[s_i,s_i+\ee^{D\beta}]} 
\hbox{ and at time t' such that } t<t' \hbox{ are } \delta\hbox{-close})\Big]\\
&\leq \ee^{(\Delta+\alpha-2U)\beta}\ee^{O(\delta)\beta},
\end{array}
$$
which implies $c(F_1)\geq2U-\Delta-\alpha-O(\delta)$. We can easily compute the cost of the event $F_{m+1}$ by applying the strong Markov property at the stopping times related to each attribution of the red color.
		
\medskip\noindent
{\it Event $H_2$.} 
Divide the time interval $[0,T_{\Delta^+}\ee^{\delta\beta}]$ into intervals $[t_i,t_i+\ee^{D\beta}]$ of length $\ee^{D\beta}$, with $1 \leq i<\ee^{(\Delta+\alpha-D)\beta}\ee^{\delta\beta}$. We obtain
$$
\begin{array}{ll}
P(H_2)
&\leq\displaystyle\sum_{1 \leq i<\ee^{(\Delta+\alpha-D)\beta}\ee^{\delta\beta}}
P(\hbox{there are two green particles in } [\bar{\Lambda},D+2\delta] \hbox{ at time } (i+1)\ee^{D\beta}) \\
&\leq\displaystyle\sum_{1 \leq i<\ee^{(\Delta+\alpha-D)\beta}\ee^{\delta\beta}}
\Bigg(\frac{\ee^{(D+2\delta)\beta}\ee^{\delta\beta}}{\ee^{(\Delta+\alpha)\beta}}\Bigg)^2
\leq \ee^{(D-\Delta-\alpha)\beta}\ee^{O(\delta)\beta},
\end{array}
$$
which implies $c(H_2)\geq \Delta-D+\alpha-O(\delta)$. Note that we use the spread-out property on time scale $T_{\Delta^+}$  for the green particle because this cannot reach $[\bar\Lambda,D+2\delta]$ on time scale $\ee^{D\beta}$.
		
\medskip\noindent
{\it Event $G$.} 
For a particle $i$ that is colored red at time $S_j$, applying the spread-out property and the strong Markov property at time $S_j$, we get
$$
P(\xi_i(t)\in[\bar\Lambda,D+\delta])\leq \mathbb{E}\Big[\frac{\ee^{(D+\delta)\beta}}{t-S_j}\wedge1\Big]
\leq \ee^{O(\delta)\beta}\int_0^t \dd s\,\ee^{-2U\beta}\Big(\frac{\ee^{D\beta}}{t-s}\wedge 1\Big) \leq \ee^{-(2U-D-O(\delta))\beta}.
$$
Dividing the time interval $[0,T_{\Delta^+}\ee^{\delta\beta}]$ into intervals $[t_i,t_i+\ee^{D\beta}]$ of length $\ee^{D\beta}$, with $1 \leq i<\ee^{(\Delta+\alpha-D)\beta}\ee^{\delta\beta}$, we get $c(G)\geq U-d+\epsilon-\alpha-O(\delta)$.
		
\medskip\noindent
{\it Event $G'$.} 
Divide the time interval $[0,T_{\Delta^+}\ee^{\delta\beta}]$ into intervals $[t_i,t_i+\ee^{D\beta}]$ of length $\ee^{D\beta}$, with $1 \leq i<\ee^{(\Delta+\alpha-D)\beta}\ee^{\delta\beta}$. Arguing as for the events $H_2$ and $G$, we deduce that $c(G')\geq U-d-\alpha-O(\delta)$. 
		
\medskip\noindent
{\it Event $G_4'$.} 
In case the four particles do not come from a cluster, namely, at time $t=0$ they are outside the box $[\bar\Lambda,\Delta-\alpha]$, the cost of this event has already been computed in \eqref{eq:newbox}. Consider the case in which the four particles come from a cluster. In particular, we consider the case in which all four particles are green, otherwise the cost of the event is larger. Dividing $\Lambda_\beta$ into boxes of volume $\ee^{(D+\delta)\beta}$ and the time interval $[0,T_{\Delta^+}\ee^{\delta\beta}]$ into intervals of length $\ee^{D\beta}$, we obtain
$$
P(G_4')\leq\sum_{i<\ee^{(\theta-D-\delta)\beta}} \sum_{j<\ee^{(\Delta+\alpha-D+\delta)\beta}}
\Big(\dfrac{\ee^{(D+\delta)\beta}\ee^{\delta\beta}}{\ee^{(\Delta+\alpha)\beta}}\Big)^4
\leq \ee^{-(3\Delta-2U-\theta+3\alpha-2d)\beta}\ee^{O(\delta)\beta},
$$
where we use the spread-out property on time scale $\ee^{(\Delta+\alpha)\beta}$, because particles cannot be colored green on a shorter time scale. The cases in which there is at least one green particle and one particle not coming from a cluster can be treated in a similar way.
		
\medskip\noindent
{\it Event $\tilde{B}$.} 
The cost of this event can be computed similarly as the cost of the event $B$.
		
\medskip\noindent
{\it Event $\tilde{C}$.} 
The cost of this event can be computed similarly as the cost of the event $C$.
		
\medskip\noindent
{\it Event $G'_3$.}
We only need to consider the case concerning the presence of one yellow particle from a cluster, otherwise we reduce to a case already taken into account by $G_4'$. We can argue in a similar way as for the event $G_4'$.
		
\medskip\noindent
{\it Event $\tilde{D}$.} 
The cost of this event can be computed similarly as the cost of the event $D$.
		
\medskip\noindent
{\it Event $\tilde{F}_{m+1}$.} 
We need to estimate the cost of the occurrence of one of the events $J^{ij}$. By Proposition \ref{prp:ambito} and the event $B^c$, we have $P(\cup_{i,j}J^{ij}) \leq \ee^{O(\delta)\beta}P(J^{ij})$. In order to estimate the probability that one of the events $J^{ij}$ occurs, we need the following observations:
\begin{itemize}
\item[(i)] 
If the particles $i$ and $j$ are both free in $\bar\Lambda$, then the cluster cannot move. Indeed, an $(m+2)^{th}$ attribution of the red color is not allowed.
\item[(ii)] 
In the time intervals in which the particles $i$ and $j$ are both free in $\bar\Lambda$, they evolve as independent random walks with simultaneous stops.
\end{itemize}
Suppose that the cluster does not move via interactions with the free particles. The dynamics of the $\ell_1'\ell_2'+2$ particles can be seen as the dynamics of two independent simple random walks $\xi=(\xi_t)_{t\geq0}$ and $\xi'=(\xi'_t)_{t\geq0}$ with a trap at the origin: the jump rate is $4\ee^{-U\beta}$ at the origin and $4$ at the other sites, towards a nearest-neighbor site chosen uniformly at random. Thus it suffices to prove that, if at least one particle starts either in the origin or at distance $\ee^{(D+\delta)\beta}$ from the origin, then
\begin{equation}
\label{eq:costoJ}
P(\exists t\leq T_{\Delta^+}\ee^{\delta\beta},\ \xi_t,\xi_t'\in\bar\Lambda\setminus\{0\})
\leq \ee^{-\frac{1}{2}(2U-\Delta-\alpha-O(\delta))\beta}.
\end{equation}
To this end, note that we can associate to $\xi$ a simple random walk $\tilde\xi=(\tilde\xi_t)_{t\geq0}$ during every time interval in which $\xi_t\notin0$. Denoting $s(t)=\max\{s\leq t| \ \xi_s=0\}$ for $t\leq T_{\Delta^+}\ee^{\delta\beta}$, we obtain
\begin{equation}
\label{eq:B2}
\begin{array}{ll}
P_0(\xi_t\in\bar\Lambda\setminus\{0\})
&=\displaystyle\sum_{x\in\bar\Lambda\setminus\{0\}} P_0(\xi_t=x)
=\displaystyle\sum_{x\in\bar\Lambda\setminus\{0\}}\int_0^t P_0(s(t)\in \dd s,\xi_t=x) \\
&\leq\displaystyle\sum_{x\in\bar\Lambda\setminus\{0\}}\int_0^t \dd s\,4\ee^{-U\beta}P_0(\tilde\xi_{t-s}=x)
\leq \displaystyle\sum_{x\in\bar\Lambda\setminus\{0\}}\int_0^t \dd s\,4\ee^{-U\beta}\Big(\dfrac{cst}{1+t-s}\Big) \\
&\leq C|\bar\Lambda|(\log t +1)\ee^{-U\beta}
\leq C|\bar\Lambda|((\Delta+\alpha+\delta)\beta+1)\ee^{-U\beta}.
\end{array}
\end{equation}
Hence, by \eqref{eq:B2},
$$
P_{(0,0)}(\exists\, t\leq T_{\Delta^+}\ee^{\delta\beta},\  
\xi_t,\xi_t'\in\bar\Lambda\setminus\{0\}) \leq \int_0^{T_{\Delta^+}\ee^{\delta\beta}} \dd t\,
P_0(\xi_t\in\bar\Lambda\setminus\{0\})^2
\leq \ee^{(\Delta+\alpha-2U)\beta}\ee^{O(\delta)\beta}.
$$
Suppose that $x\in[\bar\Lambda,D+\delta]\setminus\{0\}$. Letting $\tau$ the first time at which a particle detaches from the origin and $\tau'_0$ the first time at which $\xi'$ reaches the origin, we get
$$
\begin{array}{ll}
P_{(0,x)}(\exists t\leq T_{\Delta}\ee^{\delta\beta},\ \xi_t,\xi_t'\in\bar\Lambda\setminus\{0\})
&\leq\displaystyle\int_0^{T_{\Delta^+}\ee^{\delta\beta}} \dd t\,
P_{(0,0)}(\xi_t,\xi'_t\in\bar\Lambda\setminus\{0\}) \\	    
&\quad+\displaystyle\int_0^{T_{\Delta^+}\ee^{\delta\beta}} \dd t\,
P_{(0,x)}(\xi_t,\xi'_t\in\bar\Lambda\setminus\{0\},\tau'_0>\tau).
\end{array}
$$
To prove \eqref{eq:costoJ}, by the non-superdiffusivity property we can bound the second integral from above as
$$
\displaystyle\int_0^{T_{\Delta^+}\ee^{\delta\beta}} dt
\displaystyle\int_0^t \dd s\,\ee^{-U\beta}\ee^{-s\ee^{-U\beta}}
P_1(\xi_{t-s}\in\bar\Lambda\setminus\{0\})
\displaystyle\sum_{y\in B(x,\sqrt{s}\ee^{\delta\beta})\setminus\{0\}}
\dfrac{cst}{1+s}P_y(\xi'_{t-s}\in\bar\Lambda\setminus\{0\}).
$$
Dividing the integral from $0$ to $t$ into the integral from $0$ to $\ee^{(U-\frac{1}{2}(\epsilon-\alpha))\beta}$ and from $\ee^{(U-\frac{1}{2}(\epsilon-\alpha))\beta}$ to $t$, we obtain the desired lower bound. Indeed, the former integral gives $\ee^{(\Delta+\alpha)\beta}(\ee^{-U\beta}/\ee^{(U-\frac{1}{2}(\epsilon-\alpha))\beta}) = \ee^{-\frac{1}{2}(\epsilon-\alpha)\beta}\ee^{O(\delta)\beta}$ as upper bound,
while the second integral gives $\ee^{-(\epsilon-\alpha)\beta}\ee^{O(\delta)\beta}$ as upper bound arguing as in \eqref{eq:B2}. This concludes the proof of \eqref{eq:costoJ}.
		
It remains to consider the case in which the cluster can move after interacting with the free particles. We observe that if each time the cluster moves we translate it to the origin, then it remains fixed during the whole time interval and there is a resulting perturbation to the remaining free particle. By arguing as before, we get the same result.
		
\medskip\noindent
{\it Event $H_3$.} 
We can argue as for the event $H_2$.
		
\medskip\noindent
{\it Event $H_3'$.} 
Divide the time interval $[0,T_{\Delta^+}\ee^{\delta\beta}]$ into intervals $[t_i,t_i+\ee^{D\beta}]$ of length $\ee^{D\beta}$, with $1 \leq i<\ee^{(\Delta+\alpha-D)\beta}\ee^{\delta\beta}$. If $H_3'$ occurs, then there are two possible situations: either the two different pairs of green particles are $(l,j)$ and $(r,k)$ with $j\neq k$, which we refer to as $H_3'^{,1}$, or $(l,j)$ and $(l,k)$ with $l,j,r,k$ are all different from each other, which we refer to as $H_3'^{,2}$. Using an argument similar to the one used for the event $H_2$, we obtain
$$
\begin{array}{ll}
P(H_3'^{,1}) &\leq \ee^{2D\beta}\ee^{-2(\Delta+\alpha)}\ee^{O(\delta)\beta}, \\
P(H_3'^{,2}) &\leq \ee^{2D\beta}\ee^{-2(\Delta+\alpha)}\ee^{O(\delta)\beta},
\end{array}
$$
which imply that $c(H_3')\geq 2(\Delta-D+\alpha)-O(\delta)$.
		
\medskip\noindent
{\it Event $I$.} 
We can argue as for the event $\tilde{F}_{m+1}$.
		
\medskip\noindent
{\it Event $\tilde{H}_{2}$} 
We have $c(\tilde{H}_2) = \min\{c(H_3)$, $c(H_3'), c(I)\} \geq U - \frac{1}{2} \epsilon - \frac{3}{2}\alpha -d -O(\delta)$.
		
\medskip\noindent
{\it Event $K_{1}$} 
Use time scale $\ee^{(\Delta+\alpha)\beta}$ for the green particles, because of the condition on ${\cal X}_{\Delta^+}$ and the fact that the yellow particle is inside the box $[\bar\Lambda,D+\delta]$. Using the spread-out property for green/red and yellow particles, we obtain
$$
P(K_1\cap G'^c\cap H_2^c)\leq\displaystyle\sum_{t^*\leq i\ee^{D\beta}\leq \ee^{(\Delta+\alpha+\delta)\beta}}
\sum_{j<i}	\Bigg(\dfrac{\ee^{D\beta}\ee^{\delta\beta}}{\ee^{(\Delta+\alpha)\beta}}\Bigg)^3
\Bigg(\dfrac{\ee^{D\beta}\ee^{\delta\beta}}{(i+1)\ee^{D\beta}}\Bigg)\leq \ee^{(2(D-\Delta)-2\alpha)\beta}\ee^{O(\delta)\beta}.
$$
This implies that 
$$
P(K_1) \leq P(K_1\cap G'^c\cap H_2^c) + P(H_2)\ee^{\delta\beta} \leq \ee^{-(\Delta-D+\alpha-O(\delta))\beta}
$$
and therefore $c(K_1)\geq \Delta-D+\alpha-O(\delta)$.
		
\medskip\noindent
{\it Event $K_{2}$.} 
We argue as for the event $K_1$.

\end{appendices}
	

	

\begin{thebibliography}{99}
		
\bibitem{BGdHNOS}
S.\ Baldassarri, A.\ Gaudilli\`ere, F.\ den Hollander, F.R.\ Nardi, E.\ Olivieri and E.\ Scoppola,
Homogeneous nucleation for two-dimensional Kawasaki dynamics, in preparation.

\bibitem{BL15}
J.\ Beltr\'an, C.\ Landim,
Tunneling of the Kawasaki dynamics at low temperatures in two dimensions,
Ann. Inst. H. Poincaré Probab. Statist. 51 (2015) 59--88.

\bibitem{BdH15}
A.\ Bovier and F.\ den Hollander,
\emph{Metastability -- a Potential-Theoretic Approach}, 
Grundlehren der mathematischen Wissenschaften 351, Springer, Berlin, 2015.

\bibitem{BdHS10}
A.\ Bovier, F.\ den Hollander and C.\ Spitoni,
Homogeneous nucleation for Glauber and Kawasaki dynamics in large volumes at low temperatures, 
Ann.\ Probab.\ 38 (2010) 661--713.

\bibitem{CGOV}
M.\ Cassandro, A.\ Galves, E.\ Olivieri and M.E.\ Vares, 
Metastable behaviour of stochastic dynamics: a pathwise approach, 
J.\ Stat.\ Phys. 35 (1984) 603--634.

\bibitem{DS97}
P.\ Dehghanpour and R.\ H.\ Schonmann,
Metropolis dynamics relaxation via nucleation and growth,
Comm. Math. Phys. 188 (1997) 89--119.

\bibitem{GL15}
B.\ Gois, C.\ Landim,
Zero-temperature limit of the Kawasaki dynamics for the Ising lattice gas in a large two-dimensional torus,
Ann. Probab. 43 (2015) 2151--2203.

\bibitem{GMV20}
A.\ Gaudilli\`ere, P.\ Milanesi, M.\ Vares,
Asymptotic Exponential Law for the Transition Time to Equilibrium of the Metastable Kinetic Ising Model with Vanishing Magnetic Field
J.\  Stat.\ Phys. 179 (2020) 263--308.

\bibitem{GOS}
A.\ Gaudilli\`ere, E.\ Olivieri and E.\ Scoppola,
Nucleation pattern at low temperature for local Kawasaki dynamics in two dimensions,
Markov Process.\ Relat.\ Fields 11 (2005) 553--628.

\bibitem{G09}
A.\ Gaudilli\`ere, 
Collision probability for random trajectories in two dimensions, 
Stoch.\ Proc.\ Appl.\ 119 (2009) 775--810.

\bibitem{GdHNOS09}
A.\ Gaudilli\`ere, F.\ den Hollander, F.R.\ Nardi, E.\ Olivieri and E.\ Scoppola,
Ideal gas approximation for a two-dimensional rarified gas under Kawasaki dynamics, 
Stoch.\ Proc.\ Appl.\ 119 (2009) 737--774.

\bibitem{dHOS00a}
F.\ den Hollander, E.\ Olivieri and E.\ Scoppola, 
Metastability and nucleation for conservative dynamics, 
J.\ Math.\ Phys.\ 41 (2000) 1424--1498. 

\bibitem{OS1995}
E.\ Olivieri and E.\ Scoppola, 
Markov chains with exponentially small transition probabilities: First exit problem from a general domain. I. The reversible case, 
J.\ Stat.\ Phys. 79 (1995) 613--647.

\bibitem{OS1996}
E.\ Olivieri and E.\ Scoppola, 
Markov chains with exponentially small transition probabilities: First exit problem from a general domain. II. The general case, 
J.\ Stat.\ Phys. 84 (1996) 987--1041.

\bibitem{OV04}
E.\ Olivieri and M.E.\ Vares, 
\emph{Large Deviations and Metastability},
Cambridge University Press, Cambridge, 2004. 

\bibitem{S1993}
E.\ Scoppola, 
Renormalization group for Markov chains and application to metastability, 
J.\ Stat.\ Phys.\ 73 (1993) 83--121.

\bibitem{SS98}
R.H.\ Schonmann and S.B.\ Shlosman,
Wulff droplets and the metastable relaxing of kinetic Ising models,
Comm.\ Math.\ Phys.\ 194 (1998) 389--462.
		
\end{thebibliography}
\end{document}